\documentclass{amsart}
\usepackage[nobysame]{amsrefs}
\usepackage{amssymb}
\usepackage[usesnames,svgnames]{xcolor}
\usepackage[sc]{mathpazo}
\usepackage[T1]{fontenc}
\usepackage{tikz,pgf}
  \pgfdeclarelayer{background}
  \pgfsetlayers{background,main}
\usepackage[colorlinks=true,
  urlcolor=MidnightBlue,
  citecolor=DarkGreen]{hyperref}
\usepackage{etoolbox}
\usepackage{ifthen}
\usepackage{subfigure}
\usepackage{enumerate}
\usepackage[left=3cm,right=3cm]{geometry}
\usepackage{todonotes}
\usepackage[normalem]{ulem}

\newtheorem{theorem}{Theorem}[section]
\newtheorem{proposition}[theorem]{Proposition}
\newtheorem{conjecture}[theorem]{Conjecture}
\newtheorem{lemma}[theorem]{Lemma}
\newtheorem{corollary}[theorem]{Corollary}
\newtheorem{definition}[theorem]{Definition}

\theoremstyle{remark}
\newtheorem{remark}[theorem]{Remark}
\newtheorem{example}[theorem]{Example}
\newcommand{\alert}[1]{{\color{DarkGreen}\emph{#1}}}
\newcommand{\ie}{\text{i.e.}\;}
\newcommand{\tpt}{(\mathbf{2\!+\!2})}
\newcommand{\ww}{\mathbf{w}}
\newcommand{\xx}{\mathbf{x}}
\newcommand{\JJ}{\mathcal{J}}
\newcommand{\MM}{\mathcal{M}}
\newcommand{\mm}{\text{M}}
\newcommand{\jj}{\text{J}}
\newcommand{\PP}{\mathcal{P}}
\newcommand{\QQ}{\mathcal{Q}}
\newcommand{\LL}{\mathcal{L}}
\newcommand{\Pm}[2]{#1^{\langle #2\rangle}}
\newcommand{\DM}{\emph{\textbf{D\!M}}}
\newcommand{\mdyck}[2]{\ifstrequal{#2}{1}{D_{#1}}{D_{#1}^{(#2)}}}
\newcommand{\mtam}[2]{\ifstrequal{#2}{1}{\mathcal{T}_{#1}}{\mathcal{T}_{#1}^{(#2)}}}
\newcommand{\mcambk}[2]{\ifstrequal{#2}{1}{\mathcal{C}_{#1}}{\Pm{\mathcal{C}_{#1}}{#2}}}
\newcommand{\dmn}[2]{\Pm{D_{#1}}{#2}}
\newcommand{\tmn}[2]{\Pm{\mathcal{T}_{#1}}{#2}}
\newcommand{\cf}{\mathfrak{c}}
\newcommand{\pf}{\mathfrak{p}}
\newcommand{\qf}{\mathfrak{q}}
\newcommand{\rf}{\mathfrak{r}}
\newcommand{\hh}{\mathbf{h}}
\newcommand{\uu}{\mathbf{u}}
\newcommand{\cb}{\mathbf{c}}
\newcommand{\pb}{\mathbf{p}}
\newcommand{\qb}{\mathbf{q}}

\newcommand{\xb}{\mathbf{x}}

\newcommand{\zb}{\mathbf{z}}
\newcommand{\cOne}{orange!80!gray}
\newcommand{\cTwo}{green!50!gray}

\newcommand{\dyckThree}[4]{
	\begin{tikzpicture}[scale=#3]\tiny
		\def\r{.4};
		\ifnumgreater{#4}{0}{\draw[white!50!gray](0,0) grid[step=\r] (3*\r,3*\r);}{}
		\draw(1.5*\r,3.25*\r);
		\foreach \a/\b/\c in {#1}{
			\draw[black](0,0) -- (\a*\r,1*\r) -- (\b*\r,1*\r) -- (\b*\r,2*\r) -- (\c*\r,2*\r) 
			  -- (\c*\r,3*\r) -- (3*\r,3*\r) -- cycle;
			\begin{pgfonlayer}{background}
				\fill[#2] (0,0) -- (\a*\r,1*\r) -- (\b*\r,1*\r) -- (\b*\r,2*\r) -- (\c*\r,2*\r) 
				  -- (\c*\r,3*\r) -- (3*\r,3*\r) -- cycle;
			\end{pgfonlayer}
		
		}
	\end{tikzpicture}
}
\newcommand{\dyckTwoThree}[5]{
	\begin{tikzpicture}[scale=#4]\tiny
		\def\r{.4};
		\fill[#2](0,0) -- (0,1*\r) -- (1*\r,1*\r) -- (1*\r,.5*\r) -- cycle;
		\fill[#3](1*\r,.5*\r) -- (1*\r,1*\r) -- (2*\r,1*\r) --  cycle;
		\fill[#2](2*\r,1*\r) -- (2*\r,2*\r) -- (3*\r,2*\r) -- (3*\r,1.5*\r) -- cycle;
		\fill[#3](3*\r,1.5*\r) -- (3*\r,2*\r) -- (4*\r,2*\r) --  cycle;
		\fill[#2](4*\r,2*\r) -- (4*\r,3*\r) -- (5*\r,3*\r) -- (5*\r,2.5*\r) -- cycle;
		\fill[#3](5*\r,2.5*\r) -- (5*\r,3*\r) -- (6*\r,3*\r) --  cycle;
 		\ifnumgreater{#5}{0}{\draw[white!50!gray](0,0) grid[step=\r] (6*\r,3*\r);}{}
		\foreach \a/\b/\c in {#1}{
			\draw[black](0,0) -- (\a*\r,1*\r) -- (\b*\r,1*\r) -- (\b*\r,2*\r) -- (\c*\r,2*\r) 
			  -- (\c*\r,3*\r) -- (6*\r,3*\r) -- cycle;
			\begin{pgfonlayer}{background}
				\ifthenelse{\b<2*\r}
					{\fill[#3](1*\r,1*\r) -- (1*\r,2*\r) -- (2*\r,2*\r) -- (2*\r,1*\r) -- cycle;}{}
				\ifthenelse{\b<1*\r}
					{\fill[#2](0*\r,1*\r) -- (0*\r,2*\r) -- (1*\r,2*\r) -- (1*\r,1*\r) -- cycle;}{}
				\ifthenelse{\c<4*\r}
					{\fill[#3](3*\r,2*\r) -- (3*\r,3*\r) -- (4*\r,3*\r) -- (4*\r,2*\r) -- cycle;}{}
				\ifthenelse{\c<3*\r}
					{\fill[#2](2*\r,2*\r) -- (2*\r,3*\r) -- (3*\r,3*\r) -- (3*\r,2*\r) -- cycle;}{}
				\ifthenelse{\c<2*\r}
					{\fill[#3](1*\r,2*\r) -- (1*\r,3*\r) -- (2*\r,3*\r) -- (2*\r,2*\r) -- cycle;}{}
				\ifthenelse{\c<1*\r}
					{\fill[#2](0*\r,2*\r) -- (0*\r,3*\r) -- (1*\r,3*\r) -- (1*\r,2*\r) -- cycle;}{}
			 \end{pgfonlayer}{background}
			
		}
	\end{tikzpicture}
}
\newcommand{\dyckFour}[3]{
	\begin{tikzpicture}[scale=#3]\tiny
		\def\r{.2};
		\foreach \a/\b/\c/\d in {#1}{
			\draw[black](0,0) -- (\a*\r,1*\r) -- (\b*\r,1*\r) -- (\b*\r,2*\r) -- (\c*\r,2*\r) 
			  -- (\c*\r,3*\r) -- (\d*\r,3*\r) -- (\d*\r,4*\r) -- (4*\r,4*\r) -- cycle;
			\begin{pgfonlayer}{background}
				\fill[#2](0,0) -- (\a*\r,1*\r) -- (\b*\r,1*\r) -- (\b*\r,2*\r) -- (\c*\r,2*\r) 
				  -- (\c*\r,3*\r) -- (\d*\r,3*\r) -- (\d*\r,4*\r) -- (4*\r,4*\r);
			\end{pgfonlayer}
		}
	\end{tikzpicture}
}

\author{Myrto Kallipoliti}
\address{Fak. f\"ur Mathematik, Universit\"at Wien, Oskar-Morgenstern-Platz 1, 1090 Vienna, Austria}
\email{myrto.kallipoliti@univie.ac.at}
\author{Henri M\"uhle}
\address{LIAFA, Universit{\'e} Paris Diderot, Case 7014, F-75205 Paris Cedex 13, France}
\email{henri.muehle@liafa.univ-paris-diderot.fr}
\thanks{This work was funded by the FWF Research Grant No. Z130-N13.  The second author was also partially supported by a Public Grant overseen by the French National Research Agency (ANR) as part of the ``Investissements d'Avenir'' Program (Reference: ANR-10-LABX-0098).} 
\title{The $m$-Cover Posets and Their Applications}
\allowdisplaybreaks
\begin{document}

\begin{abstract}
	In this article we introduce the $m$-cover poset of an arbitrary bounded poset $\mathcal{P}$, which is a certain subposet of the $m$-fold direct product of $\mathcal{P}$ with itself.  Its ground set consists of multichains of $\mathcal{P}$ that contain at most three different elements, one of which has to be the least element of $\mathcal{P}$, and the other two elements have to form a cover relation in $\mathcal{P}$.  We study the $m$-cover poset from a structural and topological point of view.  In particular, we characterize the posets whose $m$-cover poset is a lattice for all $m>0$, and we characterize the special cases, where these lattices are EL-shellable, left-modular, or trim.  Subsequently, we investigate the $m$-cover poset of the Tamari lattice $\mathcal{T}_{n}$, and we show that the smallest lattice that contains the $m$-cover poset of $\mathcal{T}_{n}$ is isomorphic to the $m$-Tamari lattice $\mathcal{T}_{n}^{(m)}$ introduced by Bergeron and Pr{\'e}ville-Ratelle.  We conclude this article with a conjectural desription of an explicit realization of $\mathcal{T}_{n}^{(m)}$ in terms of $m$-tuples of Dyck paths.
\end{abstract}

\maketitle
\section{Introduction and Results}
  \label{sec:introduction}
Partially ordered sets (posets for short) play an important unifying role in combinatorics, and furthermore they provide a deep connection between combinatorics and other branches of mathematics.  It is often useful to construct new posets from old, and there are several ways for doing this, such as direct product, ordinal product or ordinal sum of given posets.  This paper is dedicated to the study and to the application of a new poset construction, called the \alert{$m$-cover poset} which we introduce for a bounded poset $\PP$ and a positive integer $m$.  The $m$-cover poset of $\PP$, denoted by $\Pm{\PP}{m}$, has as ground set the set of multichains of $\PP$ that have length $m$, and that contain at most three distinct elements of the form $\hat{0},p,q$ such that $\hat{0}$ is the least element of $\PP$ and $p$ is covered by $q$.  The order relation is given ``component-wise'', with respect to the order of $\PP$ (see Section~\ref{sec:mcover} for the precise definition).  

In the first part of this article, we focus on the study of this construction.  More precisely, we give formulas for the cardinality of the $m$-cover poset of an arbitrary bounded poset $\PP$, and we characterize the elements with precisely one lower or one upper cover.  Our structural investigation culminates in the following characterization of the bounded posets whose $m$-cover poset is a lattice for all $m$.

\begin{theorem}\label{thm:mcover_lattice_prop}
	Let $\PP$ be a bounded poset with least element $\hat{0}$ and greatest element $\hat{1}$. The $m$-cover poset $\Pm{\PP}{m}$ is a lattice for all $m>0$ if and only if the Hasse diagram of $\PP$ with $\hat{0}$ removed is a tree rooted at $\hat{1}$.
\end{theorem}

We continue our study of the $m$-cover posets by characterizing the cases where this poset is an EL-shellable, a left-modular or a trim lattice.  Before we state the corresponding results, we need to fix some notation.  Let $\PP_{k,l}$ denote the bounded poset whose proper part is the disjoint union of a $k$-chain and an $l$-antichain, and let $\pf$ be a northeast path, namely a lattice path consisting only of north-steps and east-steps.  Further, let $\PP_{k,l;\pf}$ denote the \alert{path poset} of $\PP_{k,l}$ and $\pf$, namely a certain poset that is constructed from $\PP_{k,l}$ by adding elements according to the steps in $\pf$, see Definition~\ref{def:path_poset}.  Then, we have the following results.

\begin{theorem}\label{thm:mcover_path_poset}
	Let $\PP$ be a bounded poset such that $\Pm{\PP}{m}$ is a lattice for all $m>0$.  The following are equivalent:
	\begin{enumerate}[(a)]
		\item $\PP$ is either a singleton or $\PP\cong\PP_{k,l;\pf}$ for some $k,l\geq 0$ and some northeast path $\pf$;
		\item $\PP$ is $\tpt$-free; 
		\item $\Pm{\PP}{m}$ is left-modular;\quad and
		\item $\Pm{\PP}{m}$ is EL-shellable.
	\end{enumerate}
\end{theorem}

\begin{theorem}\label{thm:mcover_trim}
	Let $\PP$ be a bounded poset such that $\Pm{\PP}{m}$ is a lattice for all $m>0$.  Then, $\Pm{\PP}{m}$ is trim if and only if $\PP\cong\PP_{k,1;\pf}$ for some $k\geq 0$, and some northeast path $\pf$ consisting only of north-steps.
\end{theorem}

\smallskip

In the second part of this article, we apply this construction to a special family of lattices, the \alert{Tamari lattices} $\mtam{n}{1}$.  These lattices were originally defined in \cite{tamari62algebra} as partial orders on the set of binary bracketings of a string of length $n+1$ using $n$ pairs of parentheses.  Their cardinality is given by the $n$-th Catalan number, defined by $\text{Cat}(n)=\tfrac{1}{n+1}\tbinom{2n}{n}$.  The Tamari lattices are a well-studied, important family of lattices, with a huge impact on many, seemingly unrelated fields of mathematics.  Much of this impact comes from the fact that the Hasse diagram of $\mtam{n}{1}$ is isomorphic (as a graph) to the $1$-skeleton of the $(n-1)$-dimensional associahedron~\cite{stasheff63homotopy1}.  The Tamari lattices also form an important family of lattices that enjoys many lattice-theoretic properties.  See for instance \cite{hoissen12associahedra} for a recent survey on the impact of the Tamari lattices.

Bergeron and Pr{\'e}ville-Ratelle introduced a generalization of $\mtam{n}{1}$, the so-called \alert{$m$-Tamari lattice} $\mtam{n}{m}$, whose cardinality is given by the $(m,n)$-th Fu{\ss}-Catalan number, defined by $\text{Cat}^{(m)}(n)=\tfrac{1}{mn+1}\tbinom{(m+1)n}{n}$.  This lattice occurs in the computation of the graded Frobenius characteristic of the spaces of higher diagonal harmonics~\cite{bergeron12higher}.  It follows from \cite{bousquet11number} that $\mtam{n}{m}$ can be embedded as an interval in $\mtam{mn}{1}$, but so far no direct connection between $\mtam{n}{1}$ and $\mtam{n}{m}$ is available\footnote{It might seem odd that \cite{bousquet11number} proves a result on the $m$-Tamari lattices, which were introduced in \cite{bergeron12higher}, a paper that was published a year after \cite{bousquet11number}.  However, a first version of \cite{bergeron12higher} appeared on the arXiv in May 2011, while the first version of \cite{bousquet11number} appeared on the arXiv in June 2011.}.  We use the $m$-cover poset of $\mtam{n}{1}$ to provide such a connection. 

\begin{theorem}\label{thm:mtamari}
	For $m,n>0$, we have $\mtam{n}{m}\cong\DM\bigl(\tmn{n}{m}\bigr)$, where $\DM$ denotes the Dedekind-MacNeille completion.
\end{theorem}

We want to stress the difference in the superscripts in Theorem~\ref{thm:mtamari}.  This theorem states that we can express the $m$-Tamari lattices $\mtam{n}{m}$ as the smallest lattice that contains the $m$-cover poset of $\mtam{n}{1}$ as a subposet.  In order to prove Theorem~\ref{thm:mtamari}, we view the lattice $\mtam{n}{m}$ as a poset on certain lattice paths, the so-called $m$-Dyck paths of length $(m+1)n$, under rotation order.  The main tool in the proof of Theorem~\ref{thm:mtamari} is a certain decomposition of these $m$-Dyck paths into $m$-tuples of classical Dyck paths of length $2n$, which we call the \alert{strip decomposition}, see Definition~\ref{def:strip_decomposition}. 

Finally, we complete this work by conjecturing that there is a more explicit way to realize $\mtam{n}{m}$ as a lattice of $m$-tuples of Dyck paths: if we modify the strip decomposition of the $m$-Dyck paths of length $(m+1)n$ in a certain way---a procedure we call \alert{bouncing}---then we obtain a slightly different set of $m$-tuples of Dyck paths of length $2n$, and computer experiments suggest that this set under componentwise rotation order is isomorphic to $\mtam{n}{m}$.  See Section~\ref{sec:strip_decomposition} for the details.

\smallskip

Before we proceed to the organization of this article, a few comments are in order.  The Tamari lattices can also be viewed as certain sublattices of the weak order on the symmetric group~\cite{bjorner97shellable}*{Theorem~9.6(ii)}.  This connection was the starting point for Reading's definition of the so-called $\gamma$-Cambrian lattices associated with a Coxeter group $W$, and some Coxeter element $\gamma\in W$~\cite{reading06cambrian}.  The cardinality of these lattices is given by the generalized Catalan number of $W$~\cite{reading07clusters}*{Theorem~9.1}, and its Hasse diagram is isomorphic (as a graph) to the $1$-skeleton of the $\gamma$-generalized associahedron associated with $W$~\cite{hohlweg11permutahedra}*{Theorem~3.4}, which beautifully generalizes the analogous properties of the Tamari lattices.  
We recover the Tamari lattices $\mtam{n}{1}$ in this construction by choosing $W$ to be the symmetric group $\mathfrak{S}_{n}$ and $\gamma=(1\;2\;\ldots\;n)$ to be a certain long cycle~\cite{reading07clusters}*{Example~2.3}.

Recently, an ``$m$-eralized'' version of the $\gamma$-Cambrian lattices was introduced by Stump, Thomas and Williams~\cite{stump15cataland}, which is associated with a Coxeter group $W$, a Coxeter element $\gamma\in W$, and some integer $m>0$.  For $m=1$, their construction recovers the $\gamma$-Cambrian lattices from the previous paragraph.  Moreover, the cardinality of the ``$m$-eralized'' $\gamma$-Cambrian lattices is given by the generalized Fu{\ss}-Catalan number of $W$.  Despite the connection between the $\gamma$-Cambrian lattices and the Tamari lattices in the case $m=1$, and the fact that the cardinality of the $m$-Tamari lattices is given by the classical Fu{\ss}-Catalan numbers, the $m$-Tamari lattices do not belong to the framework of ``$m$-eralized'' $\gamma$-Cambrian lattices~\cite{stump15cataland}*{Remark~4.39}.

The Tamari lattice $\mtam{n}{1}$ can also be realized as poset on triangulations of a convex $(n+2)$-gon, where the partial order is given by ``flipping diagonals''.  There is a straightforward generalization of these objects to $(m+2)$-angulations of a convex $(mn+2)$-gon, which in particular yields a combinatorial model of the generalized cluster complex associated with $\mathfrak{S}_{n}$~\cite{fomin05generalized}*{Section~5.1}.  We remark that for $m=1$, the corresponding generalized cluster complex is the dual complex of the $(n-1)$-dimensional associahedron mentioned above.  However, it is not clear how to generalize the process of ``flipping diagonals'' in order to recover the $m$-Tamari lattices on these $(m+2)$-angulations, or if this is possible at all.  At least the obvious constructions by ``sliding'' diagonals clockwise or counterclockwise fail already for $n=3$ and $m=2$.

For other known realizations of the $m$-Tamari lattices, such as partial orders on $m$-Dyck paths, or as partial orders on $(m+1)$-ary trees~\cite{pons13combinatoire}, there is so far no method available to generalize the corresponding ground sets to other Coxeter groups.  The $m$-cover posets introduced in this article might provide a suitable tool for such a generalization.  For instance, if we start with the $\gamma$-Cambrian lattice of the dihedral group $\mathfrak{D}_{k}$, \ie the bounded poset whose proper part is the disjoint union of a $(k-1)$-chain and a singleton, then we observe that its $m$-cover poset is always a lattice and its cardinality coincides with the generalized Catalan number of $\mathfrak{D}_{k}$.  On the other hand, if we try to mimic the construction in Theorem~\ref{thm:mtamari} for any other $\gamma$-Cambrian lattice, then we obtain lattices with too many (\ie more than allowed by the corresponding generalized Catalan number) elements.  Perhaps, the $\gamma$-Cambrian lattices are 
not the correct starting point for a generalization of $\mtam{n}{m}$ along the lines of Theorem~\ref{thm:mtamari}.  It might be worthwhile to investigate the formulas for the graded Frobenius characteristic of the spaces of higher diagonal harmonics associated with other Coxeter groups proposed in \cite{bergeron13multivariate}, and see if these can be rephrased similarly to the symmetric group case using a summation over intervals of certain posets~\cite{bergeron12higher}.  Perhaps these posets might serve as a suitable generalization of $\mtam{n}{m}$.  In any case, the generalization of the $m$-Tamari lattices to all Coxeter groups remains an intriguing problem.

\smallskip

This article is organized as follows.  In Section~\ref{sec:constructions} we introduce and study the $m$-cover posets.  
More precisely, after providing the the necessary notions in Section~\ref{sec:posets}, 
we formally define the $m$-cover poset of an arbitrary bounded poset in Section~\ref{sec:mcover}, 
and subsequently prove Theorem~\ref{thm:mcover_lattice_prop}.  In Section \ref{sec:path_poset} 
we define the path poset associated with a bounded poset and some northeast path.  
We use this construction in Section~\ref{sec:mcover_topology}, where we prove Theorems~\ref{thm:mcover_path_poset} and 
\ref{thm:mcover_trim}.  In Section~\ref{sec:application} we investigate the $m$-cover poset of the Tamari lattice $\mtam{n}{1}$. 
Again we start by recalling the necessary definitions as well as some basic properties of $\mtam{n}{m}$ 
in Section~\ref{sec:mdyck_paths}.  In Section~\ref{sec:proof_tamari} we introduce the strip decomposition of $m$-Dyck paths, 
which is the main tool in the proof of Theorem~\ref{thm:mtamari} in the same section.  We complete this paper by further investigating the strip decomposition, and by stating a conjecture on an explicit realization of $\mtam{n}{m}$ in terms of $m$-tuples of Dyck paths in Section~\ref{sec:strip_decomposition}.

\section{The $m$-Cover Poset}
 \label{sec:constructions}
In this section we define the $m$-cover poset of an arbitrary bounded poset and prove Theorems~\ref{thm:mcover_lattice_prop}-\ref{thm:mcover_trim}.  First we recall the necessary order-theoretic notions which we will use.  For a more detailed introduction to posets and lattices, we refer to \cites{stanley97enumerative,davey02introduction}. 

\subsection{Partially Ordered Sets}
   \label{sec:posets}
Let $\PP=(P,\leq)$ be a finite\footnote{In fact, in this article, we always assume $\PP$ to be finite without mentioning it explicitly.} partially ordered set (\alert{poset} for short).  By abuse of notation, we sometimes write $p\in\PP$ for $p\in P$.  We say that $\PP$ is \alert{bounded} if it has a least and a greatest element, denoted by $\hat{0}$ and $\hat{1}$, respectively.  The \alert{proper part of $\PP$}, denoted by $\overline{\PP}$, is the poset obtained by removing the elements $\hat{0}$ and $\hat{1}$.  We say that $\PP$ is a \alert{lattice} if for every two elements $p,q\in P$ there exists a least upper bound, which is called the \alert{join} of $p$ and $q$ and which is denoted by $p\vee q$, and there exists a greatest lower bound, which is called the \alert{meet} of $p$ and $q$ and which is denoted by $p\wedge q$. 

If $p<q$ and there does not exist an element $z\in P$ with $p<z<q$, then we say that $q$ \alert{covers} $p$, and we denote it by $p\lessdot q$.  In this case, we also say that $p$ is a \alert{lower cover} of $q$ and that $q$ is an \alert{upper cover} of $p$.  A set $C=\{p_{1},p_{2},\ldots,p_{s}\}\subseteq P$ with $p_{1}<p_{2}<\cdots<p_{s}$ is a \alert{chain} of $\PP$.  If $p_{1}\lessdot p_{2}\lessdot\cdots\lessdot p_{s}$, then we say that $C$ is \alert{saturated}, and if $p_{1}=\hat{0}$ and $p_{s}=\hat{1}$, then we say that $C$ is \alert{maximal}.

By abuse of notation, we call an element $p\in P$ \alert{join-irreducible} if it is not minimal and if it has a unique lower cover, denoted by $p_{\star}$.  We write $\jj(\PP)$ for the set of join-irreducible elements of $\PP$.  Similarly, we call $p$ \alert{meet-irreducible} if it is not maximal and if it has a unique upper cover, denoted by $p^{\star}$.  We write $\mm(\PP)$ for the set of meet-irreducible elements of $\PP$.  Further $p$ is an \alert{atom} of $\PP$ if $\hat{0}\lessdot p$, and $p$ is a \alert{coatom} of $\PP$ if $p\lessdot\hat{1}$.

Moreover, we recall that given two posets $\PP=(P,\leq_{P})$ and $\QQ=(Q,\leq_{Q})$, the \alert{union} of $\PP$ and $\QQ$ is the poset $\PP\cup\QQ=(P\cup Q,\leq)$, with $p\leq q$ if and only if $p\leq_{P}q$ or $p\leq_{Q}q$.  The \alert{direct product} of $\PP$ and $\QQ$ is the poset $\PP\times\QQ=(P\times Q,\leq)$, with $(p_{1},q_{1})\leq(p_{2},q_{2})$ if and only if $p_{1}\leq_{P}p_{2}$ and $q_{1}\leq_{Q}q_{2}$.

Let $\mathcal{E}(\PP)=\bigl\{(p,q)\mid p\lessdot q\bigr\}$ denote the set of edges of the Hasse diagram of $\PP$.  Given some other poset $(\Lambda,\leq_{\Lambda})$, a map $\lambda:\mathcal{E}(\PP)\to\Lambda$ is called an \alert{edge-labeling} of $\PP$.  A saturated chain of $\PP$ is called \alert{rising} with respect to $\lambda$ if the sequence of edge labels of this chain is strictly increasing with respect to $\leq_{\Lambda}$.  An edge-labeling is an \alert{EL-labeling} if in every interval of $\PP$ there exists a unique rising saturated maximal chain, and this chain is lexicographically first among all saturated maximal chains in this interval.  A bounded poset $\PP$ is \alert{EL-shellable} if it admits an EL-labeling. 

The \alert{length} of $\PP$, denoted by $\ell(\PP)$, is the maximal length of a saturated chain from $\hat{0}$ to $\hat{1}$.  If $\lvert\mathcal{J}(\PP)\rvert=\ell(\PP)=\lvert\mathcal{M}(\PP)\rvert$, then $\PP$ is \alert{extremal}~\cite{markowsky92primes}.  If $\PP$ is a lattice, then $p$ is \alert{left-modular} if for every $q<q'$ we have
\begin{equation}\label{eq:left_modularity}
	(q\vee p)\wedge q' = q\vee (p\wedge q'). 
\end{equation}
If there exists a saturated maximal chain consisting of left-modular elements, then $\PP$ is \alert{left-modular}.  Moreover, $\PP$ is \alert{trim} if it is extremal and left-modular~\cite{thomas06analogue}.

Finally, recall that a poset $\PP=(P,\leq)$ is called \alert{$\tpt$-free} if it does not contain elements $x,y,x',y'\in P$ with $x<y$ and $x'<y'$, as well as $x,y\not\leq x',y'$ and $x',y'\not\leq x,y$~\cite{fishburn70intransitive}. 

\subsection{The Construction of the $m$-Cover Poset}
  \label{sec:mcover}
Let $\PP=(P,\leq)$ be a bounded poset, let $m>0$, and consider $m$-tuples of the form 
\begin{align}
  \label{eq:long_tuple}\pb=(\underbrace{\hat{0},\hat{0},\ldots,\hat{0}}_{l_{0}},\underbrace{p_{1},p_{1},\ldots,p_{1}}_{l_{1}},
	  \underbrace{p_{2},p_{2},\ldots,p_{2}}_{l_{2}}),
\end{align}
for non-negative integers $l_{i}$ with $l_{0}+l_{1}+l_{2}=m$, and $p_{1},p_{2}\in P\setminus\{\hat{0}\}$, with $p_{1}\neq p_{2}$.  We will usually abbreviate \eqref{eq:long_tuple} by $\pb=\bigl(\hat{0}^{l_{0}},p_{1}^{l_{1}},p_{2}^{l_{2}}\bigr)$.

\begin{definition}\label{def:mcover}
	Let $\PP=(P,\leq)$ be a bounded poset and let $m>0$.  Consider the set 
  	\begin{align}
		\Pm{P}{m}=\bigl\{\bigl(\hat{0}^{l_{0}},p_{1}^{l_{1}},p_{2}^{l_{2}}\bigr)\mid p_{1}\lessdot p_{2},l_{0}+l_{1}+l_{2}=m\bigr\}.
  	\end{align}
 	The poset $\Pm{\PP}{m}=(\Pm{P}{m},\leq)$, considered as a subposet of the $m$-fold direct product of $\PP$ with itself\footnote{By abuse of notation we use the same symbol for the partial orders of $\PP$ and $\Pm{\PP}{m}$.}, is called the \alert{$m$-cover poset} of $\PP$.
\end{definition}

It is immediate from the definition that for every $m>0$ the poset $\Pm{\PP}{m}$ is an interval of $\Pm{\PP}{m+1}$.  See Figure~\ref{fig:mcover_chain} for an example.  Further examples can be found in Figures~\ref{fig:no_lattice}, \ref{fig:no_join_sub}, and \ref{fig:cover_examples}.  The length and the cardinality of $\Pm{\PP}{m}$ are determined in the following proposition. 

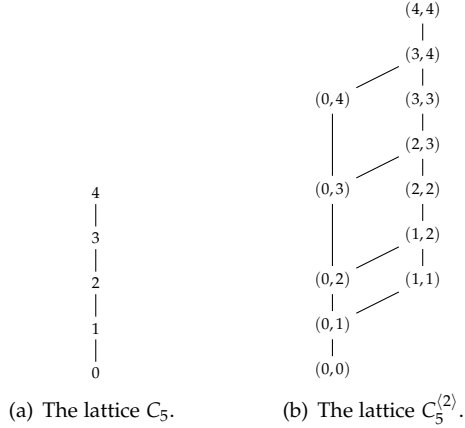
\begin{figure}
	\centering
	\subfigure[The lattice $C_{5}$.]{\label{fig:mcover_chain_1}
		\begin{tikzpicture}\tiny
			\def\x{.6};
			\def\y{.6};
			\draw(.25*\x,3*\y) node{};
			\draw(3.75*\x,3*\y) node{};
			\draw(2*\x,1*\y) node(n1){$0$};
			\draw(2*\x,2*\y) node(n2){$1$};
			\draw(2*\x,3*\y) node(n3){$2$};
			\draw(2*\x,4*\y) node(n4){$3$};
			\draw(2*\x,5*\y) node(n5){$4$};
			\draw(n1) -- (n2) -- (n3) -- (n4) -- (n5);
		\end{tikzpicture}}\hspace*{1cm}
	\subfigure[The lattice $\Pm{C_{5}}{2}$.]{\label{fig:mcover_chain_2}
		\begin{tikzpicture}\tiny
			\def\x{.6};
			\def\y{.6};
			\draw(-.25*\x,3*\y) node{};
			\draw(4.25*\x,3*\y) node{};
			\draw(1*\x,1*\y) node(n1){$(0,0)$};
			\draw(1*\x,2*\y) node(n2){$(0,1)$};
			\draw(1*\x,3*\y) node(n3){$(0,2)$};
			\draw(3*\x,3*\y) node(n4){$(1,1)$};
			\draw(3*\x,4*\y) node(n5){$(1,2)$};
			\draw(1*\x,5*\y) node(n6){$(0,3)$};
			\draw(3*\x,5*\y) node(n7){$(2,2)$};
			\draw(3*\x,6*\y) node(n8){$(2,3)$};			
			\draw(1*\x,7*\y) node(n9){$(0,4)$};
			\draw(3*\x,7*\y) node(n10){$(3,3)$};
			\draw(3*\x,8*\y) node(n11){$(3,4)$};
			\draw(3*\x,9*\y) node(n12){$(4,4)$};
			\draw(n1) -- (n2);
			\draw(n2) -- (n3);
			\draw(n2) -- (n4);
			\draw(n3) -- (n5);
			\draw(n3) -- (n6);
			\draw(n4) -- (n5);
			\draw(n5) -- (n7);
			\draw(n6) -- (n8);
			\draw(n6) -- (n9);
			\draw(n7) -- (n8);
			\draw(n8) -- (n10);
			\draw(n9) -- (n11);
			\draw(n10) -- (n11);
			\draw(n11) -- (n12);
		\end{tikzpicture}}
	\caption{A $5$-chain, and its $2$-cover poset.}
	\label{fig:mcover_chain}
\end{figure}

\begin{proposition}\label{prop:mcover_cardinality}
	Let $\PP=(P,\leq)$ be a bounded poset with $n$ elements, $c$ covering relations and $k$ atoms.  For $m>0$, we have $\ell(\Pm{\PP}{m})=m\cdot\ell(\PP)$ and  
	\begin{displaymath}
		\Bigl\lvert \Pm{P}{m}\Bigr\rvert=(c-k)\cdot\binom{m}{2}+m(n-1)+1.
	\end{displaymath}
\end{proposition}
\begin{proof}
	Suppose that $\PP$ is a bounded poset with $\ell(\PP)=s$, and let $\hat{0}=p_{0}\lessdot p_{1}\lessdot\cdots\lessdot p_{s}=\hat{1}$ be a maximal chain of $\PP$.  Define $\pb_{0,m}=\bigl(\hat{0}^{m}\bigr)$, as well as $\pb_{i,j}=\bigl(p_{i-1}^{m-j},p_{i}^{j}\bigr)$ for $i\in\{1,2,\ldots,s\}$ and $j\in\{1,2,\ldots,m\}$.  It is immediately clear that $\pb_{i,j}\lessdot\pb_{i,j+1}$ for $i\in\{1,2,\ldots,s\}$ and $j\in\{1,2,\ldots,m-1\}$, as well as $\pb_{i-1,m}\lessdot\pb_{i,1}$ for all $i\in\{1,2,\ldots,s\}$.  Thus the chain
	\begin{equation}\label{eq:mcover_maximal_chain}
		\pb_{0,m}\lessdot\pb_{1,1}\lessdot\pb_{1,2}\lessdot\pb_{1,m}\lessdot\pb_{2,1}\lessdot\pb_{2,2}\lessdot\cdots\lessdot\pb_{s,m}
	\end{equation}
	is a maximal chain in $\Pm{\PP}{m}$ with length $ms$, which implies $\ell\bigl(\Pm{\PP}{m}\bigr)\geq ms$.  Since $\Pm{\PP}{m}$ is a subposet of the $m$-fold direct product of $\PP$ with itself, it follows that $\ell\bigl(\Pm{\PP}{m}\bigr)\leq\ell\bigl(\PP^{m}\bigr)=ms$, which implies the claim $\ell(\Pm{\PP}{m})=m\cdot\ell(\PP)$.
	
	\smallskip
	
	Now we want to compute the cardinality of $\Pm{P}{m}$.  If $\pb\in\Pm{P}{m}$, then it necessarily has to be of one of the following four forms:
	
	(i) $\pb=\bigl(\hat{0}^{l_{0}},p^{l_{1}},q^{l_{2}}\bigr)$ with $l_{0},l_{1},l_{2}\neq 0$ and $\hat{0}\neq p\lessdot q$.  Clearly, there are $c-k$ possible choices for $p$ and $q$, and each such choice yields $\tbinom{m-1}{2}$ distinct elements of $\Pm{P}{m}$.
	
	(ii) $\pb=\bigl(p^{l},q^{m-l}\bigr)$ with $l\in\{1,2,\ldots,m-1\}$ and $\hat{0}\neq p\lessdot q$.  Again, there are $c-k$ possible choices for $p$ and $q$, and each such choice yields $m-1$ distinct elements of $\Pm{P}{m}$.
	
	(iii) $\pb=\bigl(\hat{0}^{l},p^{m-l}\bigr)$ with $l\in\{1,2,\ldots,m-1\}$ and $\hat{0}\neq p$.  There are $(m-1)(n-1)$ distinct elements of this form in $\Pm{P}{m}$. 
	
	(iv) $\pb=\bigl(p^{m}\bigr)$ with $p\in P$.  There are $n$ distinct elements of this form in $\Pm{P}{m}$.
	
	If we add up all these possibilities, then we obtain
	\begin{align*}
		\Bigl\lvert\Pm{P}{m}\Bigr\rvert & = (c-k)\binom{m-1}{2} + (c-k)(m-1) + (m-1)(n-1) + n\\
		& = (c-k)\binom{m}{2} + m(n-1) + 1,
	\end{align*}
	as desired.
\end{proof}

The join- and meet-irreducible elements of $\PP$ are related to the join- and meet-irreducible elements of $\Pm{\PP}{m}$ in the following way. 

\begin{proposition}\label{prop:mcover_irreducibles}
	Let $\PP$ be a bounded poset, and let $m>0$. Then,
	\begin{align*}
		\jj\bigl(\Pm{\PP}{m}\bigr) & = \Bigl\{\bigl(\hat{0}^{l},p^{m-l}\bigr)\mid p\in\jj(\PP)\;\text{and}\;0\leq l<m\Bigr\},
			\quad\text{and}\\
		\mm\bigl(\Pm{\PP}{m}\bigr) & = \Bigl\{\bigl(p^{l},(p^{\star})^{m-l}\bigr)\mid p\in\mm(\PP)\setminus\{\hat{0}\}\;\text{and}\;
			0<l\leq m\Bigr\}\\
			& \kern1cm \cup \Bigl\{\bigl(\hat{0}^{l},\hat{1}^{m-l}\bigr)\mid \hat{1}\in\jj(\PP)\;\text{and}\;0<l\leq m\Bigr\}
			\cup \Bigl\{\bigl(\hat{0}^{m}\bigr)\mid \hat{0}\in\mm(\PP)\Bigr\}.
	\end{align*}
\end{proposition}
\begin{proof}
	Let $\pb=\bigl(\hat{0}^{l_{0}},p_{1}^{l_{1}},p_{2}^{l_{2}}\bigr)\in\Pm{P}{m}$ with $\hat{0}\neq p_{1}\lessdot p_{2}$.  First suppose that $\pb\in\jj\bigl(\Pm{\PP}{m}\bigr)$.  If $l_{1}>0$ and $l_{2}>0$, then it follows that the elements $\pb'=\bigl(\hat{0}^{l_{0}+1},p_{1}^{l_{1}-1},p_{2}^{l_{2}}\bigr)$ and $\pb''=\bigl(\hat{0}^{l_{0}},p_{1}^{l_{1}+1},p_{2}^{l_{2}-1}\bigr)$ are both lower covers of $\pb$ in $\Pm{\PP}{m}$, which contradicts the assumption that $\pb$ is join-irreducible.  If $l_{1}=0$ and $l_{2}=0$, then $\pb$ is the least element of $\Pm{\PP}{m}$ and thus not join-irreducible by definition.  Hence, without loss of generality, we can assume that $\pb=\bigl(\hat{0}^{l_{0}},p_{1}^{l_{1}}\bigr)$.  For every $\bar{p}\in P$ with $\bar{p}\lessdot p_{1}$, the element $\bar{\pb}=\bigl(\hat{0}^{l_{0}},\bar{p},p_{1}^{l_{1}-1}\bigr)$ is the only lower cover of $\pb$, which implies the claim. 
	
	\smallskip
	
	Now suppose that $\pb\in\mm\bigl(\Pm{\PP}{m}\bigr)$.  If $l_{0}>0$ and $l_{1}>0$, then it follows that the elements $\pb'=\bigl(\hat{0}^{l_{0}-1},p_{1}^{l_{1}+1},p_{2}^{l_{2}}\bigr)$ and $\pb''=\bigl(\hat{0}^{l_{0}},p_{1}^{l_{1}-1},p_{2}^{l_{2}+1}\bigr)$ are both upper covers of $\pb$ in $\Pm{\PP}{m}$, which contradicts the assumption that $\pb$ is meet-irreducible.  The same reasoning holds if $l_{0}>0,l_{1}=0$ and $l_{2}>0$.  Hence we have either $l_{0}=0$ or at least one of $l_{1}$ and $l_{2}$ is zero. 
	
	First let $l_{0}=0$, and thus $\pb=\bigl(p_{1}^{l_{1}},p_{2}^{l_{2}}\bigr)$.  Clearly the element $\pb''=\bigl(p_{1}^{l_{1}-1},p_{2}^{l_{2}+1}\bigr)$ satisfies $\pb\lessdot\pb''$.  Assume that $p_{1}\notin\mm(\PP)$. It follows that $p_{2}\neq\hat{1}$, because otherwise $p_{1}$ is a coatom, and thus clearly meet-irreducible.  Thus we can choose an upper cover $q_{2}$ of $p_{2}$ in $\PP$, and some upper cover $q_{1}$ of $p_{1}$ in $\PP$ with $q_{1}\neq p_{2}$.  It follows that $q_{1}\neq\hat{1}$, and we distinguish three cases:
	
	(i) If $q_{1}\leq q_{2}$, then there exists a chain $q_{1}=w_{1}\lessdot w_{2}\lessdot\cdots\lessdot w_{k}\lessdot q_{2}$ in $\PP$, and we can choose this chain in such a way that $p_{2}\not\leq w_{k}$, because otherwise we would obtain a contradiction to $p_{1}\lessdot p_{2}\lessdot q_{2}$.  Consider the element $\bar{\pb}=\bigl(w_{k}^{l_{1}},q_{2}^{l_{2}}\bigr)$, which satisfies $\pb\leq\bar{\pb}$.  Suppose that there is some element $\qb\in\Pm{P}{m}$ with $\pb\lessdot\qb\leq\bar{\pb}$.  Since $p_{2}\not\leq w_{k}$, it follows that $\qb=\bigl(x^{l_{1}},y^{l_{2}}\bigr)$ for $p_{1}\leq x\leq w_{k}$ and $p_{2}\leq y\leq q_{2}$.  If $y=p_{2}$, then necessarily $x=p_{1}$, and we obtain $\qb=\pb$, which contradicts the choice of $\qb$.  Hence $y=q_{2}$, and since $x$ is a lower cover of $y$, it follows that $x=w_{k}$, which implies $\qb=\bar{\pb}$.  Hence $\pb\lessdot\bar{\pb}$.  However, since $\bar{\pb}\neq\pb''$, we obtain a contradiction to $\pb$ being meet-irreducible in $\Pm{\PP}{m}$.
	
	(ii) If $q_{2}\leq q_{1}$, then the reasoning is analogous to (i).
	
	(iii) If $q_{1}\not\leq q_{2}$ and $q_{2}\not\leq q_{1}$, then---since $\PP$ is bounded---there exists a (not necessarily unique) minimal element $w\in P$ with $q_{1},q_{2}\leq w$, and there exist chains $q_{1}=u_{1}\lessdot u_{2}\lessdot \cdots\lessdot u_{k}\lessdot w$ and $q_{2}=v_{1}\lessdot v_{2}\lessdot\cdots\lessdot v_{l}\lessdot w$.  Consider the element $\bar{\pb}=\bigl(u_{k}^{l_{1}},w^{l_{2}}\bigr)$, which satisfies $\pb\leq\bar{\pb}$.  Again, suppose that there is some element $\qb\in\Pm{P}{m}$ with $\pb\lessdot\qb\leq\bar{\pb}$.  The minimality of $w$ ensures that $p_{2}\not\leq u_{k}$, and it follows that $\qb=\bigl(x^{l_{1}},y^{l_{2}}\bigr)$ for $p_{1}\leq x\leq u_{k}$ and $p_{2}\leq y\leq w$.  The minimality of $w$ also ensures that $u_{i}\not\leq v_{j}$ and $v_{j}\not\leq u_{i}$ for all $i\in\{1,2,\ldots,k\}$ and $j\in\{1,2,\ldots,l\}$.  Since $x$ is a lower cover of $y$ and $\qb\neq\pb$, it follows that $x=u_{k}$ and $y=w$, which implies $\qb=\bar{\pb}$.  Thus $\pb\lessdot\bar{\pb}$. However,
 since $\bar{\pb}\neq\pb''$ we obtain a contradiction to $\pb$ being meet-irreducible in $\Pm{\PP}{m}$.
	
	Hence if $l_{0}=0$, then it follows that $p_{1}\in\mm(\PP)$ and $p_{2}=p_{1}^{\star}$.
	
	Now suppose that $l_{0}>0$, and thus that $l_{1}=0$ or $l_{2}=0$.  Without loss of generality, we can write $\pb=\bigl(\hat{0}^{l_{0}},p_{1}^{l_{1}}\bigr)$.  If $l_{1}=0$, then every atom of $\PP$ yields an upper cover of $\pb$, and hence $\pb\in\mm\bigl(\Pm{\PP}{m}\bigr)$ if and only if $\hat{0}\in\mm(\PP)$.  Now, let $l_{1}>0$.  If $p_{1}\neq\hat{1}$, then for every upper cover $q$ of $p_{1}$ the element $\pb'=\bigl(\hat{0}^{l_{0}},p_{1}^{l_{1}-1},q\bigr)$ satisfies $\pb\lessdot\pb'$.  Moreover, if $p_{1}$ is an atom, then $\bigl(\hat{0}^{l_{0}-1},p_{1}^{l_{1}+1}\bigr)$ is an upper cover of $\pb$, which is different from $\pb'$.  If $p_{1}$ is no atom, then for every lower cover $q$ of $p_{1}$, the element $\bigl(\hat{0}^{l_{0}-1},q,p_{1}^{l_{1}}\bigr)$ is an upper cover of $\pb$, which is different from $\pb'$.  This contradicts the assumption that $\pb$ is meet-irreducible.  If $p_{1}=\hat{1}$, then for every coatom $c$ of $\PP$ the element $\bigl(\hat{0}^{l_{0}-1},c,\hat{1}^{l_{1}}\bigr)$ is an upper 
cover of $\pb$.  Hence if $l_{0}>0$, then it follows that either $\pb=\bigl(\hat{0}^{m}\bigr)$ or $\pb=\bigl(\hat{0}^{l_{0}},\hat{1}^{l_{1}}\bigr)$ provided that $\hat{1}\in\jj(\PP)$.
\end{proof}

In view of Propositions~\ref{prop:mcover_cardinality} and \ref{prop:mcover_irreducibles}, we can determine the posets for which every $m$-cover poset is extremal, \ie where $\bigl\lvert\jj\bigl(\Pm{\PP}{m}\bigr)\bigr\rvert=\ell\bigl(\Pm{\PP}{m}\bigr)=\bigl\lvert\mm\bigl(\Pm{\PP}{m}\bigr)\bigr\rvert$.

\begin{corollary}\label{cor:mcover_extremal}
	Let $\PP$ be a bounded extremal poset, with $\ell(\PP)=k$. Then, $\Pm{\PP}{m}$ is extremal for every $m>0$ if and only if either $\hat{0}\in\jj(\PP)$ and $\hat{1}\in\mm(\PP)$ or $\hat{0}\notin\jj(\PP)$ and $\hat{1}\notin\mm(\PP)$. 
\end{corollary}
\begin{proof}
	Proposition~\ref{prop:mcover_cardinality} implies that $\ell\bigl(\Pm{\PP}{m}\bigr)=mk$, and it follows from the first part of Proposition~\ref{prop:mcover_irreducibles} that $\bigl\lvert\jj\bigl(\Pm{\PP}{m}\bigr)\bigr\rvert=m\bigl\lvert\jj(\PP)\bigr\rvert=mk$.  Thus it remains to determine the cardinality of the set of meet-irreducibles of $\Pm{\PP}{m}$.  If $\hat{0}\in\mm(\PP)$ and $\hat{1}\notin\jj(\PP)$, then the second part of Proposition~\ref{prop:mcover_irreducibles} implies $\bigl\lvert\mm\bigl(\Pm{\PP}{m}\bigr)\bigr\rvert=m(k-1)+1<mk$ unless $m=1$.  Analogously, if $\hat{0}\notin\mm(\PP)$ and $\hat{1}\in\jj(\PP)$, then the second part of Proposition~\ref{prop:mcover_irreducibles} implies $\bigl\lvert\mm\bigl(\Pm{\PP}{m}\bigr)\bigr\rvert=(m+1)k>mk$.  On the other hand if $\hat{0}\in\mm(\PP)$ and $\hat{1}\in\jj(\PP)$ or $\hat{0}\notin\mm(\PP)$ and $\hat{1}\notin\jj(\PP)$, then the second part of Proposition~\ref{prop:mcover_irreducibles} implies $\bigl\lvert\mm\bigl(\Pm{\PP}{m}\bigr)\bigr\rvert=mk$ 
as desired.
\end{proof}

Now we characterize the cases, where $\Pm{\PP}{m}$ is a lattice.

\begin{theorem}\label{thm:mcover_lattice}
	Let $\PP=(P,\leq)$ be a bounded poset.  The $m$-cover poset $\Pm{\PP}{m}$ is a lattice for all $m>0$ if and only if $\PP$ is a lattice and for all $p,q\in P$ we have $p\wedge q\in\{\hat{0},p,q\}$. 
\end{theorem}
\begin{proof}
	Suppose that $\PP$ is a lattice, and suppose that for every $p,q\in P$, we have $p\wedge q\in\{\hat{0},p,q\}$.  We want to show first that $\Pm{\PP}{m}$ is a lattice again.  Let $\pb=\bigl(\hat{0}^{k_{0}},p_{1}^{k_{1}},p_{2}^{k_{2}}\bigr)$ and $\qb=\bigl(\hat{0}^{l_{0}},q_{1}^{l_{1}},q_{2}^{l_{2}}\bigr)$.  We show that the componentwise meet of $\pb$ and $\qb$, denoted by $\zb$, is again contained in $\Pm{P}{m}$, and since $\Pm{\PP}{m}$ is a subposet of $\PP^{m}$ it follows that $\zb$ has to be the meet of $\pb$ and $\qb$ in $\Pm{\PP}{m}$.  We essentially have two choices for $\zb$, depending on the values of $k_{0},k_{1},k_{2}$ and $l_{0},l_{1},l_{2}$:
	\begin{align}
		\zb & = \bigl(\hat{0}^{s_{0}},(p_{1}\wedge q_{1})^{s_{1}},(p_{1}\wedge q_{2})^{s_{2}},(p_{2}\wedge q_{2})^{s_{3}}\bigr),\label{eq:meet_1}\quad\text{or}\\
		\zb & = \bigl(\hat{0}^{s_{0}},(p_{1}\wedge q_{1})^{s_{1}},(p_{2}\wedge q_{1})^{s_{2}},(p_{2}\wedge q_{2})^{s_{3}}\bigr)\label{eq:meet_2},
	\end{align}
	for suitable $s_{0},s_{1},s_{2},s_{3}\in\{0,1,\ldots,m\}$, and we distinguish three cases. 
	
	(i) Let $p_{1}\wedge q_{1}=\hat{0}$.  Here we need to distinguish three more cases:\\
	(ia) Let $p_{1}\wedge q_{2}=\hat{0}$.  If $\zb$ is of the form \eqref{eq:meet_1}, then it follows immediately that $\zb\in\Pm{P}{m}$.  So, suppose that $\zb$ is of the form \eqref{eq:meet_2}.  If $q_{1}\leq p_{2}$, then $q_{1}\leq p_{2}\wedge q_{2}\leq q_{2}$, which implies with $q_{1}\lessdot q_{2}$ that $\zb\in\Pm{P}{m}$.  If $p_{2}\leq q_{1}$, then it follows immediately that $\zb\in\Pm{P}{m}$.  If $q_{1}$ and $p_{2}$ are incomparable, then $p_{2}\wedge q_{1}=\hat{0}$ by assumption and again $\zb\in\Pm{P}{m}$.\\
	(ib) Let $p_{1}\wedge q_{2}=p_{1}$.  Then, both $p_{2}$ and $q_{2}$ are upper bounds for $p_{1}$, and hence $p_{1}\leq p_{2}\wedge q_{2}$.  If $p_{2}$ and $q_{2}$ are incomparable, then $p_{1}=\hat{0}=q_{1}$, and it follows that $\zb=\bigl(\hat{0}^{m}\bigr)\in\Pm{P}{m}$.  If $q_{2}\leq p_{2}$, then $p_{2}=q_{2}$, and it follows that $\zb=\bigl(\hat{0}^{s_{0}+s_{1}},p_{1}^{s_{2}},p_{2}^{s_{3}}\bigr)\in\Pm{P}{m}$, or $\zb=\bigl(\hat{0}^{s_{0}+s_{1}},q_{1}^{s_{2}},q_{2}^{s_{3}}\bigr)\in\Pm{P}{m}$.  If $p_{2}<q_{2}$, then it follows that $p_{2}$ and $q_{1}$ are incomparable.  Hence $\zb=\bigl(\hat{0}^{s_{0}+s_{1}},p_{1}^{s_{2}},p_{2}^{s_{3}}\bigr)\in\Pm{P}{m}$, or $\zb=\bigl(\hat{0}^{s_{0}+s_{1}+s_{2}},p_{2}^{s_{3}}\bigr)\in\Pm{P}{m}$.\\
	(ic) Let $p_{1}\wedge q_{2}=q_{2}$.  This works analogously to (ib).
	
	(ii) Let $p_{1}\wedge q_{1}=p_{1}$.  Then, it follows by assumption that either $p_{2}\leq q_{1}$ or $p_{1}=\hat{0}$.  In the first case, we have $\zb=\bigl(\hat{0}^{s_{0}},p_{1}^{s_{1}+s_{2}},p_{2}^{s_{3}}\bigr)\in\Pm{P}{m}$ if $\zb$ is of the form \eqref{eq:meet_1}, or $\zb=\bigl(\hat{0}^{s_{0}},p_{1}^{s_{1}},p_{2}^{s_{2}+s_{3}}\bigr)\in\Pm{P}{m}$ if $\zb$ is of the form \eqref{eq:meet_2}.  In the second case, we have $\zb=\bigl(\hat{0}^{s_{0}+s_{1}+s_{2}},(p_{2}\wedge q_{2})^{s_{3}}\bigr)\in\Pm{P}{m}$ if $\zb$ is of the form \eqref{eq:meet_1}.  Thus it remains to consider the case where $p_{1}=\hat{0}$, and $\zb$ is of the form \eqref{eq:meet_2}.  Then, we have either $p_{2}\wedge q_{1}=\hat{0}$ (which implies $\zb=\bigl(\hat{0}^{m}\bigr)\in\Pm{P}{m}$), or $p_{2}\wedge q_{1}=p_{2}$ (which implies $\zb=\bigl(\hat{0}^{s_{0}+s_{1}},p_{2}^{s_{2}+s_{3}}\bigr)\in\Pm{P}{m}$), or $p_{2}\wedge q_{1}=q_{1}$ (which implies $\zb=\bigl(\hat{0}^{s_{0}+s_{1}},q_{1}^{s_{2}},q_{2}^{s_{3}}\bigr)\in\Pm{P}{m}$). 
	
	(iii) Let $p_{1}\wedge q_{1}=q_{1}$. This works analogously to (ii).
	
	Hence every two elements $\pb,\qb\in\Pm{P}{m}$ have a meet in $\Pm{\PP}{m}$, and since $\Pm{\PP}{m}$ is finite and bounded, it is a classical lattice-theoretic result that $\Pm{\PP}{m}$ is a lattice.
	
	\smallskip
	
	We prove the converse argument by contradiction.  Since $\PP$ is an interval in $\Pm{\PP}{m}$, it follows immediately that if $\PP$ is no lattice, then $\Pm{\PP}{m}$ cannot be a lattice as well.  So suppose that $\PP$ is a lattice, and suppose further that there exist two elements $p,q\in P$, with $p\wedge q=z\notin\{\hat{0},p,q\}$, and we choose these elements to be maximal (\ie if $\bar{p},\bar{q}\in P$ with $p\leq\bar{p}$ and $q\leq\bar{q}$ satisfy $\bar{p}\wedge\bar{q}\notin\{\hat{0},\bar{p},\bar{q}\}$, then $p=\bar{p}$ or $q=\bar{q}$).  We explicitly construct two elements $\pb,\qb\in\Pm{P}{m}$ that do not have a meet in $\Pm{\PP}{m}$.  By assumption, neither $p=\hat{1}$ nor $q=\hat{1}$.  Hence we can find elements $p',q'\in P$ with $p\lessdot p'$ and $q\lessdot q'$, and by the maximality of $p$ and $q$ we have $p'\wedge q'\in\{\hat{0},p',q'\}$.  Moreover, since $\hat{0}\neq z\leq p'\wedge q'$, it follows that $p'\leq q'$ or $q'\leq p'$, and we assume without loss of generality that $p'\leq q'$. 
	
	On the one hand, consider the elements $\pb=\bigl(p,(p')^{m-1}\bigr)$ and $\qb=(q,(q')^{m-1}\bigr)$, and on the other hand, consider the elements $\ww_{1}=\bigl(\hat{0},(p')^{m-1}\bigr)$ and $\ww_{2}=\bigl(z,(z')^{m-1}\bigr)$, where $z'$ satisfies $z\lessdot z'\leq p$.  Then, we have $\ww_{1},\ww_{2}\leq\pb,\qb$, and both $\pb$ and $\qb$ as well as $\ww_{1}$ and $\ww_{2}$ are mutually incomparable.  The only candidate for an element that would be larger than $\ww_{1}$ and $\ww_{2}$ and at the same time smaller than $\pb$ and $\qb$ is $\bigl(z,(p')^{m-1}\bigr)$, which does, however, not belong to $\Pm{P}{m}$, since $z<p<p'$ and $z\neq\hat{0}$.  Hence $\pb$ and $\qb$ do not have a meet in $\Pm{\PP}{m}$, which implies that $\Pm{\PP}{m}$ is not a lattice.
\end{proof}

\begin{example}
	Theorem~\ref{thm:mcover_lattice} is illustrated by the examples depicted in Figure~\ref{fig:no_lattice}.  In the poset shown in Figure~\ref{fig:no_lattice_1}, the meet of the elements $2$ and $3$ is not the least element, and in its $2$-cover poset, which is shown in Figure~\ref{fig:no_lattice_2}, the elements $(2,4)$ and $(3,4)$ have three mutually incomparable lower bounds, namely $(1,2)$, $(0,4)$ and $(1,3)$.
\end{example}

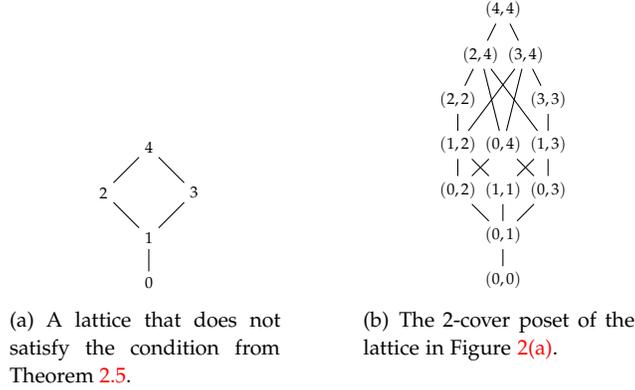
\begin{figure}
	\centering
	\subfigure[A lattice that does not satisfy the condition from Theorem~\ref{thm:mcover_lattice}.]
	  {\label{fig:no_lattice_1}
	  \begin{tikzpicture}\tiny
		\def\x{.6};
		\def\y{.6};
		\draw(.2*\x,2*\y) node{};
		\draw(5.8*\x,2*\y) node{};
		\draw(3*\x,1*\y) node(n1){$0$};
		\draw(3*\x,2*\y) node(n2){$1$};
		\draw(2*\x,3*\y) node(n3){$2$};
		\draw(4*\x,3*\y) node(n4){$3$};
		\draw(3*\x,4*\y) node(n5){$4$};
		\draw(n1) -- (n2);
		\draw(n2) -- (n3);
		\draw(n2) -- (n4);
		\draw(n3) -- (n5);
		\draw(n4) -- (n5);
	  \end{tikzpicture}}\hspace*{1cm}
	\subfigure[The $2$-cover poset of the lattice in Figure~\ref{fig:no_lattice_1}.]{\label{fig:no_lattice_2}
	  \begin{tikzpicture}\tiny
		\def\x{.6};
		\def\y{.6};
		\draw(.2*\x,2*\y) node{};
		\draw(5.8*\x,2*\y) node{};
		\draw(3*\x,1*\y) node(n1){$(0,0)$};
		\draw(3*\x,2*\y) node(n2){$(0,1)$};
		\draw(2*\x,3*\y) node(n3){$(0,2)$};
		\draw(3*\x,3*\y) node(n4){$(1,1)$};
		\draw(4*\x,3*\y) node(n5){$(0,3)$};
		\draw(2*\x,4*\y) node(n6){$(1,2)$};
		\draw(3*\x,4*\y) node(n7){$(0,4)$};
		\draw(4*\x,4*\y) node(n8){$(1,3)$};
		\draw(2*\x,5*\y) node(n9){$(2,2)$};
		\draw(4*\x,5*\y) node(n10){$(3,3)$};
		\draw(2.5*\x,6*\y) node(n11){$(2,4)$};
		\draw(3.5*\x,6*\y) node(n12){$(3,4)$};
		\draw(3*\x,7*\y) node(n13){$(4,4)$};
		\draw(n1) -- (n2);
		\draw(n2) -- (n3);
		\draw(n2) -- (n4);
		\draw(n2) -- (n5);
		\draw(n3) -- (n6);
		\draw(n3) -- (n7);
		\draw(n4) -- (n6);
		\draw(n4) -- (n8);
		\draw(n5) -- (n8);
		\draw(n6) -- (n9);
		\draw(n6) -- (n12);
		\draw(n7) -- (n5);
		\draw(n7) -- (n12);
		\draw(n8) -- (n10);
		\draw(n8) -- (n11);
		\draw(n9) -- (n11);
		\draw(n10) -- (n12);
		\draw(n11) -- (n7);
		\draw(n11) -- (n13);
		\draw(n12) -- (n13);
	  \end{tikzpicture}}
	\caption{An illustration of Theorem~\ref{thm:mcover_lattice}.}
	\label{fig:no_lattice}
\end{figure}

\begin{remark}\label{rem:no_sublattice}
	The proof of Theorem~\ref{thm:mcover_lattice} implies that if $\Pm{\PP}{m}$ is a lattice, then for every $\pb,\qb\in\Pm{P}{m}$ their meet $\pb\wedge\qb$ agrees with the componentwise meet of $\pb$ and $\qb$.  However, we can find simple examples showing that the same is not true for joins, which implies that $\Pm{\PP}{m}$ is not a sublattice of $\PP^{m}$.  Consider for instance the $2$-cover poset of the pentagon poset shown in Figure~\ref{fig:no_join_sub_2}.  The join of the elements $(0,1)$ and $(2,2)$ in $\Pm{N_{5}}{2}$ is $(3,4)$, while their componentwise join is $(2,4)$.
	
	\begin{figure}
		\centering
		\subfigure[The lattice $N_{5}$.]{\label{fig:no_join_sub_1}
		  \begin{tikzpicture}\tiny
			\def\x{.6};
			\def\y{.6};
			\draw(.2*\x,2*\y) node{};
			\draw(3.8*\x,2*\y) node{};
			\draw(2*\x,1*\y) node(n1){$0$};
			\draw(1*\x,2*\y) node(n2){$2$};
			\draw(3*\x,2*\y) node(n3){$1$};
			\draw(1*\x,3*\y) node(n4){$3$};
			\draw(2*\x,4*\y) node(n5){$4$};
			\draw(n1) -- (n2);
			\draw(n1) -- (n3);
			\draw(n2) -- (n4);
			\draw(n3) -- (n5);
			\draw(n4) -- (n5);
		  \end{tikzpicture}}\hspace*{1cm}
		\subfigure[The lattice $\Pm{N_{5}}{2}$.]{\label{fig:no_join_sub_2}
		  \begin{tikzpicture}\tiny
			\def\x{.6};
			\def\y{.6};
			\draw(3*\x,1*\y) node(n1){$(0,0)$};
			\draw(2*\x,2*\y) node(n2){$(0,2)$};
			\draw(1*\x,3*\y) node(n3){$(2,2)$};
			\draw(3*\x,3*\y) node(n4){$(0,3)$};
			\draw(2*\x,4*\y) node(n5){$(2,3)$};
			\draw(6*\x,4*\y) node(n6){$(0,1)$};
			\draw(3*\x,5*\y) node(n7){$(3,3)$};
			\draw(5*\x,5*\y) node(n8){$(0,4)$};
			\draw(7*\x,5*\y) node(n9){$(1,1)$};
			\draw(4*\x,6*\y) node(n10){$(3,4)$};
			\draw(6*\x,6*\y) node(n11){$(1,4)$};
			\draw(5*\x,7*\y) node(n12){$(4,4)$};
			\draw(n1) -- (n2);
			\draw(n1) -- (n6);
			\draw(n2) -- (n3);
			\draw(n2) -- (n4);
			\draw(n3) -- (n5);
			\draw(n4) -- (n5);
			\draw(n4) -- (n8);
			\draw(n5) -- (n7);
			\draw(n6) -- (n8);
			\draw(n6) -- (n9);
			\draw(n7) -- (n10);
			\draw(n8) -- (n10);
			\draw(n8) -- (n11);
			\draw(n9) -- (n11);
			\draw(n10) -- (n12);
			\draw(n11) -- (n12);
		  \end{tikzpicture}}
		\caption{The pentagon poset $N_{5}$, and its $2$-cover poset $\Pm{N_{5}}{2}$.}
		\label{fig:no_join_sub}
	\end{figure}
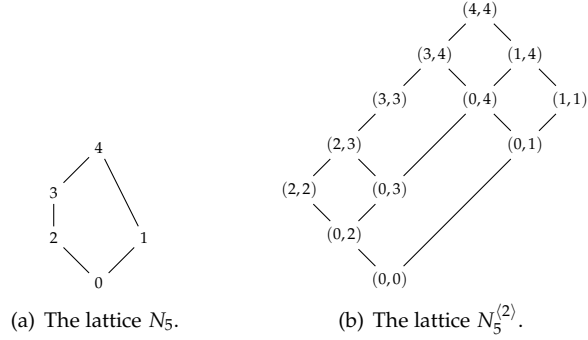
\end{remark}

Now we obtain Theorem~\ref{thm:mcover_lattice_prop} as a corollary.

\begin{proof}[Proof of Theorem~\ref{thm:mcover_lattice_prop}]
	Let $H$ denote the Hasse diagram of $\PP$ with $\hat{0}$ removed, and let $x,y\in P$. 
	
	If $H$ is a tree, then it does not contain a cycle, and it is straightforward to verify that $\PP$ is a lattice and that the meet of any two incomparable elements is the least element.
	
	Conversely, suppose that $\Pm{\PP}{m}$ is a lattice, and thus Theorem~\ref{thm:mcover_lattice} implies that $\PP$ is a lattice such that $p\wedge q\in\{p,q,\hat{0}\}$ for all $p,q\in P$.  Suppose that $H$ is not a tree, and must thus contain a cycle.  We can find distinct elements $p,q,z\in H$ in this cycle such that $z\leq p,q$, which implies $z\leq p\wedge q=\hat{0}$.  It follows that $z=\hat{0}$, which contradicts $z\in H$.  Hence $H$ is a tree, and the proof is completed.
\end{proof}

\begin{remark}
	Interestingly, the posets occurring in Theorem~\ref{thm:mcover_lattice_prop} have an intrinsic connection to the so-called \alert{chord posets} defined in \cite{kim14dyck}.  In particular, if $\PP$ is a bounded poset such that for every $m>0$, the $m$-cover poset $\Pm{\PP}{m}$ is a lattice, then the dual of the proper part of $\PP$ is a chord poset, and vice versa. 
\end{remark}

Finally, we compute the cardinality of the $m$-cover posets which are lattices.

\begin{proposition}\label{prop:mcover_lattice_card}
	Let $\PP$ be a bounded poset with $n$ elements such that $\Pm{\PP}{m}$ is a lattice for all $m>0$.  If $n>1$, then $\bigl\lvert\Pm{P}{m}\bigr\rvert=n\tbinom{m+1}{2}-m^{2}+1$.
\end{proposition}
\begin{proof}
	Theorem~\ref{thm:mcover_lattice_prop} implies that the Hasse diagram of $\PP$ with bottom element removed is a tree with $n-1$ elements.  If $k$ denotes the number of leaves of this tree, then it is immediately clear that $\PP$ has $n-2+k$ covering relations, and Proposition~\ref{prop:mcover_cardinality} implies
	\begin{align*}
		\bigl\lvert\Pm{P}{m}\bigr\rvert & = (n-2)\binom{m}{2} + m(n-1) + 1\\
		& = \frac{(n-2)(m^{2}-m)+2mn-2m+2}{2}\\
		& = \frac{nm^{2}+nm-2m^{2}+2}{2}\\
		& = n\binom{m+1}{2} - m^{2} + 1.
	\end{align*}
\end{proof}

\subsection{The Path Poset}
  \label{sec:path_poset}
A \alert{northeast path} is a lattice path that starts at $(0,0)$ and that consists only of north- and east-steps.  If $\pf$ is a northeast path with $k$ north-steps and $l$ right-steps, then we say that it \alert{has length $k+l$}.  Moreover, we associate with $\pf$ a word $w_{\pf}=w_{1}w_{2}\cdots w_{k+l}$ on the alphabet $\{N,E\}$ in the obvious way, namely by defining
\begin{displaymath}
	w_{j}=\begin{cases}
		N, & \text{if the}\;j\text{-th step of}\;\pf\;\text{is a north-step},\\
		E, & \text{if the}\;j\text{-th step of}\;\pf\;\text{is an east-step},
	\end{cases}
\end{displaymath}
for $j\in\{1,2,\ldots,k+l\}$.  If $w_{j}=E$, then we define 
\begin{displaymath}
	h_{j}=\begin{cases}
		j, & \text{if}\;j=1\;\text{or}\;w_{j-1}=N,\\
		h_{j-1}, & \text{if}\;w_{j-1}=E.
	\end{cases}
\end{displaymath}
Now we define a relation $\lessdot_{\pf}$ on the letters of $w_{\pf}$ in the following way:
\begin{enumerate}[(i)]
	\item $w_{i}\lessdot_{\pf}w_{j}$ if and only if $i\leq j$, $w_{i}=w_{j}=N$, and there is no $s\in\{i+1,i+2,\ldots,j-1\}$ with $w_{s}=N$,\quad and
	\item $w_{j}\lessdot_{\pf}w_{i}$ if and only if $w_{j}=E$, $w_{i}=N$, and $h_{j}-1=i$.
\end{enumerate}
We denote by $\leq_{\pf}$ the reflexive and transitive closure of $\lessdot_{\pf}$, and call it the \alert{path order} of $\pf$.  We remark that the Hasse diagram of $\bigl(\{w_{1},w_{2},\ldots,w_{k+l}\},\leq_{\pf}\bigr)$ is connected if and only if $w_{1}=N$. 

\begin{example}
  \label{ex:path_order}
	Let $\pf$ be given by $w_{\pf}=NENEEENNE$.  We have $h_{2}=2,h_{4}=h_{5}=h_{6}=4$ and $h_{9}=9$.  The path order of $\pf$ is given by the cover relations
	\begin{enumerate}[(i)]
		\item $w_{1}\lessdot_{\pf}w_{3}\lessdot_{\pf}w_{7}\lessdot_{\pf}w_{8}$,\quad and 
		\item $w_{2}\lessdot_{\pf}w_{1}$,\quad $w_{4},w_{5},w_{6}\lessdot_{\pf}w_{3}$,\quad $w_{9}\lessdot_{\pf}w_{8}$.
	\end{enumerate}
	See Figure~\ref{fig:path_order} for an illustration.
	
	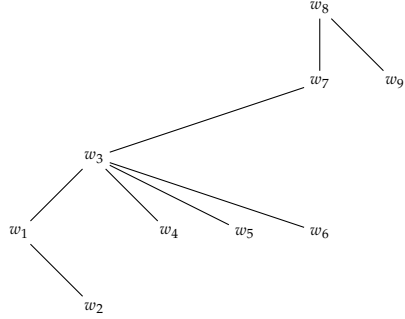
\begin{figure}
		\centering
		\begin{tikzpicture}\tiny
			\def\x{1};
			\def\y{1};
			\draw(2*\x,1*\y) node(e2){$w_{2}$};
			\draw(3*\x,2*\y) node(e4){$w_{4}$};
			\draw(4*\x,2*\y) node(e5){$w_{5}$};
			\draw(5*\x,2*\y) node(e6){$w_{6}$};
			\draw(6*\x,4*\y) node(e9){$w_{9}$};
			\draw(1*\x,2*\y) node(n1){$w_{1}$};
			\draw(2*\x,3*\y) node(n3){$w_{3}$};
			\draw(5*\x,4*\y) node(n7){$w_{7}$};
			\draw(5*\x,5*\y) node(n8){$w_{8}$};
			\draw(n1) -- (n3) -- (n7) -- (n8);
			\draw(e2) -- (n1);
			\draw(e4) -- (n3);
			\draw(e5) -- (n3);
			\draw(e6) -- (n3);
			\draw(e9) -- (n8);
		\end{tikzpicture}
		\caption{The path order on the northeast path $\pf$ given by the word $w_{\pf}=NENEEENNE$.}
		\label{fig:path_order}
	\end{figure}
\end{example}

\begin{definition}
  \label{def:path_poset}
	Let $\PP=(P,\leq_{P})$ be a bounded poset, and let $\pf$ be a northeast path having length $k+l$.  The \alert{path poset} $\PP_{\pf}=(P_{\pf},\leq_{P_{\pf}})$ is defined by $P_{\pf}=P\uplus\{w_{1},w_{2},\ldots,w_{k+l}\}$, where $\uplus$ denotes 	disjoint set union, and
	\begin{displaymath}
		p\leq_{P_{\pf}}q\quad\text{if and only if}\quad\begin{cases}
			p,q\in P\;\text{and}\;p\leq_{P}q,\\
			p,q\notin P\;\text{and}\;p\leq_{\pf}q,\\
			p=\hat{0}\;\text{and}\;q=w_{j}=E\;\text{for some}\;j,\\
			p\in P\;\text{and}\;q=w_{j}=N\;\text{for some}\;j,
		\end{cases}
	\end{displaymath}
	for every $p,q\in P_{\pf}$.
\end{definition}

See Figure~\ref{fig:path_poset} for an illustration. 

\begin{remark}
  \label{rem:exclude_trivial_paths}
	If we consider the empty poset $\mathcal{\emptyset}$ as being bounded, then the path poset $\mathcal{\emptyset}_{\pf}$ is precisely the poset $\bigl(\{w_{1},w_{2},\ldots,w_{k+l}\},\leq_{\pf}\bigr)$.  Moreover, if $\pf$ starts with an east-step, then $\PP_{\pf}$ does not possess a greatest element.  Since our focus lies on bounded posets, we usually assume that $\PP$ is non-empty and that $\pf$ starts with a north-step. 
\end{remark}

\begin{figure}
	\centering
	\begin{tikzpicture}\tiny
		\def\x{1};
		\def\y{.5};
		\draw(3*\x,1*\y) node(p0){$0$};
		\draw(1*\x,2*\y) node(p1){$1$};
		\draw(1*\x,3*\y) node(p2){$2$};
		\draw(1*\x,4*\y) node(p3){$3$};
		\draw(1*\x,5*\y) node(p4){$4$};
		\draw(2*\x,6*\y) node(q1){$w_{1}$};
		\draw(2*\x,2*\y) node(q2){$w_{2}$};
		\draw(2.5*\x,2*\y) node(q3){$w_{4}$};
		\draw(3*\x,2*\y) node(q4){$w_{5}$};
		\draw(3.5*\x,2*\y) node(q5){$w_{6}$};
		\draw(5*\x,2*\y) node(q6){$w_{9}$};
		\draw(3*\x,7*\y) node(q7){$w_{3}$};
		\draw(4*\x,8*\y) node(q8){$w_{7}$};
		\draw(5*\x,9*\y) node(q9){$w_{8}$};
		\draw[thick](p0) -- (p1) -- (p2) -- (p3) -- (p4);
		\draw(p4) -- (q1) -- (q7) -- (q8) -- (q9);
		\draw(p0) -- (q2) -- (q1);
		\draw(p0) -- (q3) -- (q7);
		\draw(p0) -- (q4) -- (q7);
		\draw(p0) -- (q5) -- (q7);
		\draw(p0) -- (q6) -- (q9);
	\end{tikzpicture}
	\caption{The path poset $(C_{5})_{\pf}$, where $\pf$ is given by $w_{\pf}=NENEEENNE$.}
	\label{fig:path_poset}
\end{figure}
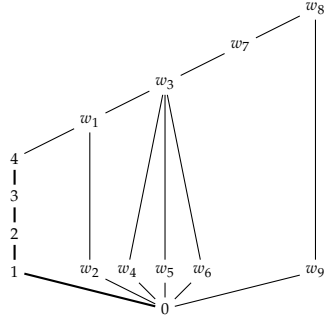

Now let $\pf$ be a northeast path, let $w_{\pf}=w_{1}w_{2}\cdots w_{n}$ be the corresponding word, and let $\PP$ be a bounded poset.  Set $\PP_{i}=\PP_{w_{1}w_{2}\cdots w_{i}}$.  It is immediately clear that $\PP_{\pf}=\bigl(\PP_{n-1}\bigr)_{w_{n}}$.  Thus it is sufficient to explicitly describe how to construct the $m$-cover poset $\Pm{(\PP_{\pf})}{m}$ from $\Pm{\PP}{m}$, provided that $\pf=N$ or $\pf=E$.  
 
\subsubsection{Adding a North-Step}
First let $\pf=N$.  Then, $\PP_{\pf}=\PP_{N}$ is just $\PP$ with a new greatest element, say $N$, attached.  The additional elements in $\Pm{(\PP_{N})}{m}$ are of the form $\bigl(\hat{0}^{l_{0}},\hat{1}^{l_{1}},N^{l_{2}}\bigr)$ with $l_{2}>0$.  Thus it is immediate that there are $\tbinom{m+1}{2}$ such elements.  Let now $G_N=\bigl\{(a,b)\in\mathbb{N}^{2}\mid 1\leq a\leq b\leq m\bigr\}$.  For all $(a_{1},b_{1}),(a_{2},b_{2})\in G_N$ we set $(a_{1},b_{1})\leq_{G_{N}}(a_{2},b_{2})$ if and only if $a_{1}\leq a_{2}$ and $b_{1}\leq b_{2}$.

If we define $\phi_{N}(a,b)=\bigl(\hat{0}^{m-b},\hat{1}^{b-a},N^{a}\bigr)$ for $(a,b)\in G_N$, then it is easy to check that for all $(a_{1},b_{1}),(a_{2},b_{2})\in G_N$, we have $(a_{1},b_{1})\leq_{G_{N}}(a_{2},b_{2})$ if and only if $\phi_{N}(a_{1},b_{1})\leq_{P_{N}}\phi_{N}(a_{2},b_{2})$.  Let us additionally abbreviate $\phi_{N}(0,b)=\bigl(\hat{0}^{m-b},\hat{1}^{b}\bigr)$ for $b\in\{0,1,\ldots,m\}$.  Then, $\Pm{(P_{N})}{m}=\Pm{P}{m}\uplus\bigl\{\phi_{N}(a,b)\mid (a,b)\in G_N\bigr\}$, and for all $\pb,\qb\in\Pm{(P_{N})}{m}$, we have 
\begin{equation}
  \label{eq:cover_north}
	\pb\lessdot_{P_{N}}\qb\quad\text{if and only if}\quad\begin{cases}
		\pb,\qb\in\Pm{P}{m},\;\text{and}\;\pb\lessdot_{P}\qb,\\
		\pb,\qb\notin\Pm{P}{m},\;\text{and}\;\phi_{N}^{-1}(\pb)\lessdot_{G_{N}}\phi_{N}^{-1}(\qb),\\
		\pb=\phi_{N}(0,b),\qb=\phi_{N}(1,b),\;\text{for}\;1\leq b\leq m.
	\end{cases}
\end{equation}

\subsubsection{Adding an East-Step}
Now let $\pf=E$.  Then, $\PP_{\pf}=\PP_{E}$ is $\PP$ with an additional element, say $E$, satisfying $\hat{0}\lessdot_{P}E\lessdot_{P}\hat{1}$.  (This implies that these are the only coverings in $\PP_{E}$ not present in $\PP$.)  The additional elements in $\Pm{(\PP_{E})}{m}$ are of the form $\bigl(\hat{0}^{l_{0}},E^{l_{1}},\hat{1}^{l_{2}}\bigr)$ with $l_{1}>0$, and again it is immediate that there are $\tbinom{m+1}{2}$ such elements.  Let $G_E=\bigl\{(a,b)\in\mathbb{N}^{2}\mid 0\leq a<b\leq m\bigr\}$.  For all $(a_{1},b_{1}),(a_{2},b_{2})\in G_E$ we set $(a_{1},b_{1})\leq_{G_{E}}(a_{2},b_{2})$ if and only if $a_{1}\leq a_{2}$ and $b_{1}\leq b_{2}$.

If we define $\phi_{E}(a,b)=\bigl(\hat{0}^{m-b},E^{b-a},\hat{1}^{a}\bigr)$ for $(a,b)\in G_E$, then it is easy to check that for all $(a_{1},b_{1}),(a_{2},b_{2})\in G_E$, we have $(a_{1},b_{1})\leq_{G_{E}}(a_{2},b_{2})$ if and only if $\phi_{E}(a_{1},b_{1})\leq_{P_{E}}\phi_{E}(a_{2},b_{2})$.  In addition, we abbreviate $\phi_{E}(b,b)=\bigl(\hat{0}^{m-b},\hat{1}^{b}\bigr)$ for $b\in\{0,1,\ldots,m\}$.  Then, $\Pm{(P_{E})}{m}=\Pm{P}{m}\uplus\bigl\{\phi_{E}(a,b)\mid (a,b)\in G_E\bigr\}$, and for all $\pb,\qb\in\Pm{(P_{E})}{m}$, we have
\begin{equation}
  \label{eq:cover_east}
	\pb\lessdot_{P_{E}}\qb\quad\text{if and only if}\quad\begin{cases}
		\pb,\qb\in\Pm{P}{m},\;\text{and}\;\pb\lessdot_{P}\qb,\\
		\pb,\qb\notin\Pm{P}{m},\;\text{and}\;\phi_{E}^{-1}(\pb)\lessdot_{G_{E}}\phi_{E}^{-1}(\qb),\\
		\pb=\phi_{E}(b,b),\qb=\phi_{E}(b,b+1),\;\text{for}\;0\leq b<m.
	\end{cases}
\end{equation}

\subsubsection{A Special Class of Path Posets}
For proving certain topological properties of the $m$-cover posets, namely Theorems~\ref{thm:mcover_path_poset} and \ref{thm:mcover_trim}, it turns out that we need to consider a particular class of path posets, which we will define next.  For nonnegative integers $k$ and $l$, let $\PP_{k,l}$ denote the bounded poset whose proper part is the disjoint union of a $k$-element chain (the elements of which are denoted by $c_{1},c_{2},\ldots,c_{k}$) and an $l$-element antichain (the elements of which are denoted by $a_{1},a_{2},\ldots,a_{l}$).  See Figure~\ref{fig:p33} for an example.  Our main interest lies in the path posets arising from $\PP_{k,l}$ and some northeast path.

\begin{remark}
  \label{rem:exclude_trivial}
	If $k=0$ and $l=0$, then $\PP_{0,0}$ is a $2$-element chain, and $\Pm{\PP_{0,0}}{m}$ is a distributive lattice for all $m>0$.  Hence $\PP_{0,0}$ satisfies the properties stated in Theorems~\ref{thm:mcover_path_poset} and ~\ref{thm:mcover_trim}. If $k=0$ and $l>0$, then $\PP_{0,l}\cong\PP_{1,l-1}$. Hence it is always safe to assume $k>0$.
\end{remark}

\begin{figure}
	\centering
	\subfigure[The lattice $\PP_{3,3}$.]{\label{fig:p33}
	  \begin{tikzpicture}\tiny
		\def\x{.6};
		\def\y{.6};
		\draw(.5*\x,3*\y) node{};
		\draw(4.5*\x,3*\y) node{};
		\draw(2.5*\x,1*\y) node(n1){$\hat{0}$};
		\draw(1*\x,2*\y) node(n2){$c_{1}$};
		\draw(2*\x,2*\y) node(n3){$a_{1}$};
		\draw(3*\x,2*\y) node(n4){$a_{2}$};
		\draw(4*\x,2*\y) node(n5){$a_{3}$};
		\draw(1*\x,3*\y) node(n6){$c_{2}$};
		\draw(1*\x,4*\y) node(n7){$c_{3}$};
		\draw(2.5*\x,5*\y) node(n8){$\hat{1}$};
		\draw(n1) -- (n2);
		\draw(n1) -- (n3);
		\draw(n1) -- (n4);
		\draw(n1) -- (n5);
		\draw(n2) -- (n6);
		\draw(n3) -- (n8);
		\draw(n4) -- (n8);
		\draw(n5) -- (n8);
		\draw(n6) -- (n7);
		\draw(n7) -- (n8);
	  \end{tikzpicture}}\hspace*{2cm}
	\subfigure[The lattice $\PP_{3,3}^{\langle 2\rangle}$.]{
	  \label{fig:2p33}
	  \begin{tikzpicture}\tiny
		\def\x{1};
		\def\y{.6};
		\draw(3*\x,1*\y) node(n1){$(\hat{0},\hat{0})$};
		\draw(2*\x,2*\y) node(n2){$(\hat{0},c_{1})$};
		\draw(3*\x,2*\y) node(n3){$(\hat{0},a_{1})$};
		\draw(4*\x,2*\y) node(n4){$(\hat{0},a_{2})$};
		\draw(5*\x,2*\y) node(n5){$(\hat{0},a_{3})$};
		\draw(1*\x,3*\y) node(n6){$(c_{1},c_{1})$};
		\draw(2*\x,3*\y) node(n7){$(\hat{0},c_{2})$};
		\draw(1*\x,4*\y) node(n8){$(c_{1},c_{2})$};
		\draw(2*\x,4*\y) node(n9){$(\hat{0},c_{3})$};
		\draw(1*\x,5*\y) node(n10){$(c_{2},c_{2})$};
		\draw(3*\x,5*\y) node(n11){$(\hat{0},\hat{1})$};
		\draw(4*\x,5*\y) node(n12){$(a_{1},a_{1})$};
		\draw(5*\x,5*\y) node(n13){$(a_{2},a_{2})$};
		\draw(6*\x,5*\y) node(n14){$(a_{3},a_{3})$};
		\draw(1*\x,6*\y) node(n15){$(c_{2},c_{3})$};
		\draw(1*\x,7*\y) node(n16){$(c_{3},c_{3})$};
		\draw(2*\x,8*\y) node(n17){$(c_{3},\hat{1})$};
		\draw(3*\x,8*\y) node(n18){$(a_{1},\hat{1})$};
		\draw(4*\x,8*\y) node(n19){$(a_{2},\hat{1})$};
		\draw(5*\x,8*\y) node(n20){$(a_{3},\hat{1})$};
		\draw(3*\x,9*\y) node(n21){$(\hat{1},\hat{1})$};
		\draw(n1) -- (n2);
		\draw(n1) -- (n3);
		\draw(n1) -- (n4);
		\draw(n1) -- (n5);
		\draw(n2) -- (n6);
		\draw(n2) -- (n7);
		\draw(n3) -- (n11);
		\draw(n3) -- (n12);
		\draw(n4) -- (n11);
		\draw(n4) -- (n13);
		\draw(n5) -- (n11);
		\draw(n5) -- (n14);
		\draw(n6) -- (n8);
		\draw(n7) -- (n8);
		\draw(n7) -- (n9);
		\draw(n8) -- (n10);
		\draw(n9) -- (n11);
		\draw(n9) -- (n15);
		\draw(n10) -- (n15);
		\draw(n11) -- (n17);
		\draw(n11) -- (n18);
		\draw(n11) -- (n19);
		\draw(n11) -- (n20);
		\draw(n12) -- (n18);
		\draw(n13) -- (n19);
		\draw(n14) -- (n20);
		\draw(n15) -- (n16);
		\draw(n16) -- (n17);
		\draw(n17) -- (n21);
		\draw(n18) -- (n21);
		\draw(n19) -- (n21);
		\draw(n20) -- (n21);
	  \end{tikzpicture}}
	\caption{The lattice $\PP_{3,3}$ and its $2$-cover poset.}
	\label{fig:cover_examples}
\end{figure}
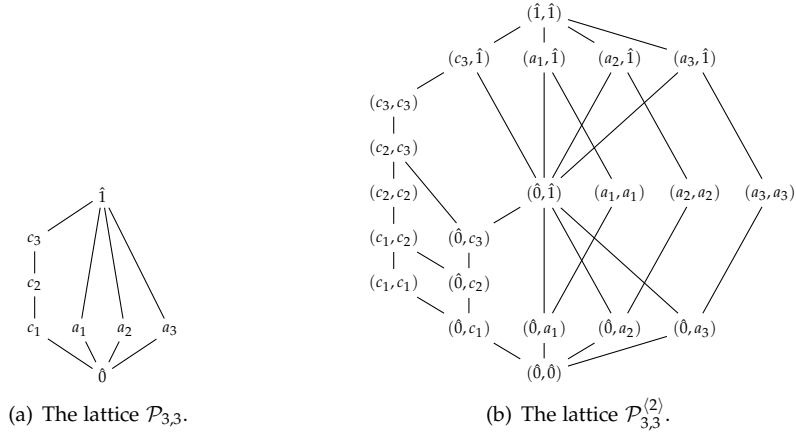

Let us first investigate a certain statistic of the poset $\PP_{k,l}$.  Given any poset $\PP=(P,\leq)$, denote by $u_{n}$ the number of elements in $P$ having precisely $n$ upper covers, and consider the generating function $\text{UF}(x;\PP)=\sum_{n\geq 0}{u_{n}x^{n}}$.  Analogously, denote by $l_{n}$ the number of elements in $P$ having precisely $n$ lower covers, and consider the generating function $\text{LF}(x;\PP)=\sum_{n\geq 0}{l_{n}x^{n}}$.  We have the following result.

\begin{proposition}\label{prop:cover_statistic}
	For $k,l,m>0$, we have 
	\begin{displaymath}
		\text{UF}\Bigl(x;\Pm{\PP_{k,l}}{m}\Bigr) = 1+(k+l)mx+\left((k+l)\binom{m}{2}\right)x^{2}+mx^{l+1} = \text{LF}\Bigl(x;\Pm{\PP_{k,l}}{m}\Bigr).
	\end{displaymath}
\end{proposition}
\begin{proof}
	First of all we note that the greatest element of $\Pm{\PP_{k,l}}{m}$ is the unique element having no upper cover.  Then $\PP_{k,l}$ has precisely $k+l$ meet-irreducible elements, and Proposition~\ref{prop:mcover_irreducibles} implies that $\Pm{\PP_{k,l}}{m}$ has $m(k+l)$ meet-irreducibles, namely elements with exactly one upper cover.  Now suppose that $\xx\in\Pm{P_{k,l}}{m}$ is neither the greatest element nor meet-irreducible.  Let $\hat{0}$ denote the least element of $\PP_{k,l}$, and let $\hat{1}$ denote the greatest element of $\PP_{k,l}$.  Recall that the proper part of $\PP_{k,l}$ is the disjoint union of a $k$-element chain (whose elements are denoted by $c_{1},c_{2},\ldots,c_{k}$), and an $l$-element antichain (whose elements are denoted by $a_{1},a_{2},\ldots,a_{l}$).  Then, $\xx$ is necessarily of one of the following forms:
	
	(i) Let $\xx=\bigl(\hat{0}^{s},\hat{1}^{m-s}\bigr)$ for $s>0$.  We find that $\xx$ has precisely $l+1$ upper covers, namely $\bigl(\hat{0}^{s-1},c_{k},\hat{1}^{m-s}\bigr)$ and $\bigl(\hat{0}^{s-1},a_{i},\hat{1}^{m-s}\bigr)$ for $i\in\{1,2,\ldots,l\}$.
	
	(ii) Let $\xx=\bigl(\hat{0}^{s_{0}},a_{i}^{s_{1}},\hat{1}^{s_{2}}\bigr)$ for $i\in\{1,2,\ldots,l\}$ and $s_{0},s_{1},s_{2}\in\{1,2,\ldots,m-1\}$.  Then, $\xx$ has precisely two upper covers, namely $\bigl(\hat{0}^{s_{0}-1},a_{i}^{s_{1}+1},\hat{1}^{s_{2}}\bigr)$ and $\bigl(\hat{0}^{s_{0}},a_{i}^{s_{1}-1},\hat{1}^{s_{2}+1}\bigr)$.
	
	(iii) Let $\xx=\bigl(\hat{0}^{s_{0}},c_{i}^{s_{1}},c_{i+1}^{s_{2}}\bigr)$ for $i\in\{1,2,\ldots,k\}$ and $s_{0},s_{1}\in\{1,2,\ldots,m-1\}$, $s_{2}\in\{0,1,\ldots,m-1\}$, where we set $c_{k+1}=\hat{1}$.  If $s_{2}=0$, then $\xx=\bigl(\hat{0}^{s_{0}},c_{1}^{m-s_{0}}\bigr)$, and $\xx$ has precisely two upper covers, namely the elements $\bigl(\hat{0}^{s_{0}-1},c_{1}^{m-s_{0}+1}\bigr)$ and $\bigl(\hat{0}^{s_{0}},c_{1}^{m-s_{0}-1},c_{2}\bigr)$.  Otherwise, if $s_{2}>0$, then $\xx$ has again precisely two upper covers, namely the elements $\bigl(\hat{0}^{s_{0}-1},c_{i}^{s_{1}+1},c_{i+1}^{s_{2}}\bigr)$ and $\bigl(\hat{0}^{s_{0}},c_{i}^{s_{1}-1},c_{i+1}^{s_{2}+1}\bigr)$.
	
	If we add all these possibilities, then we obtain the result.  The reasoning for $\text{LF}\Bigl(x;\Pm{\PP_{k,l}}{m}\Bigr)$ is analogous.
\end{proof}

It turns out that we can characterize the path posets of $\PP_{k,l}$ in terms of their Hasse diagrams.  For that, however, we need some further notation.  We abbreviate $(\PP_{k,l})_{\pf}$ by $\PP_{k,l;\pf}$.  Let $H$ be a rooted tree with root $r$, and let $x$ be some vertex of $H$.  If $c$ is a fixed child of $x$, then we say that all elements $y\in H$ such that the unique path from $r$ to $y$ (in $H$) passes through $c$ form a \alert{subtree of $x$}.  Moreover, we say that $H$ \alert{satisfies Condition~\eqref{eq:condition}} if and only if $H$ satisfies the following, recursive condition:
\begin{align}\label{eq:condition}
	\parbox{.9\textwidth}{either the root of $H$ has no subtree, or the root of $H$ has at most one subtree with more than one element, and this subtree again satisfies Condition~\eqref{eq:condition}.}\tag{S}
\end{align}
See Figure~\ref{fig:condition_s} for an example.

\begin{figure}
	\centering
	\subfigure[A tree that satisfies Condition~\eqref{eq:condition}.]{
		\begin{tikzpicture}\tiny
			\def\x{.6};
			\def\y{.6};
			\draw(1*\x,1*\y) node[draw,circle,scale=.4](n1){};
			\draw(2*\x,1*\y) node[draw,circle,scale=.4](n2){};
			\draw(3*\x,1*\y) node[draw,circle,scale=.4](n3){};
			\draw(2*\x,2*\y) node[draw,circle,scale=.4](n4){};
			\draw(3*\x,3*\y) node[draw,circle,scale=.4](n5){};
			\draw(4*\x,3*\y) node[draw,circle,scale=.4](n6){};
			\draw(5*\x,3*\y) node[draw,circle,scale=.4](n7){};
			\draw(4*\x,4*\y) node[draw,circle,scale=.4](n8){};
			\draw(5*\x,4*\y) node[draw,circle,scale=.4](n9){};
			\draw(6*\x,4*\y) node[draw,circle,scale=.4](n10){};
			\draw(5*\x,5*\y) node[draw,circle,scale=.4](n11){};
			\draw(n1) -- (n4) -- (n5) -- (n8) -- (n11);
			\draw(n2) -- (n4);
			\draw(n3) -- (n4);
			\draw(n6) -- (n8);
			\draw(n7) -- (n8);
			\draw(n9) -- (n11);
			\draw(n10) -- (n11);
		\end{tikzpicture}
	}\qquad
	\subfigure[A tree that does not satisfy Condition~\eqref{eq:condition}.]{
		\begin{tikzpicture}\tiny
			\def\x{.6};
			\def\y{.6};
			\draw(1*\x,1*\y) node[draw,circle,scale=.4](n1){};
			\draw(2*\x,1*\y) node[draw,circle,scale=.4](n2){};
			\draw(3*\x,1*\y) node[draw,circle,scale=.4](n3){};
			\draw(2*\x,2*\y) node[draw,circle,scale=.4](n4){};
			\draw(3*\x,3*\y) node[draw,circle,scale=.4](n5){};
			\draw(4*\x,3*\y) node[draw,circle,scale=.4](n6){};
			\draw(5*\x,3*\y) node[draw,circle,scale=.4](n7){};
			\draw(6*\x,3*\y) node[draw,circle,scale=.4](e1){};
			\draw(4*\x,4*\y) node[draw,circle,scale=.4](n8){};
			\draw(5*\x,4*\y) node[draw,circle,scale=.4](n9){};
			\draw(6*\x,4*\y) node[draw,circle,scale=.4](n10){};
			\draw(5*\x,5*\y) node[draw,circle,scale=.4](n11){};
			\draw(n1) -- (n4) -- (n5) -- (n8) -- (n11);
			\draw(n2) -- (n4);
			\draw(n3) -- (n4);
			\draw(n6) -- (n8);
			\draw(n7) -- (n8);
			\draw(n9) -- (n11);
			\draw(n10) -- (n11);
			\draw(e1) -- (n10);
		\end{tikzpicture}
	}
	\caption{Illustration of Condition~\eqref{eq:condition}.}
	\label{fig:condition_s}
\end{figure}
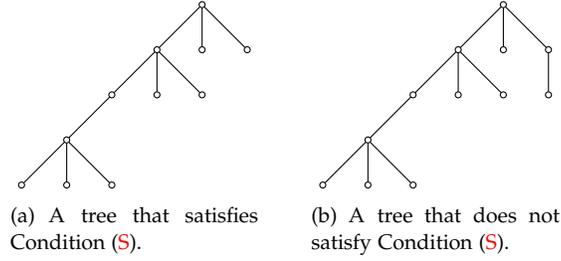

\begin{remark}
	We notice that if $\mathbf{1}$ denotes the singleton-poset, namely the poset consisting of a single element, then we have $\PP_{k,l}\cong\mathbf{1}_{N^{k+1}E^{l}}$.  However, since we frequently use posets of the form $\PP_{k,l}$ as an induction base for results on the path-poset $\PP_{k,l;\pf}$, we keep this explicit notion. 
\end{remark}

\begin{proposition}\label{prop:path_poset_char}
	A bounded poset $\PP$ with more than one element is isomorphic to $\PP_{k,l;\pf}$, for some $k,l\geq 0$ and some northeast path $\pf$, if and only if the Hasse diagram of $\PP$ with $\hat{0}$ removed is a tree rooted at $\hat{1}$ satisfying Condition \eqref{eq:condition}.
\end{proposition}
\begin{proof}
	First suppose that $\PP\cong\PP_{k,l;\pf}$, for some northeast path $\pf$.  Since we assumed that $\PP$ is not a singleton, it follows that either $k>0$ or $l>0$.  Let $H$ be the Hasse diagram of $\PP$ with $\hat{0}$ removed.  We need to show that $H$ satisfies Condition~\eqref{eq:condition}, and we proceed by induction on the length of $\pf$.  If $\pf$ is empty, then $\PP\cong\PP_{k,l}$, and $H$ clearly satisfies Condition~\eqref{eq:condition}.  Now suppose that $\pf$ has length $n$ and that the claim is true for all northeast paths of length $<n$.  Let $\bar{\pf}$ denote the subpath of $\pf$ consisting of the first $n-1$ steps, and write $\bar{\PP}=\PP_{k,l;\bar{\pf}}$.  Let $\bar{H}$ denote the Hasse diagram of $\bar{\PP}$ with $\hat{0}$ removed.  We distinguish two cases.
	
	(i) $\pf$ ends with an east-step.  By construction, $\bar{\PP}$ and $\PP$ differ in exactly one element, say $E$, satisfying $\hat{0}\lessdot E\lessdot\hat{1}$.  In particular, $H$ is constructed from $\bar{H}$ by adding a new, single leaf to $\hat{1}$.  Since by induction assumption $\bar{H}$ satisfies Condition~\eqref{eq:condition}, so does $H$.
	
	(ii) $\pf$ ends with a north-step.  By construction, $\bar{\PP}$ and $\PP$ differ in exactly one element, say $N$, satisfying $\hat{1}\lessdot N$.  In particular, $H$ is constructed from $\bar{H}$ by adding a new root.  Thus $\bar{H}$ is the unique subtree of $N$, and since $\bar{H}$ satisfies Condition~\eqref{eq:condition}, so does $H$.
	
	\smallskip
	
	Conversely, let $H$ be the Hasse diagram of $\PP$ with $\hat{0}$ removed, and suppose that $H$ satisfies Condition~\eqref{eq:condition}.  We proceed by induction on the length of $\PP$.  Since we require $\PP$ to have at least two elements, we have $\ell(\PP)>0$.  If $\ell(\PP)=1$, then $\PP$ is a $2$-chain, and we have $\PP\cong\PP_{0,0;\emptyset}$.  So suppose that $\ell(\PP)=n$ and that the claim is true for all such posets with length $<n$.  If $\PP$ is itself a chain, say of length $k$, then we have $\PP\cong\PP_{k,0}$.  Otherwise, since $\PP$ satisfies Condition~\eqref{eq:condition}, we can find a unique subtree of $\hat{1}$, say $H_{0}$, that has more than one element, and we denote by $z_{0}$ the root of $H_{0}$.  Further, let $H_{1},H_{2},\ldots,H_{s}$ denote the other (one-element) subtrees of $\hat{1}$.  By induction assumption, the interval $[\hat{0},z_{0}]$ is isomorphic to some $\PP_{k,l;\bar{\pf}}$, and if we consider the path $\pf=\bar{\pf}NE^{s}$ that is constructed from $\bar{\pf}$ by 
subsequently adding one north-step and $s$ east-steps, then we see immediately that $\PP\cong\PP_{k,l;\pf}$, and we are done.
\end{proof}

\subsection{Topology of the $m$-Cover Poset}
  \label{sec:mcover_topology} 
In the last part of this section, we prove Theorems~\ref{thm:mcover_path_poset} and \ref{thm:mcover_trim}, \ie we derive a characterization of the posets whose $m$-cover posets are left-modular or EL-shellable lattices, and we explicitly characterize which of those posets yield $m$-cover posets that are trim lattices.  We first recall some helpful results.

\begin{theorem}[\cite{liu00left}*{Theorem~1.4}]\label{thm:left_modular}
	Let $\PP=(P,\leq)$ be a finite lattice, and let $x\in P$.  The following are equivalent:
	\begin{enumerate}[(a)]
		\item the element $x$ is left-modular;\quad and
		\item for any $y,z\in P$ with $y\lessdot z$, we have $x\wedge y=x\wedge z$ or $x\vee y=x\vee z$, but not both.
	\end{enumerate}
\end{theorem}

\begin{theorem}[\cite{mcnamara06poset}*{Theorem~8}]\label{thm:left_modular_el}
	Every left-modular lattice is EL-shellable.
\end{theorem}

Now we consider a special case of Theorem~\ref{thm:mcover_path_poset}, namely the case where $\pf$ is the empty path.

\begin{lemma}\label{lem:induction_base}
	If $\PP\cong\PP_{k,l}$ for $k,l\geq 0$, then $\Pm{\PP}{m}$ is left-modular for $m>0$.
\end{lemma}
\begin{proof}
	Recall from Remark~\ref{rem:exclude_trivial} that it is safe to assume $k>0$.  It is immediate that the chain $\hat{0}\lessdot c_{1}\lessdot c_{2}\lessdot\cdots\lessdot c_{k}\lessdot\hat{1}$ is a left-modular chain of $\PP_{k,l}$.  Set $c_{0}=\hat{0}$ and $c_{k+1}=\hat{1}$.  Furthermore, let $\xb_{0,m}=(\hat{0}^m)$, and for $i\in\{1,2,\ldots,k+1\}$ and $j\in\{1,2,\ldots,m\}$, define $\xb_{i,j}=(c_{i-1}^{m-j},c_{i}^{j})$.  Then, $\xb_{i,j}\in\Pm{\PP}{m}$, and we claim that the chain $C$ given by
	\begin{equation}\label{eq:chain}
		\xb_{0,m}\lessdot\xb_{1,1}\lessdot\xb_{1,2}\lessdot\cdots\lessdot\xb_{1,m}\lessdot\xb_{2,1}\lessdot\cdots\lessdot\xb_{k+1,m}
	\end{equation}
	is a maximal saturated chain in $\Pm{\PP}{m}$ consisting of left-modular elements.  The maximality follows immediately from Proposition~\ref{prop:mcover_cardinality}, since $C$ is precisely the chain of maximal length constructed in the proof there.  It remains to show that each $\xb_{i,j}$ is left-modular.  For that, fix indices $i$ and $j$, and let $\pb,\qb\in\Pm{P}{m}$ with $\pb\lessdot\qb$.  Without loss of generality, we can assume that $\pb\neq\bigl(\hat{0}^{m}\bigr)$ and $\qb\neq\bigl(\hat{1}^{m}\bigr)$, because these elements are trivially left-modular.  Essentially $\qb$ can be of two forms:
	
	(i) Let $\qb=\bigl(\hat{0}^{t_{0}},c_{s}^{t_{1}},c_{s+1}^{t_{2}}\bigr)$ with $t_{0}<m$.  We have two choices for $\pb$, namely $\pb_{1}=\bigl(\hat{0}^{t_{0}+1},c_{s}^{t_{1}-1},c_{s+1}^{t_{2}}\bigr)$ or $\pb_{2}=\bigl(\hat{0}^{t_{0}},c_{s}^{t_{1}+1},c_{s+1}^{t_{2}-1}\bigr)$. 
	We have
	\begin{displaymath}
		\xb_{i,j}\wedge\qb=\begin{cases}
			\bigl(\hat{0}^{t_{0}},c_{s}^{t_{1}},c_{s+1}^{t_{2}}\bigr), & \text{if}\;s<i,\\
			\bigl(\hat{0}^{t_{0}},c_{i-1}^{m-t_{0}-\min\{j,t_{2}\}},c_{i}^{\min\{j,t_{2}\}}\bigr), & \text{if}\;s=i,\\
			\bigl(\hat{0}^{t_{0}},c_{i-1}^{m-j-t_{0}},c_{i}^{j}\bigr), & \text{if}\;s>i.
		\end{cases}
	\end{displaymath}
	It follows that $\xb_{i,j}\wedge\pb_{1}\neq\xb_{i,j}\wedge\qb$ and $\xb_{i,j}\wedge\pb_{2}\neq\xb_{i,j}\wedge\qb$ if and only if $s<i$ or $s=i$ and $t_{2}\leq j$.  On the other hand, we have
	\begin{displaymath}
		\xb_{i,j}\vee\qb=\begin{cases}
			\bigl(c_{i-1}^{m-j},c_{i}^{j}\bigr), & \text{if}\;s<i,\\
			\bigl(c_{i-1}^{m-\max\{j,t_{2}\}},c_{i}^{\max\{j,t_{2}\}}\bigr), & \text{if}\;s=i,\\
			\bigl(c_{s}^{m-t_{2}},c_{s+1}^{t_{2}}\bigr), & \text{if}\;s>i,
		\end{cases}
	\end{displaymath}
	and it follows that $\xb_{i,j}\vee\pb_{1}=\xb_{i,j}\wedge\qb$ and $\xb_{i,j}\wedge\pb_{2}=\xb_{i,j}\wedge\qb$ if and only if $s<i$ or $s=i$ and $t_{2}\leq j$.  The element $\xb_{i,j}$ satisfies Condition~(ii) of Theorem~\ref{thm:left_modular}, and is thus left-modular.
	
	(ii) Let $\qb=\bigl(\hat{0}^{t_{0}},a_{s}^{t_{1}},\hat{1}^{t_{2}}\bigr)$ with $t_{0},t_{2}<m$.  We have two choices for $\pb$, namely $\pb_{1}=\bigl(\hat{0}^{t_{0}+1},a_{s}^{t_{1}-1},\hat{1}^{t_{2}}\bigr)$ or $\pb_{2}=\bigl(\hat{0}^{t_{0}},a_{s}^{t_{1}+1},\hat{1}^{t_{2}-1}\bigr)$.  We have
	\begin{displaymath}
		\xb_{i,j}\wedge\qb = \begin{cases}
			\bigl(\hat{0}^{m-t_{2}},c_{i}^{t_{2}}\bigr), & \text{if}\;t_{2}\leq j,\\
			\bigl(\hat{0}^{m-t_{2}},c_{i-1}^{t_{2}-j},c_{i}^{j}\bigr), & \text{if}\;t_{2}>j.
		\end{cases}
	\end{displaymath}
	It follows that $\xb_{i,j}\wedge\pb_{1}=\xb_{i,j}\wedge\qb$, and $\xb_{i,j}\wedge\pb_{2}\neq\xb_{i,j}\wedge\qb$.  On the other hand, we have 
	\begin{displaymath}
		\xb_{i,j}\vee\qb=\bigl(c_{k}^{t_{0}},\hat{1}^{m-t_{0}}\bigr),
	\end{displaymath}
	and it follows that $\xb_{i,j}\vee\pb_{1}\neq\xb_{i,j}\vee\qb$ and $\xb_{i,j}\vee\pb_{2}=\xb_{i,j}\vee\qb$ as desired for Condition~(ii) in Theorem~\ref{thm:left_modular}.  Hence $\xb_{i,j}$ is left-modular. 

	We conclude that the chain in \eqref{eq:chain} consists of left-modular elements, which by definition means that $\Pm{\PP}{m}$ is left-modular.
\end{proof}

Now we are ready to prove Theorem~\ref{thm:mcover_path_poset}. 

\begin{proof}[Proof of Theorem~\ref{thm:mcover_path_poset}]
	$(a)\Rightarrow(b)$: If $\PP$ is a singleton, then $\Pm{\PP}{m}$ is also a singleton and thus clearly $\tpt$-free.  Now suppose that $\PP\cong\PP_{k,l;\pf}$ for some $k,l\geq 0$ and some northeast path $\pf$.  Let $H$ denote the Hasse diagram of $\PP$ with $\hat{0}$ removed.  We proceed by induction on the length of $\PP$.  If $\ell(\PP)=1$, then $\PP\cong\PP_{0,0;\emptyset}$, hence it is a $2$-chain and thus clearly $\tpt$-free.  Now suppose that $\ell(\PP)=n$ and the claim is true for all such posets of length $<n$.  In view of Proposition~\ref{prop:path_poset_char}, it follows that $H$ satisfies Condition~\eqref{eq:condition}.  Thus there is a unique subtree of $\hat{1}$ with more than one element, say $H_{0}$, and possibly some other one-element subtrees of $\hat{1}$.  Let $z_{0}$ be the root of $H_{0}$.  Again by Proposition~\ref{prop:path_poset_char}, the interval $[\hat{0},z_{0}]$ is isomorphic to some $\PP_{k,l;\bar{\pf}}$, and by induction assumption it is $\tpt$-free.  Thus $\PP$ itself is $\tpt$-
free.
	
	\smallskip
	
	$(b)\Rightarrow(a)$: Let $H$ denote the Hasse diagram of $\PP$ with $\hat{0}$ removed.  Since $\Pm{\PP}{m}$ is a lattice, 
	Theorem~\ref{thm:mcover_lattice} implies that $H$ is a tree rooted at $\hat{1}$.  
	We proceed by induction on the length of $\PP$.  If $\ell(P)=0$, then $\PP$ is a singleton.  
	If $\ell(\PP)=1$, then $\PP\cong\PP_{0,0;\emptyset}$, and if $\ell(\PP)=2$, 
	then $\PP\cong\PP_{1,l;\emptyset}$ for $l\geq 0$, and we are done.  Now suppose that $\ell(\PP)=n>2$, 
	and the claim is true for all posets with length $<n$.  Let $x,y\in P$ with $\hat{0}<x<y<\hat{1}$.  
	(Such elements exist, since $\ell(\PP)>2$.)  Denote by $H_{0}$ the subtree of $\hat{1}$ that contains both $x$ and $y$, 
	and let $H_{1},H_{2},\ldots,H_{s}$ denote the other (non-empty) subtrees of $\hat{1}$.  
	By induction assumption and Proposition~\ref{prop:path_poset_char}, we conlude that $H_{0}$ satisfies 
	Condition~\eqref{eq:condition}.  If there is some $j\in\{1,2,\ldots,s\}$ such that $H_{j}$ contains two or more elements, 
	say $x'$ and $y'$, then---since $H$ is a tree---there cannot exist an element $z\in H_{0}$ satisfying $x'<z$ or $z<y'$.  In particular, $x,y\not\leq x',y'$ and $x',y'\not\leq x,y$, which is a contradiction to $\PP$ being $\tpt$-free.  Hence the cardinality of each subtree $H_{j}$ for $j\in\{1,2,\ldots,s\}$ is one.  This means, however, that $H$ satisfies Condition~\eqref{eq:condition}, and Proposition~\ref{prop:path_poset_char} implies that $\PP=\PP_{k,l;\pf}$ for some $k,l\geq 0$ and some northeast path $\pf$.
	
	\smallskip 
	
	$(a)\Rightarrow(c)$: If $\PP$ is a singleton, then so is $\Pm{\PP}{m}$, and the result is trivial.  Otherwise, in view of Remarks~\ref{rem:exclude_trivial_paths} and \ref{rem:exclude_trivial}, it suffices to consider fixed integers $k>0$ and $l\geq 0$ as well as northeast paths starting with a north-step.  We proceed by induction on the length of the northeast path $\pf=w_{1}w_{2}\cdots w_{n}$.  If $n=0$, then the result follows from Lemma~\ref{lem:induction_base}.  Now assume that $n>0$, and the statement is true for all northeast paths of length $<n$.  Let $\bar{\pf}$ be the subpath of $\pf$ consisting of the first $n-1$ steps of $\pf$, and fix some $m>1$.  Further, define $\PP=\PP_{k,l;\pf}$ and $\bar{\PP}=\PP_{k,l;\bar{\pf}}$.  By induction assumption $\bar{\PP}$ is left-modular, so fix some left-modular chain denoted by $\bar{C}$.  Let $\{s_{1},s_{2},\ldots,s_{t}\}$ denote the indices of the north-steps of $\bar{\pf}$ and let $w_{s_{0}}=c_{k+1}$.  Moreover, let $\pb,\qb\in\Pm{P}{m}$ with 
	$\pb\lessdot\qb$.  We distinguish two cases.
	
	\smallskip
		
	(i) $\pf$ ends with an east-step, \ie $w_{n}=E$.  The greatest element of $\bar{\PP}$ is the same as the greatest element of $\PP$, namely $w_{s_{t}}$.  Further, Proposition~\ref{prop:mcover_cardinality} implies $\ell\bigl(\Pm{\PP}{m}\bigr)=\ell\bigl(\Pm{\bar{\PP}}{m}\bigr)$, and $C=\bar{C}$ will be the candidate for the left-modular chain of $\PP$.  The elements of $C$ are essentially of one of the following forms: for some $j\in\{1,2,\ldots,m\}$, let $\xb=\bigl(c_{i}^{m-j},c_{i+1}^{j}\bigr)$, where $i\in\{0,1,\ldots,k\}$, and let $\xb'=\bigl(w_{s_{i}}^{m-j},w_{s_{i+1}}^{j}\bigr)$, where $i\in\{0,1,\ldots,t-1\}$.
	
	Let $\qb=\bigl(\hat{0}^{a_{0}},w_{n}^{a_{1}},w_{s_{t}}^{a_{2}}\bigr)$ with $a_{0}+a_{1}+a_{2}=m$.  Then, $\qb$ has two possible lower covers, namely $\pb_{1}=\bigl(\hat{0}^{a_{0}+1},w_{n}^{a_{1}-1},w_{s_{t}}^{a_{2}}\bigr)$ and $\pb_{2}=\bigl(\hat{0}^{a_{0}},w_{n}^{a_{1}+1},w_{s_{t}}^{a_{2}-1}\bigr)$, and we notice that these are the only cases that we need to consider.  (If $\pb$ and $\qb$ do not contain $w_{n}$, then $\xb$ and $\xb'$ belong to $\bar{\PP}$ and they satisfy Condition~(ii) in Theorem~\ref{thm:left_modular} by induction hypothesis.)  We obtain
	\begin{displaymath}
		\xb\wedge\pb_{1}=\left.\begin{cases}\bigl(\hat{0}^{m-a_{2}},c_{i}^{a_{2}-j},c_{i+1}^{j}\bigr), & \text{if}\;j<a_{2},\\
			\bigl(\hat{0}^{m-a_{2}},c_{i+1}^{a_{2}}\bigr), & \text{if}\;a_{2}\leq j,\end{cases}\right\}
		=\xb\wedge\qb,
	\end{displaymath}
	and 
	\begin{align*}
		\xb\vee\pb_{1} & = \begin{cases}
			\bigl(\hat{0}^{a_{0}+1},w_{n}^{m-a_{0}-1-\max\{a_{2},j\}},w_{s_{t}}^{\max\{a_{2},j\}}\bigr), & \text{if}\;i=0,\\
			\bigl(w_{s_{t-1}}^{a_{0}+1},w_{s_{t}}^{m-a_{0}-1}\bigr), & \text{if}\;i>0
		\end{cases}\\
		& \neq \begin{cases}
			\bigl(\hat{0}^{a_{0}},w_{n}^{m-a_{0}-\max\{a_{2},j\}},w_{s_{t}}^{\max\{a_{2},j\}}\bigr), 
			  & \text{if}\;i=0,\\
			\bigl(w_{s_{t-1}}^{a_{0}},w_{s_{t}}^{m-a_{0}}\bigr), & \text{if}\;i>0,
		\end{cases}\\
		& = \xb\vee\qb,
	\end{align*}
	as desired.  Next $\pb_{2}$ exists only if $a_{2}>0$, and we have
	\begin{displaymath}
		\xb\wedge\pb_{2}=\begin{cases}
			\bigl(\hat{0}^{m-a_{2}+1},c_{i}^{a_{2}-1-j},c_{i+1}^{j}\bigr), & \text{if}\;j<a_{2},\\
			\bigl(\hat{0}^{m-a_{2}+1},c_{i+1}^{a_{2}-1}\bigr), & \text{if}\;j\geq a_{2},
		\end{cases}
	\end{displaymath}
	and we notice that $\xb\wedge\pb_{2}\neq\xb\wedge\qb$ only if $i>0$ or $j\geq a_{2}$.  On the other hand, we have 
	\begin{displaymath}
		\xb\vee\pb_{2} = \begin{cases}
			\bigl(\hat{0}^{a_{0}},w_{n}^{m-a_{0}-\max\{a_{2}-1,j\}},w_{s_{t}}^{\max\{a_{2}-1,j\}}\bigr), & \text{if}\;i=0,\\
			\bigl(w_{s_{t-1}}^{a_{0}},w_{s_{t}}^{m-a_{0}}\bigr), & \text{if}\;i>0,
		\end{cases}
	\end{displaymath}
	and we notice that $\xb\vee\pb_{2}=\xb\vee\qb$ only if $i>0$ or $j\geq a_{2}$, as desired.  Thus $\xb$ satisfies Condition~(ii) in Theorem~\ref{thm:left_modular}. 

	Now by construction, we have $w_{n}\leq w_{s_{i}}$ if and only if $i=t$.  We have
	\begin{displaymath}
		\xb'\wedge\qb=\left.\begin{cases}
			\bigl(\hat{0}^{m-a_{2}},w_{s_{i+1}}^{a_{2}}\bigr), & \text{if}\;a_{2}\leq j\;\text{and}\;i<t-1,\\
			\bigl(\hat{0}^{m-j},w_{n}^{j-a_{2}},w_{s_{t}}^{a_{2}}\bigr), & \text{if}\;a_{2}\leq j\;\text{and}\;i=t-1,\\
			\bigl(\hat{0}^{m-a_{2}},w_{s_{i}}^{a_{2}-j},w_{s_{i+1}}^{j}\bigr), & \text{if}\;j<a_{2},
		\end{cases}\right\}=\xb'\wedge\pb_{1},
	\end{displaymath}
	and 
	\begin{displaymath}
		\xb'\vee\pb_{1} = \bigl(w_{s_{t-1}}^{a_{0}+1},\hat{1}^{m-a_{0}-1}\bigr) \neq \bigl(w_{s_{t-1}}^{a_{0}},\hat{1}^{m-a_{0}}\bigr) = \xb'\vee\qb,
	\end{displaymath}
	as desired.  Furthermore, we have 
	\begin{displaymath}
		\xb'\wedge\pb_{2}=\begin{cases}
			\bigl(\hat{0}^{m-a_{2}+1},w_{s_{i+1}}^{a_{2}-1}\bigr), & \text{if}\;a_{2}\leq j\;\text{and}\;i<t-1,\\
			\bigl(\hat{0}^{m-j},w_{n}^{j-a_{2}+1},w_{s_{t}}^{a_{2}-1}\bigr), & \text{if}\;a_{2}\leq j\;\text{and}\;i=t-1,\\
			\bigl(\hat{0}^{m-a_{2}+1},w_{s_{i}}^{a_{2}-1-j},w_{s_{i+1}}^{j}\bigr), & \text{if}\;j<a_{2},
		\end{cases},
	\end{displaymath}
	and $\xb'\vee\pb_{2} = \bigl(w_{s_{t-1}}^{a_{0}},w_{s_{t}}^{m-a_{0}}\bigr)$, which implies $\xb'\wedge\pb_{2}\neq\xb'\wedge\qb$ and $\xb'\vee\pb_{2}=\xb'\vee\qb$.  Thus $\xb'$ satisfies Condition~(ii) in Theorem~\ref{thm:left_modular} as well, and $C$ is a left-modular chain of $\PP$.

	(ii) $\pf$ ends with a north-step, \ie $w_{n}=N$.  Then, the greatest element of $\bar{\PP}$ is $w_{s_{t}}$ and the greatest element of $\PP$ is $w_{n}$.  Further, Proposition~\ref{prop:mcover_cardinality} implies $\ell\bigl(\Pm{\PP}{m}\bigr)=\ell\bigl(\Pm{\bar{\PP}}{m}\bigr)+m$, and the candidate for the left-modular chain of $\PP$ will be the chain $C$ which is constructed from $\bar{C}$ by appending the chain $\bigl(w_{s_{t}}^{m-1},w_{n}\bigr)\lessdot\bigl(w_{s_{t}}^{m-2},w_{n}^{2}\bigr)\lessdot\cdots\lessdot\bigl(w_{n}^{m}\bigr)$.  The elements of $C$ are essentially of one of the following forms: for some $j\in\{1,2,\ldots,m\}$, let $\xb=\bigl(c_{i}^{m-j},c_{i+1}^{j}\bigr)$, where $i\in\{0,1,\ldots,k\}$, let $\xb'=\bigl(w_{s_{i}}^{m-j},w_{s_{i+1}}^{j}\bigr)$, where $i\in\{0,1,\ldots,t-1\}$, and let $\xb''=\bigl(w_{s_{t}}^{m-j},w_{n}^{j}\bigr)$.	
	
	If $\pb,\qb\in\Pm{\bar{P}}{m}$, then $\xb$ and $\xb'$ satisfy Condition~(ii) of Theorem~\ref{thm:left_modular} by induction assumption.  We notice by construction, that $\pb\lessdot\qb\leq\xb''$, which implies $\pb\vee\xb''=\xb''=\qb\vee\xb''$ and $\pb\wedge\xb''=\pb\neq\qb=\qb\wedge\xb''$ as desired.  So let $\qb=\bigl(\hat{0}^{a_{0}},w_{s_{t}}^{a_{1}},w_{n}^{a_{2}}\bigr)$ for $a_{0}+a_{1}+a_{2}=m$.  Again, we have two possible lower covers, namely $\pb_{1}=\bigl(\hat{0}^{a_{0}+1},w_{s_{t}}^{a_{1}-1},w_{n}^{a_{2}}\bigr)$ and $\pb_{2}=\bigl(\hat{0}^{a_{0}},w_{s_{t}}^{a_{1}+1},w_{n}^{a_{2}-1}\bigr)$.  We have
	\begin{displaymath}
		\xb\wedge\pb_{1}=\bigl(\hat{0}^{a_{0}+1},c_{i}^{m-a_{0}-1-j},c_{i+1}^{j}\bigr)\neq 
			\bigl(\hat{0}^{a_{0}},c_{i}^{m-a_{0}-j},c_{i+1}^{j}\bigr)=\xb\wedge\qb,
	\end{displaymath}
	and 
	\begin{displaymath}
		\xb\vee\pb_{1}=\bigl(w_{s_{t}}^{m-a_{2}},w_{n}^{a_{2}}\bigr)=\xb\vee\qb,
	\end{displaymath}
	as desired.  Moreover, we have 
	\begin{displaymath}
		\xb\wedge\pb_{2}=\bigl(\hat{0}^{a_{0}},c_{i}^{m-a_{0}-j},c_{i+1}^{j}\bigr)=\xb\wedge\qb,
	\end{displaymath}
	and 
	\begin{displaymath}
		\xb\vee\pb_{2}=\bigl(w_{s_{t}}^{m-a_{2}+1},w_{n}^{a_{2}-1}\bigr)\neq\bigl(w_{s_{t}}^{m-a_{2}},w_{n}^{a_{2}}\bigr)=\xb\vee\qb,
	\end{displaymath}
	as desired.  Hence $\xb$ satisfies Condition~(ii) in Theorem~\ref{thm:left_modular}. The reasoning for $\xb'$ and $\xb''$ is exactly analogous.  It follows that $C$ is a left-modular saturated maximal chain of $\PP$ which concludes this part of the proof.
	
	$(c)\Rightarrow(d)$: This follows from Theorem~\ref{thm:left_modular_el}.
	
	$(d)\Rightarrow(a)$: First recall that $\Pm{\PP}{m}$ is a lattice, and hence by Theorem~\ref{thm:mcover_lattice} the Hasse diagram of $\PP$ with $\hat{0}$ removed is a tree rooted at $\hat{1}$.  Assume that $\PP$ is not isomorphic to $\PP_{k,l;\pf}$ for $k,l\geq 0$ and some northeast path $\pf$.  Then the implication $(b)\Rightarrow(a)$ (which is already proven) implies that $\PP$ is not $\tpt$-free, and hence that we can find elements $x,y,x',y'\in P\setminus\{\hat{0},\hat{1}\}$ with $x<y$ and $x'<y'$, as well as $x,y\not\leq x',y'$ and $x',y'\not\leq x,y$.  Theorem~\ref{thm:mcover_lattice} implies now that $x\wedge x'=x\wedge y'=y\wedge x'=y\wedge y'=\hat{0}$.  Let $z=y\vee y'$, $z'=x\vee x'$, and suppose that $z'<z$.  Then, however, $x$ is a lower bound for $y$ and $z'$, and $x'$ is a lower bound for $y'$ and $z'$, and it follows that $x\leq y\wedge z'$ and $x'\leq y'\wedge z'$.  If $y\wedge z'=\hat{0}$, then $x=\hat{0}$, contradicting the choice of $x$.  If $y\wedge z'=z'$, then $x'\leq z'\leq y$, 
contradicting the choice of $x'$ and $y$.  We have the analogous results for $y'\wedge z'$.  It remains only the case that $y\wedge z'=y$ and $y'\wedge z'=y'$.  In this case, however, we obtain $z=y\vee y'\leq z'<z$, which is a contradiction.  Hence it follows that $x\vee x'=z$, and we focus on the interval $I=[\hat{0},z]$, which looks as depicted in Figure~\ref{fig:2+2_interval}.  (Note that Theorem~\ref{thm:mcover_lattice} implies that we can choose the elements $x,x',y,y'$ in such a way that $\hat{0}\lessdot x,x'$, and $y,y'\lessdot z$.)  Now the proper part of $I$ contains at least two disjoint posets, each having length $\geq 1$.  It follows then immediately  that $I$ cannot be EL-shellable.  Since $I$ is an interval of $\PP$, and $\PP$ is an interval of $\Pm{\PP}{m}$, it follows that $\Pm{\PP}{m}$ is not EL-shellable, which concludes the proof.
\end{proof}

\begin{figure}
	\centering
	\begin{tikzpicture}
		\def\x{1};
		\def\y{1};
		\draw(3*\x,1*\y) node(o){$\hat{0}$};
		\draw(2*\x,2*\y) node(x1){$x$};
		\draw(4*\x,2*\y) node(x2){$x'$};
		\draw(2*\x,4*\y) node(y1){$y$};
		\draw(4*\x,4*\y) node(y2){$y'$};
		\draw(3*\x,5*\y) node(z){$z$};
		\draw(o) -- (x1);
		\draw(o) -- (x2);
		\draw(x1) .. controls (1.8*\x,2.75*\y) and (2.2*\x,3.25*\y) .. (y1);
		\draw(x2) .. controls (4.2*\x,2.75*\y) and (3.8*\x,3.25*\y) .. (y2);
		\draw(y1) -- (z);
		\draw(y2) -- (z);
		\begin{pgfonlayer}{background}
			\fill[white!60!gray](3*\x,1*\y) 
			  -- (2*\x,2*\y) .. controls (1.8*\x,2.75*\y) and (2.2*\x,3.25*\y) .. (2*\x,4*\y) 
			  -- (3*\x,5*\y) .. controls (.5*\x,4*\y) and (.5*\x,2*\y) .. (3*\x,1*\y);
			\fill[white!60!gray](3*\x,1*\y) 
			  -- (4*\x,2*\y) .. controls (4.2*\x,2.75*\y) and (3.8*\x,3.25*\y) .. (4*\x,4*\y) 
			  -- (3*\x,5*\y) .. controls (5.5*\x,4*\y) and (5.5*\x,2*\y) .. (3*\x,1*\y);
		\end{pgfonlayer}
	\end{tikzpicture}
	\caption{Illustration of a case in the proof of Theorem~\ref{thm:mcover_path_poset}.  Straight edges are coverings, curved edges are chains.  The gray areas may contain chains, the white (interior) area does not contain any chains.}
	\label{fig:2+2_interval}
\end{figure}
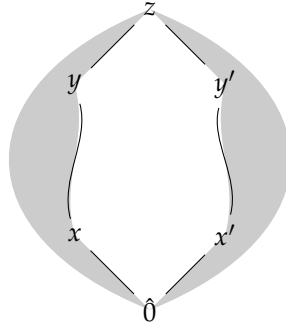

\begin{proof}[Proof of Theorem~\ref{thm:mcover_trim}]
	In view of Theorem~\ref{thm:mcover_path_poset}, it follows that $\Pm{\PP}{m}$ is left-modular if and only if $\PP\cong\PP_{k,l;\pf}$.  Now let us investigate the extremality of $\PP_{k,l;\pf}$.  For the moment, let us assume that $l>0$.  We can quickly check that $\ell(\PP_{k,l})=k+1$ and that $\bigl\lvert\jj(\PP_{k,l})\bigr\rvert=\bigl\lvert\mm(\PP_{k,l})\bigr\rvert=k+l$.  Now let $\pf$ be a northeast path, consisting of $n$ steps, and let $\bar{\pf}$ be the subpath of $\pf$ containing all those steps but the last one.  In view of Remarks~\ref{rem:exclude_trivial_paths} and \ref{rem:exclude_trivial}, we can assume that $\pf$ starts with a north-step.  We distinguish two cases.
	
	(i) $\pf$ ends with an east-step.  By construction we have 
	\begin{align*}
		\ell(\PP_{k,l;\pf}) & = \ell(\PP_{k,l;\bar{\pf}}),\\
		\bigl\lvert\mm(\PP_{k,l;\pf})\bigr\rvert & = \bigl\lvert\mm(\PP_{k,l;\bar{\pf}})\bigr\rvert+1,\\
		\bigl\lvert\jj(\PP_{k,l;\pf})\bigr\rvert & = \begin{cases}
		  \bigl\lvert\jj(\PP_{k,l;\bar{\pf}})\bigr\rvert, & \mbox{if}\;\bar{\pf}\;\mbox{ends with a north-step},\\
		  \bigl\lvert\jj(\PP_{k,l;\bar{\pf}})\bigr\rvert+1, & \mbox{if}\;\bar{\pf}\;\mbox{ends with an east-step},\\
		\end{cases}
	\end{align*}
	
	(ii) $\pf$ ends with a north-step.  By construction, we have $\ell(\PP_{k,l;\pf})=\ell(\PP_{k,l;\bar{\pf}})+1$, and $\bigl\lvert\jj(\PP_{k,l;\pf})\bigr\rvert=\bigl\lvert\jj(\PP_{k,l;\bar{\pf}})\bigr\rvert+1$ and $\bigl\lvert\mm(\PP_{k,l;\pf})\bigr\rvert=\bigl\lvert\mm(\PP_{k,l;\bar{\pf}})\bigr\rvert+1$.
	
	Thus if $\pf$ is a northeast path that contains a subpath of the form $NE$, then we have $\bigl\lvert\jj(\PP_{k,l;\pf})\bigr\rvert\neq\bigl\lvert\mm(\PP_{k,l;\pf})\bigr\rvert$.  Since we have assumed that $\pf$ starts with a north-step, the only remaining candidates for trim $m$-cover lattices are path posets of the form $\PP_{k,l;N^{s}}$, and in this case we obtain $\bigl\lvert\jj(\PP_{k,l;\pf})\bigr\rvert=\bigl\lvert\mm(\PP_{k,l;\pf})\bigr\rvert=k+l+s$, and $\ell(\PP_{k,l;\pf})=k+1+s$.  In view of Proposition~\ref{prop:mcover_irreducibles}, we conclude that $\Pm{\PP}{m}$ is trim if and only if $l=1$, and $\pf$ consists only of north-steps.
	
	Now suppose that $l=0$.  Then, $\PP_{k,0}$ is a $(k+2)$-chain, and hence trivially trim.  If $\pf$ starts with an east-step, then $\PP$ is not bounded, which contradicts the assumption of the theorem.  So suppose that $\pf$ starts with $s$ north-steps followed by an east-step.  Then, let $\bar{\pf}$ be the subpath of $\pf$ starting from the $(s+1)$-st step, and replace $\PP_{k,0;\pf}$ by $\PP_{k+s,1;\bar{\pf}}$, and we are back in the case $l>0$.  If $\pf$ does not contain an east-step, then $\PP_{k,0;\pf}$ is a $(k+2+n)$-chain, where $n$ is the length of $\pf$, and hence trivially trim.
\end{proof}

Recall that a closed interval $[p,q]$ in a lattice is called nuclear if $q$ is the join of atoms of $[p,q]$. We have the following result.

\begin{theorem}[\cite{thomas06analogue}*{Theorem~7}]\label{thm:trim_topology}
	Let $\PP$ be a finite lattice.  If $\PP$ is trim and nuclear, then its order complex is homotopic to a sphere, whose dimension is $2$ less than the number of atoms of $\PP$.  If $\PP$ is trim but not nuclear, then its order complex is contractible. 
\end{theorem}

\begin{proposition}\label{prop:mcover_trim_mobius}
	Let $\PP$ be a bounded poset such that $\Pm{\PP}{m}$ is a trim lattice for all $m>0$.  If $\mu$ denotes the M{\"o}bius function of $\Pm{\PP}{m}$, then for $\pb,\qb\in\Pm{P}{m}$ with $\pb\leq\qb$, we have
	\begin{displaymath}
		\mu(\pb,\qb) = \begin{cases}
			1, & \text{if}\;[\pb,\qb]\;\text{is nuclear and has two atoms},\\
			-1, & \text{if}\;\qb\;\text{covers}\;\pb,\quad\text{or}\\
			0, & \text{otherwise}.
		\end{cases}
	\end{displaymath}
\end{proposition}
\begin{proof}
	In view of Theorem~\ref{thm:mcover_trim}, it follows that $\PP\cong\PP_{k,1;N^{s}}$ for some $k,s\geq 0$.  First suppose that $s=0$, then Proposition~\ref{prop:cover_statistic} implies that every element in $\Pm{\PP}{m}$ has at most two upper covers.  If $s>0$, then \eqref{eq:cover_north} implies the same.  Now, let $\pb,\qb\in\Pm{\PP}{m}$ with $\pb\leq\qb$.  If $[\pb,\qb]$ is nuclear, then we have either $\pb\lessdot\qb$ or $\qb$ is the join of the two atoms in $[\pb,\qb]$, and we obtain $\mu(\pb,\qb)=\pm 1$ as desired.  If $[\pb,\qb]$ is not nuclear, then Theorem~\ref{thm:trim_topology} implies that the associated order complex is contractible, and hence has reduced Euler characteristic $0$.  Proposition~3.8.6 in \cite{stanley97enumerative} implies that the M{\"o}bius function of an interval in a poset takes the same value as the reduced Euler characteristic of the associated order complex, and the result follows.
\end{proof}

\section{Application: The $m$-Cover Posets of the Tamari Lattices}
  \label{sec:application}
In this section we investigate the $m$-cover poset of the Tamari lattices $\mtam{n}{1}$. In particular, 
we use the results from the previous section to relate the $m$-cover posets of $\mtam{n}{1}$ to the $m$-Tamari lattices $\mtam{n}{m}$. We start by recalling some terminology and fixing some notation. Consecutively, we introduce the strip decomposition 
of $m$-Dyck paths and prove Theorem~\ref{thm:mtamari}. We complete this section by listing a few properties of the strip decomposition, and by stating a conjecture that explicitly describes 
the lattice $\mtam{n}{m}$ in terms of $m$-tuples of Dyck paths.

\subsection{The $m$-Dyck Paths and the $m$-Tamari Lattices}
  \label{sec:mdyck_paths} 
\subsubsection{The $m$-Dyck Paths.} For $m,n\in\mathbb{N}$, we say that an \alert{$m$-Dyck path of length $(m+1)n$} is a lattice path from $(0,0)$ to $(mn,n)$ that stays weakly above the line $x=my$, and that consists only of north-steps (steps of the form $(0,1)$) and east-steps (steps of the form $(1,0)$).  Let $D_n^{(m)}$ denote the set of all $m$-Dyck paths of length $(m+1)n$.  The cardinality of $D_n^{(m)}$ is given by the \alert{Fu{\ss}-Catalan number} 
\begin{equation}
  \label{eq:fuss_catalan}\text{Cat}^{(m)}(n)=\frac{1}{mn+1}\binom{(m+1)n}{n},
\end{equation}
see for instance \cite{duchon00enumeration}*{Section~6}.  If $m=1$, then we usually write $D_n$ instead of $D_n^{(1)}$.

Let $\pf$ be an $m$-Dyck path $\pf\in D_{n}^{(m)}$.  Then, $\pf$ can be encoded by its \alert{step sequence} $\uu_{\pf}=(u_{1},u_{2},\ldots,u_{n})$, which satisfies
\begin{align}
  \label{eq:mdyck_1} u_{1} & \leq u_{2}\leq\cdots\leq u_{n},\quad\text{and}\\
  \label{eq:mdyck_2} u_{k} & \leq m(k-1),\quad\text{for all}\;k\in\{1,2,\ldots,n\}.
\end{align}
In other words, the entry $u_{k}$ represents at which $x$-coordinate the $k$-th north-step takes place.  Equivalently, $\pf$ can be encoded by its \alert{height sequence} $\hh_{\pf}=(h_{1},h_{2},\ldots,h_{mn})$, which satisfies
\begin{align}
  \label{eq:height_1} h_{1} & \leq h_{2}\leq\cdots\leq h_{mn}\leq n,\quad\text{and}\\
  \label{eq:height_2} h_{k} & \geq\bigl\lceil\tfrac{k}{m}\bigr\rceil,\quad\text{for all}\;k\in\{1,2,\ldots,mn\}.
\end{align}
In other words, the entry $h_{k}$ represents which height the path $\pf$ has at $x$-coordinate $k-\tfrac{1}{2}$.  It is straightforward to prove the following lemma.

\begin{lemma}\label{lem:height_step_conversion}
	Let $\pf\in D_n^{(m)}$ with step sequence $\uu_{\pf}=(u_{1},u_{2},\ldots,u_{n})$ and height sequence $\hh_{\pf}=(h_{1},h_{2},\ldots,h_{mn})$.  If we set $h_{0}=0$, then we have 
	\begin{align*}
		u_{h_{k}+1} & =u_{h_{k}+2}=\cdots=u_{h_{k+1}}=k,\;\text{for all}\;k\in\{0,1,\ldots,mn\}\;\text{with}\;h_{k}<h_{k+1},\quad\text{and}\\
		h_{k} & =\max\bigl\{j\in\{1,2,\ldots,n\}\mid u_{j}<k\bigr\},\;\text{for all}\;k\in\{1,2,\ldots,mn\}.
	\end{align*}
\end{lemma}

\begin{example}\label{ex:dyck35_1}
	\begin{figure}
		\centering
		\begin{tikzpicture}
			\draw(0,0) node{
				\begin{tikzpicture}[scale=.8]\tiny
					\draw[white!50!gray](0,0) grid[step=.4] (6,2);
					\draw(0,0) -- (6,2);
					\draw(0,0) -- (0,.4) -- (.8,.4) -- (.8,1.2) -- (3.2,1.2) -- (3.2,1.6) -- (4,1.6) 
					  -- (4,2) -- (6,2);
					\draw[white!50!gray,->](6,0) -- (6.5,0);
					\draw[white!50!gray,->](0,2) -- (0,2.5);
					\draw[white!50!gray](6.4,-.25) node{$x$};
					\draw[white!50!gray](-.25,2.4) node{$y$};
					\foreach\i in {0,1,2,3,4,5,6,7,8,9,10,11,12,13,14,15} {
						\draw(.4*\i,-.25) node[white!50!gray]{$\i$};
					}
					\foreach\i in {0,1,2,3,4,5} {
						\draw(-.25,.4*\i) node[white!50!gray]{$\i$};
					}
					\draw(3,2.25) node{$\pf$};
				\end{tikzpicture}};
		  \end{tikzpicture}
		\caption{A $3$-Dyck path $\pf$ of length $20$.}
		\label{fig:dyck35}
	\end{figure}
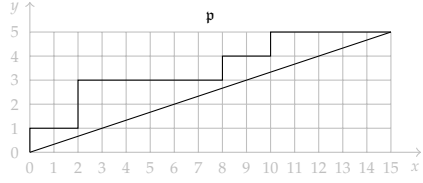
	
	Let $\pf\in D_5^{(3)}$ be the path shown in Figure~\ref{fig:dyck35}.  Its step sequence is $\uu_{\pf}=(0,2,2,8,10)$, and its height sequence is $\hh_{\pf}=(1,1,3,3,3,3,3,3,4,4,5,5,5,5,5)$. 
\end{example}

\subsubsection{The $m$-Tamari Lattices.} Let $\pf\in D_n^{(m)}$ with step sequence $\uu_{\pf}=(u_{1},u_{2},\ldots,u_{n})$.  For $i\in\{1,2,\ldots,n\}$, we say that the \alert{primitive subsequence of $\uu_{\pf}$ at position $i$} is the unique subsequence $(u_{i},u_{i+1},\ldots,u_{k})$ that satisfies 
\begin{align}
	\label{eq:prim_1} & u_{j}-u_{i} < m(j-i),\quad\text{for}\;j\in\{i+1,i+2,\ldots,k\},\quad\text{and}\\
	\label{eq:prim_2} & \text{either}\quad k=n,\quad\text{or}\quad u_{k+1}-u_{i}\geq m(k+1-i).
\end{align}
Bergeron and Pr{\'e}ville-Ratelle define in \cite{bergeron12higher}*{Section~5} a partial order on $D_n^{(m)}$ as follows: let $\pf,\pf'\in D_n^{(m)}$ such that $\uu_{\pf}=(u_{1},u_{2},\ldots,u_{n})$ is the step sequence of $\pf$ and $\uu_{\pf'}$ denotes the step sequence of $\pf'$.  Define 
\begin{equation}\label{eq:rotation_mtamari}
	\pf\lessdot_{\text{rot}}\pf'\quad\text{if and only if}\quad\uu_{\pf'}=(u_{1},\ldots,u_{i-1},u_{i}-1,\ldots,u_{k}-1,u_{k+1},\ldots,u_{n}),
\end{equation}
for some $i\in\{2,3,\ldots,n\}$ satisfying $u_{i-1}<u_{i}$ such that $\{u_{i},u_{i+1},\ldots,u_{k}\}$ is the primitive subsequence of $\uu_{\pf}$ at position $i$.  Let $\leq_{\text{rot}}$ denote the transitive and reflexive closure of $\lessdot_{\text{rot}}$, and call this partial order the \alert{rotation order} on $D_n^{(m)}$.  The name rotation order comes from the fact that we can also describe this partial order as follows: let $E$ be an east-step of $\pf$ that is followed by a north-step $N$, and let $\bar{\pf}$ be the unique shortest nontrivial $m$-Dyck path of length $(m+1)n'<(m+1)n$ starting with $N$.  Denote by $\varrho_{E}(\pf)$ the $m$-Dyck path of length $(m+1)n$ that is constructed from $\pf$ by exchanging $E$ and $\bar{\pf}$.  We have $\pf\lessdot_{\text{rot}}\pf'$ if and only if $\pf'=\varrho_{E}(\pf)$ for a suitable east-step $E$ of $\pf$.  See Figure~\ref{fig:rotation_illustration} for an illustration. 

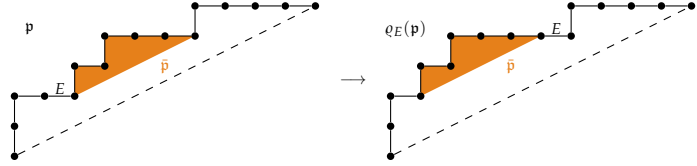
\begin{figure}
	\centering
	\begin{tikzpicture}\tiny
		\def\d{5};
		\draw(0,0) node[circle,fill,scale=.5](n1){};
		\draw(0,.4) node[circle,fill,scale=.5](n2){};
		\draw(0,.8) node[circle,fill,scale=.5](n3){};
		\draw(.4,.8) node[circle,fill,scale=.5](n4){};
		\draw(.8,.8) node[circle,fill,scale=.5](n5){};
		\draw(.8,1.2) node[circle,fill,scale=.5](n6){};
		\draw(1.2,1.2) node[circle,fill,scale=.5](n7){};
		\draw(1.2,1.6) node[circle,fill,scale=.5](n8){};
		\draw(1.6,1.6) node[circle,fill,scale=.5](n9){};
		\draw(2,1.6) node[circle,fill,scale=.5](n10){};
		\draw(2.4,1.6) node[circle,fill,scale=.5](n11){};
		\draw(2.4,2) node[circle,fill,scale=.5](n12){};
		\draw(2.8,2) node[circle,fill,scale=.5](n13){};
		\draw(3.2,2) node[circle,fill,scale=.5](n14){};
		\draw(3.6,2) node[circle,fill,scale=.5](n15){};
		\draw(4,2) node[circle,fill,scale=.5](n16){};
		\draw(n1) -- (n2) -- (n3) -- (n4) -- (n5) -- (n6) -- (n7) -- (n8) -- (n9) -- (n10) -- (n11) -- (n12) -- (n13) -- (n14) -- (n15) -- (n16);
		\draw[dashed](n1) -- (n16);
		\draw(.6,.9) node{$E$};
		\begin{pgfonlayer}{background}
			\fill[\cOne](.8,.8) -- (.8,1.2) -- (1.2,1.2) -- (1.2,1.6) -- (2.4,1.6) -- cycle;
		\end{pgfonlayer}
		\draw(2,1.2) node[\cOne]{$\bar{\pf}$};
		\draw(.2,1.7) node{$\pf$};
		\draw(4.5,1) node{$\longrightarrow$};
		\draw(0+\d,0) node[circle,fill,scale=.5](m1){};
		\draw(0+\d,.4) node[circle,fill,scale=.5](m2){};
		\draw(0+\d,.8) node[circle,fill,scale=.5](m3){};
		\draw(.4+\d,.8) node[circle,fill,scale=.5](m4){};
		\draw(.4+\d,1.2) node[circle,fill,scale=.5](m5){};
		\draw(.8+\d,1.2) node[circle,fill,scale=.5](m6){};
		\draw(.8+\d,1.6) node[circle,fill,scale=.5](m7){};
		\draw(1.2+\d,1.6) node[circle,fill,scale=.5](m8){};
		\draw(1.6+\d,1.6) node[circle,fill,scale=.5](m9){};
		\draw(2+\d,1.6) node[circle,fill,scale=.5](m10){};
		\draw(2.4+\d,1.6) node[circle,fill,scale=.5](m11){};
		\draw(2.4+\d,2) node[circle,fill,scale=.5](m12){};
		\draw(2.8+\d,2) node[circle,fill,scale=.5](m13){};
		\draw(3.2+\d,2) node[circle,fill,scale=.5](m14){};
		\draw(3.6+\d,2) node[circle,fill,scale=.5](m15){};
		\draw(4+\d,2) node[circle,fill,scale=.5](m16){};
		\draw(m1) -- (m2) -- (m3) -- (m4) -- (m5) -- (m6) -- (m7) -- (m8) -- (m9) -- (m10) -- (m11) -- (m12) -- (m13) -- (m14) -- (m15) -- (m16);
		\draw[dashed](m1) -- (m16);
		\draw(2.2+\d,1.7) node{$E$};
		\begin{pgfonlayer}{background}
			\fill[\cOne](.4+\d,.8) -- (.4+\d,1.2) -- (.8+\d,1.2) -- (.8+\d,1.6) -- (2+\d,1.6) -- cycle;
		\end{pgfonlayer}
		\draw(1.6+\d,1.2) node[\cOne]{$\bar{\pf}$};
		\draw(.2+\d,1.7) node{$\varrho_{E}(\pf)$};
	\end{tikzpicture}
	\caption{The rotation on $m$-Dyck paths.}
	\label{fig:rotation_illustration}
\end{figure}

Proposition~4 in \cite{bousquet11number} states that $\mtam{n}{m}=\bigl(D_n^{(m)},\leq_{\text{rot}}\bigr)$ is an interval in the classical Tamari lattice $\mathcal{T}_{mn}$, and is thus called the \alert{$m$-Tamari lattice of parameter $n$}.  Figure~\ref{fig:tamari23} shows the $2$-Tamari lattice $\mtam{3}{2}$.  We conclude this section with another easy lemma.

\begin{lemma}\label{lem:sequence_order}
	Let $\pf,\pf'\in D_{n}^{(m)}$ with $\pf\leq_{\text{rot}}\pf'$.  If $\uu_{\pf}=(u_{1},u_{2},\ldots,u_{n})$ and $\uu_{\pf'}=(u'_{1},u'_{2},\ldots,u'_{n})$ denote the step sequences of $\pf$ and $\pf'$, respectively, then we have $u_{k}\geq u'_{k}$ for all $k\in\{1,2,\ldots,n\}$.  Moreover, if $\hh_{\pf}=(h_{1},h_{2},\ldots,h_{mn})$ and $\hh_{\pf'}=(h'_{1},h'_{2},\ldots,h'_{mn})$ denote the height sequences of $\pf$ and $\pf'$, then we have $h_{k}\leq h'_{k}$ for all $k\in\{1,2,\ldots,mn\}$.
\end{lemma}
\begin{proof}
	This is straightforward from the definition.
\end{proof}

\begin{figure}
	\centering
		\begin{tikzpicture}\tiny
			\def\x{.8};
			\def\y{.8};
			\draw(3*\x,1*\y) node(v1){\dyckTwoThree{0/2/4}{white}{white}{.4}{-1}};
			\draw(2*\x,2*\y) node(v2){\dyckTwoThree{0/1/4}{white}{white}{.4}{-1}};
			\draw(1*\x,3*\y) node(v3){\dyckTwoThree{0/0/4}{white}{white}{.4}{-1}};
			\draw(3*\x,3*\y) node(v4){\dyckTwoThree{0/1/3}{white}{white}{.4}{-1}};
			\draw(2*\x,4*\y) node(v5){\dyckTwoThree{0/0/3}{white}{white}{.4}{-1}};
			\draw(6*\x,4*\y) node(v6){\dyckTwoThree{0/2/3}{white}{white}{.4}{-1}};
			\draw(3*\x,5*\y) node(v7){\dyckTwoThree{0/0/2}{white}{white}{.4}{-1}};
			\draw(5*\x,5*\y) node(v8){\dyckTwoThree{0/1/2}{white}{white}{.4}{-1}};
			\draw(7*\x,5*\y) node(v9){\dyckTwoThree{0/2/2}{white}{white}{.4}{-1}};
			\draw(4*\x,6*\y) node(v10){\dyckTwoThree{0/0/1}{white}{white}{.4}{-1}};
			\draw(6*\x,6*\y) node(v11){\dyckTwoThree{0/1/1}{white}{white}{.4}{-1}};
			\draw(5*\x,7*\y) node(v12){\dyckTwoThree{0/0/0}{white}{white}{.4}{-1}};
			\draw(v1) -- (v2);
			\draw(v1) -- (v6);
			\draw(v2) -- (v3);
			\draw(v2) -- (v4);
			\draw(v3) -- (v5);
			\draw(v4) -- (v5);
			\draw(v4) -- (v8);
			\draw(v5) -- (v7);
			\draw(v6) -- (v8);
			\draw(v6) -- (v9);
			\draw(v7) -- (v10);
			\draw(v8) -- (v10);
			\draw(v8) -- (v11);
			\draw(v9) -- (v11);
			\draw(v10) -- (v12);
			\draw(v11) -- (v12);
		  \end{tikzpicture}
	\caption{The $2$-Tamari lattice of parameter $3$.}
	\label{fig:tamari23}
\end{figure}
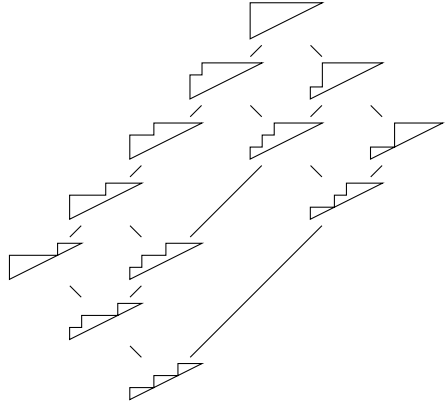

\subsubsection{Irreducible Elements of the $m$-Tamari Lattices}
In this section we characterize the meet- and join-irreducible elements of the lattice $\mtam{n}{m}$.  These characterizations will be useful in the proof of Theorem~\ref{thm:mtamari} in Section~\ref{sec:proof_tamari}.

\begin{proposition}\label{prop:meet_irreducibles_tamari}
	An element $\pf\in D_n^{(m)}$ is meet-irreducible in $\mtam{n}{m}$ if and only if its step sequence $\uu_{\pf}=(u_{1},u_{2},\ldots,u_{n})$ satisfies 
	\begin{equation}
		u_{j}=\begin{cases}\label{eq:mtamari_meet_irreducibles}
			0, & \text{for}\;j\leq i, \\
			a, & \text{for}\;j>i
		\end{cases}
	\end{equation}
	where $a\in\{1,2,\ldots,mi\}$ and $i\in\{1,2,\dots,n-1\}$.
\end{proposition}
\begin{proof}
	Let $\pf\in D_n^{(m)}$ with step sequence $\uu_{\pf}=(u_1,u_2,\dots,u_n)$.  It follows immediately from the definition that the number of upper covers of $\pf$ is precisely the cardinality of the set $\bigl\{i\in\{1,2,\ldots,n-1\}\mid u_{i}<u_{i+1}\bigr\}$.  Hence $\pf\in\mm\bigl(\mtam{n}{m}\bigr)$ if and only if $\uu_{\pf}$ is of the form \eqref{eq:mtamari_meet_irreducibles}.
\end{proof}

\begin{corollary}\label{cor:card_meet_irreducibles_tamari}
	For every $m,n>0$, we have $\Bigl\lvert\mm\bigl(\mtam{n}{m}\bigr)\Bigr\rvert=m\tbinom{n}{2}$.
\end{corollary}
\begin{proof}
	It follows from Proposition~\ref{prop:meet_irreducibles_tamari} that $\pf\in\mm\bigl(\mtam{n}{m}\bigr)$ if and only if $\uu_{\pf}=(0,0,\ldots,0,a,a,\ldots,a)$, where the first non-zero entry occurs in the $(i+1)$-st position, and where $1\leq a\leq mi$.  Thus 
	\begin{displaymath}
		\Bigl\lvert\mm\bigl(\mtam{n}{m}\bigr)\Bigr\rvert=\sum_{i=1}^{n-1}{mi}=m\cdot\frac{n(n-1)}{2}=m\binom{n}{2}.
	\end{displaymath}
\end{proof}

\begin{proposition}\label{prop:join_irreducibles_tamari}
	An element $\pf\in D_n^{(m)}$ is join-irreducible in $\mtam{n}{m}$ if and only if its step sequence $\uu_{\pf}=(u_{1},u_{2},\ldots,u_{n})$ satisfies 
	\begin{equation}\label{eq:mtamari_join_irreducibles}
		u_{j}=\begin{cases}
			m(j-1), & \mbox{for}\;j\notin\{i,i+1,\ldots,k\},\\
			m(j-1)-s, & \mbox{for}\;j\in\{i,i+1,\ldots,k\}
		\end{cases}
	\end{equation}
	for exactly one $i\in\{1,2,\ldots,n\}$, where $k\in\{i+1,i+2,\ldots,n\}$ and $s\in\{1,2,\ldots,m\}$.
\end{proposition}
\begin{proof}
	Let $\pf\in D_n^{(m)}$ with associated step sequence $\uu_{\pf}=(u_{1},u_{2},\ldots,u_{n})$, and suppose that $\uu_{\pf}$ is of the form \eqref{eq:mtamari_join_irreducibles}.  Since the entries $u_{j}$ for $j\notin\{i,i+1,\ldots,k\}$ are maximal, a lower cover of $\pf$ can only be obtained by increasing some of the values $u_{i},u_{i+1},\ldots,u_{k}$.  First, we increase only one entry, \ie we consider the path $\pf_{l}\in D_{n}^{(m)}$ given by the step sequence $\uu_{\pf_{l}}=(u_{1},u_{2},\ldots,u_{l-1},u_{l}+1,u_{l+1},u_{l+2},\ldots,u_{n})$, where $l\in\{i,i+1,\ldots,k\}$.  We show that $\pf_{l}\lessdot_{\text{rot}}\pf$ only if $l=k$.  Indeed, if $l<k$, then we have $u_{l+1}-(u_{l}+1)=ml-s-m(l-1)+s-1=m-1<m$.  Hence $u_{l+1}$ is contained in the primitive subsequence of $\pf_{l}$ at position $l$, which implies that $\pf_{l}$ is no lower cover of $\pf$.  If $l=k$, then we have $u_{l+1}-(u_{l}+1)=ml-m(l-1)+s-1=m+s-1\geq m$, and hence $u_{l+1}$ is not contained in the primitive subsequence of $\pf_{l}$ at 
position $l$.  Thus $\pf_{k}\lessdot_{\text{rot}}\pf$. 
	
	Now we increase at least two entries: for $l_{1},l_{2}\in\{i,i+1,\ldots,k\}$ with $l_{1}<l_{2}$, we consider the path $\pf_{l_{1},l_{2}}$ given by the step sequence 
	\begin{displaymath}
		\uu_{\pf_{l_{1},l_{2}}}=(u_{1},u_{2},\ldots,u_{l_{1}-1},u_{l_{1}}+1,u_{l_{1}+1}+1,\ldots,u_{l_{2}}+1,u_{l_{2}+1},u_{l_{2}+2},\ldots,u_{n}).
	\end{displaymath}
	As before we can show that necessarily $l_{2}=k$.  Moreover, we have $u_{k}+1-(u_{l_{1}}+1)=u_{k}-u_{l_{1}}=m(k-l_{1})\geq m$.  Thus $\pf_{l_{1},k}\leq_{\text{rot}}\pf$, but $\pf_{l_{1},k}$ is no lower cover of $\pf$.  Again we see that $\pf_{l_{1},k-1}\not\leq_{\text{rot}}\pf$, which implies that every chain from $\pf_{l_{1},k}$ to $\pf$ has to pass through $\pf_{k}$.  Hence $\pf_{k}$ is the unique lower cover of $\pf$, and it follows that $\pf\in\jj\bigl(\mtam{n}{m}\bigr)$ as desired. 
	
	\smallskip
	
	For the converse, suppose that $\pf$ is not of the form \eqref{eq:mtamari_join_irreducibles}.  Then, we can find two indices $j_{1},j_{2}$ such that $u_{j_{1}}=m(j_{1}-1)-s_{1}$ and $u_{j_{2}}=m(j_{2}-1)-s_{2}$ with $s_{1},s_{2}>0$, and $s_{1}\neq s_{2}$.  Without loss of generality, we can assume that $j_{2}=j_{1}+1$, which implies $u_{j_{2}}-u_{j_{1}}=mj_{1}-s_{2}-m(j_{1}-1)+s_{1}=m+s_{1}-s_{2}$.  If $s_{1}<s_{2}$, then we have $u_{j_{2}}-u_{j_{1}}<m$, and $u_{j_{2}}$ lies in the primitive subsequence of $\uu_{\pf}$ at position $j_{1}$.  Now if we increase the entries of the primitive subsequence of $\uu_{\pf}$ at position $j_{1}$ by one, then we obtain an $m$-Dyck path $\pf_{1}$, and the $j_{2}$-nd entry of $\uu_{\pf_{1}}$ is contained in the primitive subsequence of $\uu_{\pf_{1}}$ at position $j_{1}$.  Then, $\pf_{1}\lessdot_{\text{rot}}\pf$.  Analogously, if we increase the entries of the primitive subsequence of $\uu_{\pf}$ at position $j_{2}$, then we obtain another $m$-Dyck path $\pf_{2}$ with $\pf_
{2}\lessdot_{\text{rot}}\pf$. 	Clearly we have $\pf_{1}\neq\pf_{2}$, which implies that $\pf\notin\jj\bigl(\mtam{n}{m}\bigr)$.  If $s_{1}>s_{2}$, then $u_{j_{2}}-u_{j_{1}}>m$.  Now increasing $u_{j_{1}}$ by one yields an $m$-Dyck path $\pf_{1}$ with $\pf_{1}\lessdot_{\text{rot}}\pf$, and increasing the entries of the primitive subsequence of $\uu_{\pf}$ at position $j_{2}$ by one yields an $m$-Dyck path $\pf_{2}$ with $\pf_{2}\lessdot_{\text{rot}}\pf$.  Again we have $\pf_{1}\neq\pf_{2}$, which implies $\pf\notin\jj\bigl(\mtam{n}{m}\bigr)$. 
\end{proof}

\begin{corollary}\label{cor:card_join_irreducibles_tamari}
	For every $m,n\in\mathbb{N}_{>0}$, we have $\Bigl\lvert\jj\bigl(\mtam{n}{m}\bigr)\Bigr\rvert=m\binom{n}{2}$.
\end{corollary}
\begin{proof}
	For every two indices $i,k\in\{1,2,\ldots,n\}$ with $i<k$, we can find $m$ join-irreducible elements as described in Lemma~\ref{prop:join_irreducibles_tamari}. 
\end{proof}

\subsection{Proof of Theorem~\ref{thm:mtamari}}
  \label{sec:proof_tamari}
In this section we investigate the $m$-cover poset of $\mtam{n}{1}$.  Theorem~\ref{thm:mcover_lattice} implies that $\tmn{3}{m}$ is a lattice for all $m>0$.  However, for $n>3$ this is no longer true.  Moreover, Proposition~\ref{prop:mcover_cardinality} implies the following result.

\begin{lemma}\label{lem:mcover_tamari_cardinality}
	For $m,n>0$, we have
	\begin{displaymath}
		\left\lvert\dmn{n}{m}\right\rvert = \frac{n-1}{2}\biggl(\mbox{Cat}(n)-2\biggr)\binom{m}{2} + m\cdot\mbox{Cat}(n) - m + 1.
	\end{displaymath}
\end{lemma}
\begin{proof}
	It is well-known that $\lvert D_{n}\rvert=\mbox{Cat}(n)$ and it follows by construction that $\mtam{n}{1}$ has $n-1$ atoms.  Moreover, \cite{geyer94on}*{Theorem~5.3} states that the number of cover relations in $\mtam{n}{1}$ is precisely $\tfrac{n-1}{2}\mbox{Cat}(n)$.  Now the result follows from Proposition~\ref{prop:mcover_cardinality}.
\end{proof}

We observe that $\bigl\lvert\dmn{n}{m}\bigr\rvert<\mbox{Cat}^{(m)}(n)$ for $n>3$ and $m>1$.  Since $\tmn{n}{m}$ is in this case not a lattice, it makes sense to consider a lattice completion of $\tmn{n}{m}$.  And indeed, Theorem~\ref{thm:mtamari} claims that the Dedekind-MacNeille completion of $\tmn{n}{m}$, namely the smallest lattice that contains $\tmn{n}{m}$ as a subposet, is isomorphic to $\mtam{n}{m}$.  In order to establish this connection, we introduce the following map from $\mdyck{n}{m}$ to $(\mdyck{n}{1})^{m}$.

\begin{definition}\label{def:strip_decomposition}
	Let $\pf\in\mdyck{n}{m}$ be an $m$-Dyck path with associated height sequence $\hh_{\pf}=(h_{1},h_{2},\ldots,h_{mn})$ and consider the sequence $\delta(\pf)=(\qf_{1},\qf_{2},\ldots,\qf_{m})$, where for each $i\in\{1,2,\ldots,m\}$ the path $\qf_i\in\mdyck{n}{1}$ is determined by the height sequence $\hh_{i}=\bigl(h_{i},h_{i+m},\ldots,h_{i+(n-1)m}\bigr)$.  We call the sequence $\delta(\pf)$ the \alert{strip decomposition of $\pf$}. 
\end{definition}

The following lemma is immediate.

\begin{lemma} \label{lem:delta_injective}
	The map $\delta: D_n^{(m)}\to (D_n)^{m}$ described in Definition~\ref{def:strip_decomposition} is well-defined and injective.
\end{lemma}

The inverse of $\delta$ can be described explicitly using Lemma~\ref{lem:height_step_conversion}.

\begin{example}\label{ex:dyck35_2}
	We continue with Example~\ref{ex:dyck35_1}.  Let $\pf$ be the $3$-Dyck path of length $20$ shown in Figure~\ref{fig:dyck35_whole} given by the height sequence $\hh_{\pf}=(1,1,3,3,3,3,3,3,4,4,5,5,5,5,5)$.  Since $m=3$, we obtain three sequences
	\begin{align*}
		\hh_{\qf_{1}} & = (h_{1},h_{4},h_{7},h_{10},h_{13}) = (1,3,3,4,5),\\
		\hh_{\qf_{2}} & = (h_{2},h_{5},h_{8},h_{11},h_{14}) = (1,3,3,5,5),\quad\text{and}\\
		\hh_{\qf_{3}} & = (h_{3},h_{6},h_{9},h_{12},h_{15}) = (3,3,4,5,5),
	\end{align*}
	which correspond to the three Dyck paths shown in Figure~\ref{fig:dyck35_decomp}.
	
	\begin{figure}
		\centering
		\subfigure[A $3$-Dyck path $\pf$ of length $20$.]{\label{fig:dyck35_whole}
		  \begin{tikzpicture}
			\draw(0,0) node{
				\begin{tikzpicture}[scale=.8]\tiny
					\draw[white!50!gray](0,0) grid[step=.4] (6,2);
					\draw(0,0) -- (6,2);
					\draw(0,0) -- (0,.4) -- (.8,.4) -- (.8,1.2) -- (3.2,1.2) -- (3.2,1.6) -- (4,1.6) 
					  -- (4,2) -- (6,2);
					\draw[white!50!gray,->](6,0) -- (6.5,0);
					\draw[white!50!gray,->](0,2) -- (0,2.5);
					\draw[white!50!gray](6.4,-.25) node{$x$};
					\draw[white!50!gray](-.25,2.4) node{$y$};
					\foreach\i in {0,1,2,3,4,5,6,7,8,9,10,11,12,13,14,15} {
						\draw(.4*\i,-.25) node[white!50!gray]{$\i$};
					}
					\foreach\i in {0,1,2,3,4,5} {
						\draw(-.25,.4*\i) node[white!50!gray]{$\i$};
					}
					\draw(3,2.25) node{$\pf$};
				\end{tikzpicture}};
		  \end{tikzpicture}}\qquad
		\subfigure[The strip decomposition of $\pf$.]{\label{fig:dyck35_decomp}
		  \begin{tikzpicture}
			\draw(0,0) node{
				\begin{tikzpicture}[scale=.8]\tiny
					\draw[white!50!gray](0,0) grid[step=.4] (2,2);
					\draw[white!50!gray,->](2,0) -- (2.5,0);
					\draw[white!50!gray,->](0,2) -- (0,2.5);
					\draw[white!50!gray](2.4,-.25) node{$x$};
					\draw[white!50!gray](-.25,2.4) node{$y$};
					\foreach\i in {0,1,2,3,4,5} {
						\draw(.4*\i,-.25) node[white!50!gray]{$\i$};
					}
					\foreach\i in {0,1,2,3,4,5} {
						\draw(-.25,.4*\i) node[white!50!gray]{$\i$};
					}
					\draw(0,0) -- (0,.4) -- (.4,.4) -- (.4,1.2) -- (1.2,1.2) 
					  -- (1.2,1.6) -- (1.6,1.6) -- (1.6,2) -- (2,2) -- (0,0);
					\draw(1,2.25) node{$\qf_{1}$};
				\end{tikzpicture}};
			\draw(2.5,0) node{
				\begin{tikzpicture}[scale=.8]\tiny
					\draw[white!50!gray](0,0) grid[step=.4] (2,2);
					\draw[white!50!gray,->](2,0) -- (2.5,0);
					\draw[white!50!gray,->](0,2) -- (0,2.5);
					\draw[white!50!gray](2.4,-.25) node{$x$};
					\draw[white!50!gray](-.25,2.4) node{$y$};
					\foreach\i in {0,1,2,3,4,5} {
						\draw(.4*\i,-.25) node[white!50!gray]{$\i$};
					}
					\foreach\i in {0,1,2,3,4,5} {
						\draw(-.25,.4*\i) node[white!50!gray]{$\i$};
					}
					\draw(0,0) -- (0,.4) -- (.4,.4) -- (.4,1.2) -- (1.2,1.2) 
					  -- (1.2,2) -- (1.6,2) -- (2,2) -- (0,0);
					\draw(1,2.25) node{$\qf_{2}$};
				\end{tikzpicture}};
			\draw(5,0) node{
				\begin{tikzpicture}[scale=.8]\tiny
					\draw[white!50!gray](0,0) grid[step=.4] (2,2);
					\draw[white!50!gray,->](2,0) -- (2.5,0);
					\draw[white!50!gray,->](0,2) -- (0,2.5);
					\draw[white!50!gray](2.4,-.25) node{$x$};
					\draw[white!50!gray](-.25,2.4) node{$y$};
					\foreach\i in {0,1,2,3,4,5} {
						\draw(.4*\i,-.25) node[white!50!gray]{$\i$};
					}
					\foreach\i in {0,1,2,3,4,5} {
						\draw(-.25,.4*\i) node[white!50!gray]{$\i$};
					}
					\draw(0,0) -- (0,1.2) -- (.8,1.2) -- (.8,1.6) -- (1.2,1.6) 
					  -- (1.2,2) -- (2,2) -- (0,0);
					\draw(1,2.25) node{$\qf_{3}$};
				\end{tikzpicture}};
		  \end{tikzpicture}}
		\caption{Illustration of the strip decomposition.}
		\label{fig:dyck35_rev}
	\end{figure}
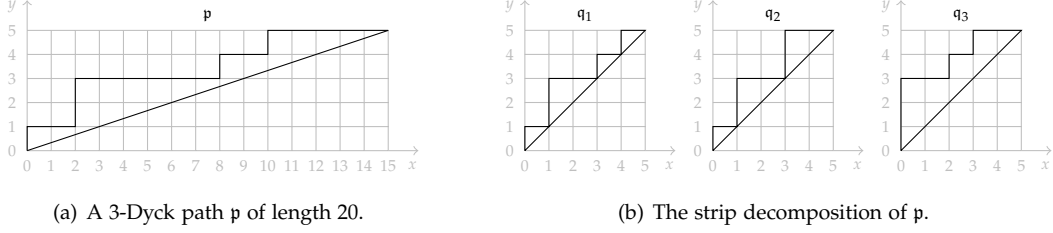
\end{example}

In the remainder of this section, we prove Theorem~\ref{thm:mtamari} by showing that the strip decomposition is an order-isomorphism between the restriction of $\mtam{n}{m}$ and $\tmn{n}{m}$ to their respective sets of join- and meet-irreducible elements.  Recall that for any poset $\PP=(P,\leq)$, the sets of join- and meet-irreducible elements, denoted by $\jj(\PP)$ and $\mm(\PP)$, naturally induce subposets $\JJ(\PP)=\bigl(\jj(\PP),\leq\bigr)$ and $\MM(\PP)=\bigl(\mm(\PP),\leq\bigr)$, respectively. 

We start with the investigation of the subposet of $\tmn{n}{m}$ consisting of meet-irreducible elements.

\begin{proposition}\label{prop:meet_isomorphic}
	The posets $\MM\bigl(\mtam{n}{m}\bigr)$ and $\MM\bigl(\tmn{n}{m}\bigr)$ are isomorphic.
\end{proposition}
\begin{proof}
	If $n\leq 3$, then $\mtam{n}{m}\cong\tmn{n}{m}$, and we are done.  So suppose that $n>3$.  Since $\mtam{n}{m}$ has $n-1$ atoms and $n-1$ coatoms, Proposition~\ref{prop:mcover_irreducibles} implies that the set of meet-irreducible elements of $\tmn{n}{m}$ can be written as 
	\begin{displaymath}
		\mm\bigl(\tmn{n}{m}\bigr)=\Bigl\{\bigl(\qf^{l},(\qf^{\star})^{m-l}\bigr)\mid\qf\in\mm\bigl(\mtam{n}{1}\bigr)\;\text{and}\;1\leq l\leq m\Bigr\}.
	\end{displaymath}
	We show that the map $\delta$ from Definition~\ref{def:strip_decomposition} is a poset-isomorphism from $\MM\bigl(\mtam{n}{m}\bigr)$ to $\MM\bigl(\tmn{n}{m}\bigr)$.  
    
	First we show that $\delta\Bigl(\mm\bigl(\mtam{n}{m}\bigr)\Bigr)\subseteq\mm\bigl(\tmn{n}{m}\bigr)$.  Let $\pf\in\mm\bigl(\mtam{n}{m}\bigr)$.  Proposition~\ref{prop:meet_irreducibles_tamari} implies that $\uu_{\pf}=(0,0,\dots,0,a,a,\dots,a)$, where $a$ first appears at the $(i+1)$-st position and satisfies $1\leq a\leq mi$.  Moreover, if we write $a=mk+t$ with $t\in\{0,1,\dots,m-1\}$ and $k\in\{0,1,\ldots,i\}$, then we obtain $\delta(\pf)=(\qf_{1},\qf_{2},\dots,\qf_{m})$, where 
	\begin{equation}\label{eq:image_step_meet}
		\uu_{\qf_j}=\begin{cases}
		  (0,0,\dots,0,k+1,k+1,\dots,k+1), & \ \text{if}\ j\leq t \\
		  (0,0,\dots,0,k,k,\dots,k), & \ \text{if}\ j>t,  
		\end{cases}
	\end{equation}
	and $k+1$ (respectively $k$) first appears at the $(i+1)$-st position of $\uu_{\qf_{j}}$.  Proposition~\ref{prop:meet_irreducibles_tamari} implies $\qf_{j}\in\mm\bigl(\mtam{n}{1}\bigr)$, and it follows immediately from the definition of $\leq_{\text{rot}}$ that $\qf_{t}\lessdot_{\text{rot}}\qf_{t+1}$.  Hence $\delta(\pf)\in\mm\bigl(\tmn{n}{m}\bigr)$, which implies the claim.  Lemma~\ref{lem:delta_injective} implies that $\delta$ is injective, and Corollary~\ref{cor:card_meet_irreducibles_tamari} and Proposition~\ref{prop:mcover_irreducibles} imply that 
	\begin{displaymath}
		\Bigl\lvert\mm\bigl(\tmn{n}{m}\bigr)\Bigr\rvert = m\binom{n}{2} = \Bigl\lvert\mm\bigl(\mtam{n}{m}\bigr)\Bigr\rvert.
	\end{displaymath}
	Hence we have $\delta\Bigl(\mm\bigl(\mtam{n}{m}\bigr)\Bigr)=\mm\bigl(\tmn{n}{m}\bigr)$.
  
	It remains to show that (in the current situation) $\delta$ and its inverse are both order-preserving.  Let $\pf,\pf'\in\mm(\mtam{n}{m})$, and first assume that $\pf\leq_{\text{rot}}\pf'$.  By assumption, we can write 
	\begin{displaymath}
		\uu_{\pf}=(0,0,\dots,0,a,a,\dots,a)\quad\text{and}\quad\uu_{\pf'}=(0,0,\dots,0,a',a',\dots,a'),
	\end{displaymath}
	where $a=mk+t$ first appears in the $i$-th position of $\uu_{\pf}$ and $a'=mk'+t'$ first appears in the $j$-th position of $\uu_{\pf'}$.  Since $\pf\leq_{\text{rot}}\pf'$, we conclude that $i=j$, and either $k'<k$, or $k'=k$ and $t'<t$.  Let $\delta(\pf)=(\qf_{1},\qf_{2},\ldots,\qf_{m})$, and let $\delta(\pf')=(\qf'_{1},\qf'_{2},\ldots,\qf'_{m})$. \\
	(i) First suppose that $t'<t$.  Then, $k'\leq k$, and in view of \eqref{eq:image_step_meet} we obtain 
	\begin{align*}
		\uu_{\qf_{1}} = \uu_{\qf_{2}} = \cdots = \uu_{\qf_{t'}} & = (0,0,\ldots,0,k+1,k+1,\ldots,k+1),\\
		\uu_{\qf_{t'+1}} = \uu_{\qf_{t'+2}} = \cdots = \uu_{\qf_{t}} & = (0,0,\ldots,0,k+1,k+1,\ldots,k+1),\quad\text{and}\\
		\uu_{\qf_{t+1}} = \uu_{\qf_{t+2}} = \cdots = \uu_{\qf_{m}} & = (0,0,\ldots,0,k,k,\ldots,k),
	\end{align*}
	as well as
	\begin{align*}
		\uu_{\qf'_{1}} = \uu_{\qf'_{2}} = \cdots = \uu_{\qf'_{t'}} & = (0,0,\ldots,0,k'+1,k'+1,\ldots,k'+1),\\
		\uu_{\qf'_{t'+1}} = \uu_{\qf'_{t'+2}} = \cdots = \uu_{\qf'_{t}} & = (0,0,\ldots,0,k',k',\ldots,k'),\quad\text{and}\\
		\uu_{\qf'_{t+1}} = \uu_{\qf'_{t+2}} = \cdots = \uu_{\qf'_{m}} & = (0,0,\ldots,0,k',k',\ldots,k'),
	\end{align*}
	and it follows immediately that $\qf_{s}\leq_{\text{rot}}\qf'_{s}$ for all $s\in\{1,2,\ldots,m\}$.\\ 
	(ii) Now, suppose that $t'\geq t$.  Then, $k'<k$, and in view of \eqref{eq:image_step_meet}, we obtain 
	\begin{align*}
		\uu_{\qf_{1}} = \uu_{\qf_{2}} = \cdots = \uu_{\qf_{t}} & = (0,0,\ldots,0,k+1,k+1,\ldots,k+1),\\
		\uu_{\qf_{t+1}} = \uu_{\qf_{t+2}} = \cdots = \uu_{\qf_{t'}} & = (0,0,\ldots,0,k,k,\ldots,k),\quad\text{and}\\
		\uu_{\qf_{t'+1}} = \uu_{\qf_{t'+2}} = \cdots = \uu_{\qf_{m}} & = (0,0,\ldots,0,k,k,\ldots,k),
	\end{align*}
	as well as
	\begin{align*}
		\uu_{\qf'_{1}} = \uu_{\qf'_{1}} = \cdots = \uu_{\qf'_{t}} & = (0,0,\ldots,0,k'+1,k'+1,\ldots,k'+1),\\
		\uu_{\qf'_{t+1}} = \uu_{\qf'_{t+2}} = \cdots = \uu_{\qf'_{t'}} & = (0,0,\ldots,0,k'+1,k'+1,\ldots,k'+1),\quad\text{and}\\
		\uu_{\qf'_{t'+1}} = \uu_{\qf'_{t'+2}} = \cdots = \uu_{\qf'_{m}} & = (0,0,\ldots,0,k',k',\ldots,k'),
	\end{align*}
	and it follows again that $\qf_{s}\leq_{\text{rot}}\qf'_{s}$ for all $s\in\{1,2,\ldots,m\}$.
	
	If we assume that $\delta(\pf)\leq_{\text{rot}}\delta(\pf')$, then we obtain analogously that $\pf\leq_{\text{rot}}\pf'$, and we are done.
\end{proof}

Now we investigate the subposet of $\tmn{n}{m}$ consisting of join-irreducible elements.

\begin{proposition}\label{prop:join_isomorphic}
	The posets $\JJ\bigl(\mtam{n}{m}\bigr)$ and $\JJ\bigl(\tmn{n}{m}\bigr)$ are isomorphic.
\end{proposition}
\begin{proof}
	Proposition~\ref{prop:mcover_irreducibles} implies that that the set of join-irreducible elements of $\tmn{n}{m}$ can be written as 
	\begin{displaymath}
		\jj\bigl(\tmn{n}{m}\bigr)=\Bigl\{\bigl(\mathfrak{0}^{l},\qf^{m-l}\bigr)\mid\qf\in\jj\bigl(\mtam{n}{1}\bigr)\;\text{and}\;0\leq l<m\Bigr\},
	\end{displaymath}
	where $\mathfrak{0}$ is the Dyck path with associated step sequence $(0,1,\ldots,n-1)$.  We show that the map $\delta$ from Definition~\ref{def:strip_decomposition} is a poset-isomorphism from $\JJ\bigl(\mtam{n}{m}\bigr)$ to $\JJ\bigl(\tmn{n}{m}\bigr)$.
    
	First we show that $\delta\Bigl(\jj\bigl(\mtam{n}{m}\bigr)\Bigr)\subseteq\jj\bigl(\tmn{n}{m}\bigr)$.  Let $\pf\in\jj\bigl(\mtam{n}{m}\bigr)$ and let $\uu_{\pf}=(u_{1},u_{2},\ldots,u_{n})$ be its step sequence.  Then, Proposition~\ref{prop:join_irreducibles_tamari} implies that 
	\begin{displaymath}
		u_{j}=\begin{cases}m(j-1), & \text{for}\;j\notin\{i+1,i+2,\ldots,k\},\\
		  m(j-1)-s, & \text{for}\;j\in\{i+1,i+2,\ldots,k\}.
		\end{cases}
	\end{displaymath}
	for exactly one $i\in\{1,2,\ldots,n\}$ and some $k\in\{i+1,i+2,\ldots,n\}$ and some (fixed) $s\in\{1,2,\ldots,m\}$.  Now let $\delta(\pf)=(\qf_{1},\qf_{2},\ldots,\qf_{m})$.  Then, it is straightforward to show that 
	\begin{equation}\label{eq:image_step_join}
		\uu_{\qf_{j}}=\begin{cases}
			(0,1,\ldots,n-1), & \text{if}\;j\leq m-s,\\
			(0,1,\ldots,i-2,i-2,i-1\ldots,k-2,k,k+1\ldots,n-1), & \text{if}\;j>m-s.
		\end{cases}
	\end{equation}
	Proposition~\ref{prop:join_irreducibles_tamari} implies that either $\qf_{j}=\mathfrak{0}$ or $\qf_{j}\in\jj\bigl(\mtam{n}{1}\bigr)$.  Hence it follows that $\delta(\pf)\in\jj\bigl(\tmn{n}{m}\bigr)$, which implies the claim.  Lemma~\ref{lem:delta_injective} implies that $\delta$ is injective, and Corollary~\ref{cor:card_join_irreducibles_tamari} and Proposition~\ref{prop:mcover_irreducibles} imply that 
	\begin{displaymath}
		\Bigl\lvert\jj\bigl(\tmn{n}{m}\bigr)\Bigr\rvert = m\binom{n}{2} = \Bigl\lvert\jj\bigl(\mtam{n}{m}\bigr)\Bigr\rvert.
	\end{displaymath}
	Hence we have $\delta\Bigl(\jj\bigl(\mtam{n}{m}\bigr)\Bigr)=\jj\bigl(\tmn{n}{m}\bigr)$.
	
	It remains to show that (in the current situation) $\delta$ and its inverse are both order-preserving.  Let $\pf,\pf'\in\jj(\mtam{n}{m})$, and first assume that $\pf\leq_{\text{rot}}\pf'$.  Denote by $i,k,s$ the parameters of $\pf$ according to Proposition~\ref{prop:join_irreducibles_tamari}, and denote by $i',k',s'$ the analogous parameters of $\pf'$.  In view of \eqref{eq:image_step_join}, we can write $\delta(\pf)=\bigl(\mathfrak{0}^{m-s},\qf^{s}\bigr)$ and $\delta(\pf')=\bigl(\mathfrak{0}^{m-s'},(\qf')^{s'}\bigr)$, for $\qf,\qf'\in\jj\bigl(\mtam{n}{1}\bigr)$.  Moreover, let $\uu_{\pf}=(u_{1},u_{2},\ldots,u_{n})$ and $\uu_{\pf'}=(u'_{1},u'_{2},\ldots,u'_{n})$ denote the step sequences of $\pf$ and $\pf'$.  We need to show that $s\leq s'$ and that $\qf\leq_{\text{rot}}\qf'$.  First of all Lemma~\ref{lem:sequence_order} implies that $i'\leq i$ and $k\leq k'$.  For $j\in\{i+1,i+2,\ldots,k\}$, we obtain $u_{j}=m(j-1)-s$ and $u'_{j}=m(j-1)-s'$.  Again, Lemma~\ref{lem:sequence_order} implies $u_{j}\geq u'_{j}$,
 which yields $s\leq s'$.  In view of \eqref{eq:image_step_join}, we obtain $\uu_{\qf}=(0,1,\ldots,i-2,i-2,i-1,\ldots,k-2,k,k+1,\ldots,n)$ and $\uu_{\qf'}=(0,1,\ldots,i'-2,i'-2,i'-1,\ldots,k'-2,k',k'+1,\ldots,n)$.  It is immediate that we can construct a sequence of coverings in $\mtam{n}{1}$ that yields a saturated chain from $\qf$ to $\qf'$.  Hence $\qf\leq_{\text{rot}}\qf'$, and we obtain $\delta(\pf)\leq_{\text{rot}}\delta(\pf')$.
	
	For the converse, suppose that $(\mathfrak{0}^{m-s},\qf^{s})\leq_{\text{rot}}(\mathfrak{0}^{m-s'},(\qf')^{s'})$.  This implies immediately that $s\leq s'$ and $\qf\leq_{\text{rot}}\qf'$.  Since $\qf$ and $\qf'$ are join-irreducible in $\mtam{n}{1}$, we can apply Proposition~\ref{prop:join_irreducibles_tamari}.  We obtain two parameters $i,k$ for $\qf$ and two parameters $i',k'$ for $\qf'$, which satisfy $i'\leq i\leq k\leq k'$.  Thus we have $u_{j}=m(j-1)$ for $j\notin\{i+1,i+2,\ldots,k\}$ and $u_{j}=m(j-1)-s$ for $j\in\{i+1,i+2,\ldots,k\}$, as well as $u'_{j}=m(j-1)$ for $j\notin\{i'+1,i'+2,\ldots,k'\}$ and $u'_{j}=m(j-1)-s'$ for $j\in\{i'+1,i'+2,\ldots,k'\}$.  Again it is immediate that we can construct a sequence of coverings in $\mtam{n}{m}$ that yields a chain from $\pf$ to $\pf'$.  Hence $\pf\leq_{\text{rot}}\pf'$, and we are done.
\end{proof}

By abuse of notation, we say that for any poset $\PP=(P,\leq)$, a set $\{a_{1},a_{2},\ldots,a_{s}\}\subseteq P$ is a \alert{join representation} of $p\in P$ if $p=a_{1}\vee a_{2}\vee\cdots\vee a_{s}$.  A join representation $A$ of $p$ is \alert{canonical} if for any other join representation $B$ of $p$ and any $a\in A$ there exists some $b\in B$ with $a\leq b$.  Dually, a set $\{a_{1},a_{2},\ldots,a_{s}\}\subseteq P$ is a \alert{meet representation} of $p\in P$ if $p=a_{1}\wedge a_{2}\wedge\cdots\wedge a_{s}$.  A meet representation $A$ of $p$ is \alert{canonical} if for any other meet representation $B$ of $p$ and any $a\in A$ there exists some $b\in B$ with $b\leq a$.  See \cite{freese95free}*{Chapter~I, Section~3} for more background on canonical representations in lattices.  

A subset $Q\subseteq P$ is \alert{join-dense} if any $p\in P$ has a join representation that is entirely contained in $Q$, and dually $Q\subseteq P$ is \alert{meet-dense} if any $p\in P$ has a meet representation that is entirely contained in $Q$.  The following result is a consequence of \cite{urquhart78topological}*{Theorem~12} and \cite{freese95free}*{Theorems~2.13~and~2.21}.

\begin{theorem}\label{thm:tamari_canonical}
	For $n>0$, any element of $\mtam{n}{1}$ has both a canonical join representation and a canonical meet representation. 
\end{theorem}

\begin{proposition}\label{prop:mcover_join_dense}
	The set $\jj\bigl(\tmn{n}{m}\bigr)$ is join-dense in $\tmn{n}{m}$.
\end{proposition}
\begin{proof}
	Let $\bigl(\mathfrak{0}^{l_{0}},\qf^{l_{1}},(\qf')^{l_{2}}\bigr)\in\dmn{n}{m}$.  If $l_{1}=l_{2}=0$, then $\bigl(\mathfrak{0}^{m}\bigr)$ is the least element of $\tmn{n}{m}$ and can clearly be expressed as the empty join over any set.  Otherwise, suppose that $\{\rf_{1},\rf_{2},\ldots,\rf_{k}\}$ is the canonical join representation of $\qf$ and that $\{\rf'_{1},\rf'_{2},\ldots,\rf'_{k'}\}$ is the canonical join representation of $\qf'$.  (These exist thanks to Theorem~\ref{thm:tamari_canonical}.)  In particular, we have $\qf=\rf_{1}\vee\rf_{2}\vee\cdots\vee\rf_{k}$ and $\qf'=\rf'_{1}\vee\rf'_{2}\vee\cdots\vee\rf'_{k'}$ with $\rf_{i},\rf'_{j}\in\jj\bigl(\mtam{n}{1}\bigr)$ for $i\in\{1,2,\ldots,k\},j\in\{1,2,\ldots,k'\}$.  (It is straightforward to verify that the elements of a canonical join representation are join-irreducible.)
	
	Define $\ww_{i}=\bigl(\mathfrak{0}^{l_{0}},\rf_{i}^{l_{1}+l_{2}}\bigr)$ for $i\in\{1,2,\ldots,k\}$, and $\ww'_{j}=(\mathfrak{0}^{l_{0}+l_{1}},(\rf'_{j})^{l_{2}}\bigr)$ for $j\in\{1,2,\ldots,k'\}$.  Proposition~\ref{prop:mcover_irreducibles} implies that $\ww_{i},\ww'_{j}\in\jj\bigl(\tmn{n}{m}\bigr)$ for $i\in\{1,2,\ldots,k\}$ and $j\in\{1,2,\ldots,k'\}$.  Moreover,
	\begin{align*}
		\ww_{1}\vee\cdots\vee\ww_{k}\vee\ww'_{1}\vee\cdots\vee\ww'_{k'} 
		& = \left(\Bigl(\bigvee_{i=1}^{k}{\mathfrak{0}}\vee\bigvee_{j=1}^{k'}{\mathfrak{0}}\Bigr)^{l_{0}},\Bigl(\bigvee_{i=1}^{k}{\rf_{i}}\vee\bigvee_{j=1}^{k'}{\mathfrak{0}}\Bigr)^{l_{1}},\Bigl(\bigvee_{i=1}^{k}{\rf_{i}}\vee\bigvee_{j=1}^{k'}{\rf'_{j}}\Bigr)^{l_{2}}\right)\\
		& = \left(\mathfrak{0}^{l_{0}},\Bigl(\qf\vee\mathfrak{0}\Bigr)^{l_{1}},\Bigl(\qf\vee\qf'\Bigr)^{l_{2}}\right)\\
		& = \Bigl(\mathfrak{0}^{l_{0}},\qf^{l_{1}},(\qf')^{l_{2}}\Bigr),
	\end{align*}
	as desired.
\end{proof}

\begin{proposition}\label{prop:mcover_meet_dense}
	The set $\mm\bigl(\tmn{n}{m}\bigr)$ is meet-dense in $\tmn{n}{m}$. 
\end{proposition}
\begin{proof}
	Let $\bigl(\mathfrak{0}^{l_{0}},\qf^{l_{1}},(\qf')^{l_{2}}\bigr)\in\dmn{n}{m}$.  If $\qf'=\mathfrak{1}$, \ie the Dyck path given by the step sequence $(0,0,\ldots,0)$, and $l_{0}=l_{1}=0$, then $\bigl(\mathfrak{1}^{m}\bigr)$ is the greatest element of $\tmn{n}{m}$ and can clearly be expressed as the empty meet over any set.  Otherwise, suppose that $\{\rf_{1},\rf_{2},\ldots,\rf_{k}\}$ is the canonical meet representation of $\qf$ and that $\{\rf'_{1},\rf'_{2},\ldots,\rf'_{k'}\}$ is the canonical meet representation of $\qf'$.  (These exist thanks to Theorem~\ref{thm:tamari_canonical}.)  In particular, we have $\qf=\rf_{1}\wedge\rf_{2}\wedge\cdots\wedge\rf_{k}$ and $\qf'=\rf'_{1}\wedge\rf'_{2}\wedge\cdots\wedge\rf'_{k'}$ with $\rf_{i},\rf'_{j}\in\mm\bigl(\mtam{n}{1}\bigr)$ for $i\in\{1,2,\ldots,k\},j\in\{1,2,\ldots,k'\}$.  (It is straightforward to verify that the elements of a canonical meet representation are meet-irreducible.)
	
	Recall that for $i\in\{1,2,\ldots,k\}$, the meet-irreducible element $\rf_{i}$ has a unique upper cover, denoted by $\rf_{i}^{\star}$.  Define $\ww_{i}=\bigl(\rf_{i}^{l_{0}+l_{1}},(\rf_{i}^{\star})^{l_{2}}\bigr)$ for $i\in\{1,2,\ldots,k\}$, and $\ww'_{j}=\bigl((\rf')^{m}\bigr)$ for $j\in\{1,2,\ldots,k'\}$.  Proposition~\ref{prop:mcover_irreducibles} implies that $\ww_{i},\ww'_{j}\in\mm\bigl(\tmn{n}{m}\bigr)$ for $i\in\{1,2,\ldots,k\}$ and $j\in\{1,2,\ldots,k'\}$.  Moreover, let $\qf^{\star}=\rf_{1}^{\star}\wedge\rf_{2}^{\star}\wedge\cdots\wedge\rf_{k}^{\star}$.  Since $\{\rf_{1},\rf_{2},\ldots,\rf_{k}\}$ is the canonical meet representation of $\qf$, we conclude that $\qf<_{\text{rot}}\qf^{\star}$.  Suppose that $\qf'\not\leq_{\text{rot}}\qf^{\star}$.  Since $\qf<_{\text{rot}}\qf^{\star}$, this can only be the case if $\qf'$ and $\qf^{\star}$ are incomparable, and it follows that $\qf=\qf'\wedge\qf^{\star}$.  We can thus write
	\begin{displaymath}
		\qf = \rf'_{1}\wedge\rf'_{2}\wedge\cdots\wedge\rf'_{k'}\wedge\rf_{1}^{\star}\wedge\rf_{2}^{\star}\wedge\cdots\wedge\rf_{k}^{\star}.
	\end{displaymath}
	Since $\{\rf_{1},\rf_{2},\ldots,\rf_{k}\}$ is the canonical meet representation of $\qf$ (which is easily checked to consist of pairwise incomparable elements), and since $\rf_{i}\lessdot_{\text{rot}}\rf_{i}^{\star}$ we conclude that for any $\rf_{i}$ there must be some $\rf'_{j}$ with $\rf'_{j}\leq_{\text{rot}}\rf_{i}$.  Hence $\qf'\leq_{\text{rot}}\rf_{i}$ for all $i\in\{1,2,\ldots,k\}$, and we conclude
	\begin{displaymath}
		\qf'\leq_{\text{rot}}\rf_{1}\wedge\rf_{2}\wedge\cdots\wedge\rf_{k}=\qf\lessdot_{\text{rot}}\qf',
	\end{displaymath}
	which is a contradiction.  Hence we have $\qf'\leq_{\text{rot}}\qf^{\star}$.  
	
	For $s\in\{1,2,\ldots,n-1\}$, let $\cf_{s}\in\mdyck{n}{1}$ be the Dyck path with step sequence $\uu_{\cf_{s}}=(0,\ldots,0,1,\ldots,1)$, where the first $1$ appears in the $s$-th position.  It is immediate that $\cf_{s}\lessdot_{\text{rot}}\mathfrak{1}$, and consequently $\cf_{s}\in\mm\bigl(\mtam{n}{1}\bigr)$ for all $s\in\{1,2,\ldots,n-1\}$.  Define $\cb_{s}=\bigl(\cf_{s}^{l_{0}},\mathfrak{1}^{l_{1}+l_{2}}\bigr)$ for $s\in\{1,2,\ldots,n-1\}$.  Proposition~\ref{prop:mcover_irreducibles} implies that $\cb_{s}\in\mm\bigl(\tmn{n}{m}\bigr)$ for all $s\in\{1,2,\ldots,n-1\}$.  We have
	\begin{align*}
		\ww_{1}\wedge\cdots\wedge\ww_{k}\wedge\ww'_{1}\wedge\cdots\wedge\ww'_{k'}\wedge\cb_{1}\wedge\cdots\wedge\cb_{n-1}
			& = \left(\Bigl(\bigwedge_{i=1}^{k}{\rf_{i}}\wedge\bigwedge_{j=1}^{k'}{\rf'_{j}}\wedge\bigwedge_{s=1}^{n-1}{\cf_{s}}\Bigr)^{l_{0}},\right.\\
			& \kern1cm \Bigl(\bigwedge_{i=1}^{k}{\rf_{i}}\wedge\bigwedge_{j=1}^{k'}{\rf'_{j}}\wedge\bigwedge_{s=1}^{n-1}{\mathfrak{1}}\Bigr)^{l_{1}},\\
			& \kern1cm \left.\Bigl(\bigwedge_{i=1}^{k}{\rf_{i}^{\star}}\wedge\bigwedge_{j=1}^{k'}{\rf'_{j}}\wedge\bigwedge_{s=1}^{n-1}{\mathfrak{1}}\Bigr)^{l_{2}}\right)\\
			& = \left(\Bigl(\qf\wedge\qf'\wedge\mathfrak{0}\Bigr)^{l_{0}},\Bigl(\qf\wedge\qf'\wedge\mathfrak{1}\Bigr)^{l_{1}},\Bigl(\qf^{\star}\wedge\qf'\wedge\mathfrak{1}\Bigr)^{l_{2}}\right)\\
			& = \left(\mathfrak{0}^{l_{0}},\qf^{l_{1}},(\qf')^{l_{2}}\right),
	\end{align*}
	as desired.
\end{proof}

Let us now formally define the Dedekind-MacNeille completion of a poset $\PP=(P,\leq)$.  For $A\subseteq P$, define 
\begin{displaymath}
	A^{u} = \{x\in P\mid a\leq x\;\text{for all}\;a\in A\},\quad\text{and}\quad
	A^{l} = \{x\in P\mid x\leq a\;\text{for all}\;a\in A\}.
\end{displaymath}
The poset $\DM(\PP)=\bigl(\DM(P),\subseteq\bigr)$, where $\DM(P)=\{A\subseteq P\mid A^{ul}=A\}$, is the \alert{Dedekind-MacNeille completion} of $\PP$, \ie the smallest lattice that contains $\PP$ as a subposet.  Recall the following results.

\begin{theorem}[\cite{banaschewski56huellensysteme}*{Korollar~3}]\label{thm:completion_irreducibles}
	For any finite lattice $\LL$ we have
	\begin{displaymath}
		\LL\cong\DM\Bigl(\JJ(\LL)\cup\MM(\LL)\Bigr).
	\end{displaymath}
\end{theorem}

\begin{theorem}[\cite{davey02introduction}*{Theorem~7.42}]\label{thm:completion_characterization}
	Let $\PP=(P,\leq)$ be a poset, and let $\varphi:P\to\DM(P)$ be the map given by $\varphi(p)=\{x\in P\mid x\leq p\}$.  
	\begin{enumerate}[(i)]
		\item $\varphi(P)$ is both join-dense and meet-dense in $\DM(P)$. 
		\item Let $\LL$ be a lattice such that $P$ is join-dense and meet-dense in $\LL$.  Then $\LL\cong\DM(P)$ via an order-isomorphism which agrees with $\varphi$ on $P$. 
	\end{enumerate}
\end{theorem}

Now we have gathered all the ingredients to prove Theorem~\ref{thm:mtamari}.

\begin{proof}[Proof of Theorem~\ref{thm:mtamari}]
	Theorem~\ref{thm:completion_irreducibles}, and Propositions~\ref{prop:meet_isomorphic}, and \ref{prop:join_isomorphic} imply
	\begin{displaymath}
		\mtam{n}{m}\cong\DM\Bigl(\JJ\bigl(\mtam{n}{m}\bigr)\cup\MM\bigl(\mtam{n}{m}\bigr)\Bigr)\cong\DM\Bigl(\JJ\bigl(\tmn{n}{m}\bigr)\cup\MM\bigl(\tmn{n}{m}\bigr)\Bigr).
	\end{displaymath}
	Propositions~\ref{prop:mcover_join_dense} and \ref{prop:mcover_meet_dense} imply that $\jj\bigl(\tmn{n}{m}\bigr)\cup\mm\bigl(\tmn{n}{m}\bigr)$ is join-dense and meet-dense in $\tmn{n}{m}$, and by Theorem~\ref{thm:completion_characterization}(i) also in $\DM\Bigl(\tmn{n}{m}\Bigr)$.  Hence Theorem~\ref{thm:completion_characterization}(ii) implies 
	\begin{displaymath}
		\DM\Bigl(\JJ\bigl(\tmn{n}{m}\bigr)\cup\MM\bigl(\tmn{n}{m}\bigr)\Bigr)\cong\DM\Bigl(\tmn{n}{m}\Bigr).
	\end{displaymath}
\end{proof}

\begin{example}
  \label{ex:23_tamari_completion}
	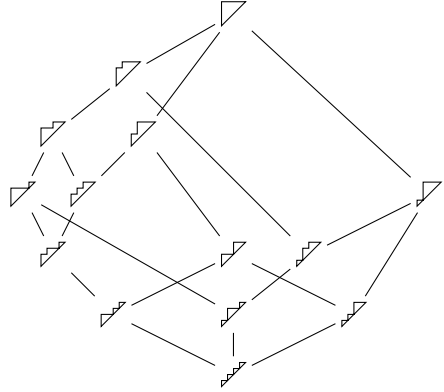
\begin{figure}
		\centering
		\begin{tikzpicture}\tiny
			\def\x{.8};
			\def\y{.8};
			\draw(4.5*\x,1*\y) node(f1){\dyckFour{0/1/2/3}{white}{.4}};
			\draw(2.5*\x,2*\y) node(f2){\dyckFour{0/0/2/3}{white}{.4}};
			\draw(4.5*\x,2*\y) node(f3){\dyckFour{0/1/1/3}{white}{.4}};
			\draw(6.5*\x,2*\y) node(f4){\dyckFour{0/1/2/2}{white}{.4}};
			\draw(1.5*\x,3*\y) node(f5){\dyckFour{0/0/1/3}{white}{.4}};
			\draw(4.5*\x,3*\y) node(f6){\dyckFour{0/0/2/2}{white}{.4}};
			\draw(5.75*\x,3*\y) node(f7){\dyckFour{0/1/1/2}{white}{.4}};
			\draw(1*\x,4*\y) node(f8){\dyckFour{0/0/0/3}{white}{.4}};
			\draw(2*\x,4*\y) node(f9){\dyckFour{0/0/1/2}{white}{.4}};
			\draw(7.75*\x,4*\y) node(f10){\dyckFour{0/1/1/1}{white}{.4}};
			\draw(1.5*\x,5*\y) node(f11){\dyckFour{0/0/0/2}{white}{.4}};
			\draw(3*\x,5*\y) node(f12){\dyckFour{0/0/1/1}{white}{.4}};
			\draw(2.75*\x,6*\y) node(f13){\dyckFour{0/0/0/1}{white}{.4}};
			\draw(4.5*\x,7*\y) node(f14){\dyckFour{0/0/0/0}{white}{.4}};
			\draw(f1) -- (f2);
			\draw(f1) -- (f3);
			\draw(f1) -- (f4);
			\draw(f2) -- (f5);
			\draw(f2) -- (f6);
			\draw(f3) -- (f7);
			\draw(f3) -- (f8);
			\draw(f4) -- (f6);
			\draw(f4) -- (f10);
			\draw(f5) -- (f8);
			\draw(f5) -- (f9);
			\draw(f6) -- (f12);
			\draw(f7) -- (f10);
			\draw(f7) -- (f13);
			\draw(f8) -- (f11);
			\draw(f9) -- (f11);
			\draw(f9) -- (f12);
			\draw(f10) -- (f14);
			\draw(f11) -- (f13);
			\draw(f12) -- (f14);
			\draw(f13) -- (f14);
		\end{tikzpicture}
		\caption{The Tamari lattice $\mtam{4}{1}$.}
		\label{fig:tamari4}
	\end{figure}
	
	Figure~\ref{fig:tamari4} shows the Tamari lattice $\mtam{4}{1}$.  The elements of $\dmn{4}{2}$ are precisely: 
	\begin{align*}
		& \bigl(\raisebox{-.05cm}{\dyckFour{0/1/2/3}{white}{.4},\dyckFour{0/1/2/3}{white}{.4}}\bigr), 
		  && \bigl(\raisebox{-.05cm}{\dyckFour{0/1/2/3}{white}{.4},\dyckFour{0/0/2/3}{white}{.4}}\bigr), 
		  && \bigl(\raisebox{-.05cm}{\dyckFour{0/1/2/3}{white}{.4},\dyckFour{0/1/1/3}{white}{.4}}\bigr),
		  && \bigl(\raisebox{-.05cm}{\dyckFour{0/1/2/3}{white}{.4},\dyckFour{0/1/2/2}{white}{.4}}\bigr),
		  && \bigl(\raisebox{-.05cm}{\dyckFour{0/1/2/3}{white}{.4},\dyckFour{0/0/1/3}{white}{.4}}\bigr),\\ 
		& \bigl(\raisebox{-.05cm}{\dyckFour{0/1/2/3}{white}{.4},\dyckFour{0/0/2/2}{white}{.4}}\bigr),
		  && \bigl(\raisebox{-.05cm}{\dyckFour{0/1/2/3}{white}{.4},\dyckFour{0/1/1/2}{white}{.4}}\bigr), 
		  && \bigl(\raisebox{-.05cm}{\dyckFour{0/1/2/3}{white}{.4},\dyckFour{0/0/0/3}{white}{.4}}\bigr),
		  && \bigl(\raisebox{-.05cm}{\dyckFour{0/1/2/3}{white}{.4},\dyckFour{0/0/1/2}{white}{.4}}\bigr),
		  && \bigl(\raisebox{-.05cm}{\dyckFour{0/1/2/3}{white}{.4},\dyckFour{0/1/1/1}{white}{.4}}\bigr), \\
		& \bigl(\raisebox{-.05cm}{\dyckFour{0/1/2/3}{white}{.4},\dyckFour{0/0/0/2}{white}{.4}}\bigr),
		  && \bigl(\raisebox{-.05cm}{\dyckFour{0/1/2/3}{white}{.4},\dyckFour{0/0/1/1}{white}{.4}}\bigr),
		  && \bigl(\raisebox{-.05cm}{\dyckFour{0/1/2/3}{white}{.4},\dyckFour{0/0/0/1}{white}{.4}}\bigr),
		  && \bigl(\raisebox{-.05cm}{\dyckFour{0/1/2/3}{white}{.4},\dyckFour{0/0/0/0}{white}{.4}}\bigr),
		  && \bigl(\raisebox{-.05cm}{\dyckFour{0/0/2/3}{white}{.4},\dyckFour{0/0/2/3}{white}{.4}}\bigr),\\
		& \bigl(\raisebox{-.05cm}{\dyckFour{0/0/2/3}{white}{.4},\dyckFour{0/0/1/3}{white}{.4}}\bigr),
		  && \bigl(\raisebox{-.05cm}{\dyckFour{0/0/2/3}{white}{.4},\dyckFour{0/0/2/2}{white}{.4}}\bigr), 
		  && \bigl(\raisebox{-.05cm}{\dyckFour{0/1/1/3}{white}{.4},\dyckFour{0/1/1/3}{white}{.4}}\bigr),
		  && \bigl(\raisebox{-.05cm}{\dyckFour{0/1/1/3}{white}{.4},\dyckFour{0/1/1/2}{white}{.4}}\bigr), 
		  && \bigl(\raisebox{-.05cm}{\dyckFour{0/1/1/3}{white}{.4},\dyckFour{0/0/0/3}{white}{.4}}\bigr),\\
		& \bigl(\raisebox{-.05cm}{\dyckFour{0/1/2/2}{white}{.4},\dyckFour{0/1/2/2}{white}{.4}}\bigr),
		  && \bigl(\raisebox{-.05cm}{\dyckFour{0/1/2/2}{white}{.4},\dyckFour{0/0/2/2}{white}{.4}}\bigr), 
		  && \bigl(\raisebox{-.05cm}{\dyckFour{0/1/2/2}{white}{.4},\dyckFour{0/1/1/1}{white}{.4}}\bigr), 
		  && \bigl(\raisebox{-.05cm}{\dyckFour{0/0/1/3}{white}{.4},\dyckFour{0/0/1/3}{white}{.4}}\bigr),
		  && \bigl(\raisebox{-.05cm}{\dyckFour{0/0/1/3}{white}{.4},\dyckFour{0/0/0/3}{white}{.4}}\bigr), \\
		& \bigl(\raisebox{-.05cm}{\dyckFour{0/0/1/3}{white}{.4},\dyckFour{0/0/1/2}{white}{.4}}\bigr), 
		  && \bigl(\raisebox{-.05cm}{\dyckFour{0/0/2/2}{white}{.4},\dyckFour{0/0/2/2}{white}{.4}}\bigr),
		  && \bigl(\raisebox{-.05cm}{\dyckFour{0/0/2/2}{white}{.4},\dyckFour{0/0/1/1}{white}{.4}}\bigr),
		  && \bigl(\raisebox{-.05cm}{\dyckFour{0/1/1/2}{white}{.4},\dyckFour{0/1/1/2}{white}{.4}}\bigr), 
		  && \bigl(\raisebox{-.05cm}{\dyckFour{0/1/1/2}{white}{.4},\dyckFour{0/1/1/1}{white}{.4}}\bigr),\\
		& \bigl(\raisebox{-.05cm}{\dyckFour{0/1/1/2}{white}{.4},\dyckFour{0/0/0/1}{white}{.4}}\bigr), 
		  && \bigl(\raisebox{-.05cm}{\dyckFour{0/0/0/3}{white}{.4},\dyckFour{0/0/0/3}{white}{.4}}\bigr),
		  && \bigl(\raisebox{-.05cm}{\dyckFour{0/0/0/3}{white}{.4},\dyckFour{0/0/0/2}{white}{.4}}\bigr),
		  && \bigl(\raisebox{-.05cm}{\dyckFour{0/0/1/2}{white}{.4},\dyckFour{0/0/1/2}{white}{.4}}\bigr), 
		  && \bigl(\raisebox{-.05cm}{\dyckFour{0/0/1/2}{white}{.4},\dyckFour{0/0/0/2}{white}{.4}}\bigr), \\
		& \bigl(\raisebox{-.05cm}{\dyckFour{0/0/1/2}{white}{.4},\dyckFour{0/0/1/1}{white}{.4}}\bigr),
		  && \bigl(\raisebox{-.05cm}{\dyckFour{0/1/1/1}{white}{.4},\dyckFour{0/1/1/1}{white}{.4}}\bigr), 
		  && \bigl(\raisebox{-.05cm}{\dyckFour{0/1/1/1}{white}{.4},\dyckFour{0/0/0/0}{white}{.4}}\bigr), 
		  && \bigl(\raisebox{-.05cm}{\dyckFour{0/0/0/2}{white}{.4},\dyckFour{0/0/0/2}{white}{.4}}\bigr),
		  && \bigl(\raisebox{-.05cm}{\dyckFour{0/0/0/2}{white}{.4},\dyckFour{0/0/0/1}{white}{.4}}\bigr),\\ 
		& \bigl(\raisebox{-.05cm}{\dyckFour{0/0/1/1}{white}{.4},\dyckFour{0/0/1/1}{white}{.4}}\bigr), 
		  && \bigl(\raisebox{-.05cm}{\dyckFour{0/0/1/1}{white}{.4},\dyckFour{0/0/0/0}{white}{.4}}\bigr),
		  && \bigl(\raisebox{-.05cm}{\dyckFour{0/0/0/1}{white}{.4},\dyckFour{0/0/0/1}{white}{.4}}\bigr), 
		  && \bigl(\raisebox{-.05cm}{\dyckFour{0/0/0/1}{white}{.4},\dyckFour{0/0/0/0}{white}{.4}}\bigr),
		  && \bigl(\raisebox{-.05cm}{\dyckFour{0/0/0/0}{white}{.4},\dyckFour{0/0/0/0}{white}{.4}}\bigr).
	\end{align*}
	Now we notice that for instance the pairs $\bigl(\raisebox{-.05cm}{\dyckFour{0/0/2/2}{white}{.4},\dyckFour{0/0/0/2}{white}{.4}}\bigr)$ and $\bigl(\raisebox{-.05cm}{\dyckFour{0/0/1/3}{white}{.4},\dyckFour{0/0/1/2}{white}{.4}}\bigr)$ do not have a meet in $\tmn{4}{2}$, since 
	\begin{displaymath}
		\bigl(\raisebox{-.05cm}{\dyckFour{0/1/2/3}{white}{.4},\dyckFour{0/0/1/2}{white}{.4}}\bigr)\leq_{\text{rot}}
		  \bigl(\raisebox{-.05cm}{\dyckFour{0/0/2/2}{white}{.4},\dyckFour{0/0/0/2}{white}{.4}}\bigr),
		  \bigl(\raisebox{-.05cm}{\dyckFour{0/0/1/3}{white}{.4},\dyckFour{0/0/1/2}{white}{.4}}\bigr),
		  \quad\text{and}\quad
		\bigl(\raisebox{-.05cm}{\dyckFour{0/0/2/3}{white}{.4},\dyckFour{0/0/1/3}{white}{.4}}\bigr)\leq_{\text{rot}}
		  \bigl(\raisebox{-.05cm}{\dyckFour{0/0/2/2}{white}{.4},\dyckFour{0/0/0/2}{white}{.4}}\bigr),
		  \bigl(\raisebox{-.05cm}{\dyckFour{0/0/1/3}{white}{.4},\dyckFour{0/0/1/2}{white}{.4}}\bigr),
	\end{displaymath}
	but $\bigl(\raisebox{-.05cm}{\dyckFour{0/1/2/3}{white}{.4},\dyckFour{0/0/1/2}{white}{.4}}\bigr)$ and $\bigl(\raisebox{-.05cm}{\dyckFour{0/0/2/3}{white}{.4},\dyckFour{0/0/1/3}{white}{.4}}\bigr)$ are mutually incomparable. 
	
	The componentwise meet of $\bigl(\raisebox{-.05cm}{\dyckFour{0/0/2/2}{white}{.4},\dyckFour{0/0/0/2}{white}{.4}}\bigr)$ and $\bigl(\raisebox{-.05cm}{\dyckFour{0/0/1/3}{white}{.4},\dyckFour{0/0/1/2}{white}{.4}}\bigr)$ is $\bigl(\raisebox{-.05cm}{\dyckFour{0/0/2/3}{white}{.4},\dyckFour{0/0/1/2}{white}{.4}}\bigr)$.  If we now successively add all the missing meets, then we can check that we have to include the following ten elements:
	\begin{align*}
		& \bigl(\raisebox{-.05cm}{\dyckFour{0/0/1/3}{white}{.4},\dyckFour{0/0/0/2}{white}{.4}}\bigr), 
		  && \bigl(\raisebox{-.05cm}{\dyckFour{0/1/1/3}{white}{.4},\dyckFour{0/0/0/2}{white}{.4}}\bigr), 
		  && \bigl(\raisebox{-.05cm}{\dyckFour{0/0/2/3}{white}{.4},\dyckFour{0/0/1/2}{white}{.4}}\bigr),
		  && \bigl(\raisebox{-.05cm}{\dyckFour{0/0/2/3}{white}{.4},\dyckFour{0/0/1/1}{white}{.4}}\bigr),
		  && \bigl(\raisebox{-.05cm}{\dyckFour{0/0/1/2}{white}{.4},\dyckFour{0/0/0/1}{white}{.4}}\bigr),\\
		& \bigl(\raisebox{-.05cm}{\dyckFour{0/1/1/3}{white}{.4},\dyckFour{0/0/0/1}{white}{.4}}\bigr),
		  && \bigl(\raisebox{-.05cm}{\dyckFour{0/1/2/2}{white}{.4},\dyckFour{0/0/1/1}{white}{.4}}\bigr), 
		  && \bigl(\raisebox{-.05cm}{\dyckFour{0/1/2/2}{white}{.4},\dyckFour{0/0/0/0}{white}{.4}}\bigr),
		  && \bigl(\raisebox{-.05cm}{\dyckFour{0/0/1/2}{white}{.4},\dyckFour{0/0/0/0}{white}{.4}}\bigr), 
		  && \bigl(\raisebox{-.05cm}{\dyckFour{0/1/1/2}{white}{.4},\dyckFour{0/0/0/0}{white}{.4}}\bigr),
	\end{align*}
	and these $55$ elements form a lattice which is indeed isomorphic to $\mtam{4}{2}$, see Figure~\ref{fig:dm_2tam4}.
	
	\begin{figure}
		\centering
		\begin{tikzpicture}\tiny
			\def\x{.77};
			\def\y{1.2};
			\def\d{-.1cm}
			\draw(7.67*\x,1.33*\y) node(n0){\dyckFour{0/1/2/3}{white}{.4}\hspace*{\d}\dyckFour{0/1/2/3}{white}{.4}};
			\draw(7.67*\x,2.33*\y) node(n1){\dyckFour{0/1/2/3}{white}{.4}\hspace*{\d}\dyckFour{0/0/2/3}{white}{.4}};
			\draw(3.67*\x,3.33*\y) node(n2){\dyckFour{0/1/2/3}{white}{.4}\hspace*{\d}\dyckFour{0/1/1/3}{white}{.4}};
			\draw(5.67*\x,3.33*\y) node(n3){\dyckFour{0/1/2/3}{white}{.4}\hspace*{\d}\dyckFour{0/0/1/3}{white}{.4}};
			\draw(7.67*\x,3.33*\y) node(n4){\dyckFour{0/0/2/3}{white}{.4}\hspace*{\d}\dyckFour{0/0/2/3}{white}{.4}};
			\draw(2*\x,4*\y) node(n5){\dyckFour{0/1/1/3}{white}{.4}\hspace*{\d}\dyckFour{0/1/1/3}{white}{.4}};
			\draw(5.67*\x,4.33*\y) node(n6){\dyckFour{0/0/2/3}{white}{.4}\hspace*{-.2cm}\dyckFour{0/0/1/3}{white}{.4}};
			\draw(7.67*\x,4.33*\y) node(n7){\dyckFour{0/1/2/3}{white}{.4}\hspace*{\d}\dyckFour{0/0/1/2}{white}{.4}};
			\draw(3*\x,5*\y) node(n8){\dyckFour{0/1/2/3}{white}{.4}\hspace*{\d}\dyckFour{0/0/0/3}{white}{.4}};
			\draw(5*\x,5*\y) node(n9){\dyckFour{0/0/1/3}{white}{.4}\hspace*{\d}\dyckFour{0/0/1/3}{white}{.4}};
			\draw(7.67*\x,5.33*\y) node(n10){\dyckFour{0/0/2/3}{orange!80!gray}{.4}\hspace*{\d}\dyckFour{0/0/1/2}{orange!80!gray}{.4}};
			\draw(1*\x,6*\y) node(n11){\dyckFour{0/1/1/3}{white}{.4}\hspace*{\d}\dyckFour{0/0/0/3}{white}{.4}};
			\draw(3*\x,6*\y) node(n12){\dyckFour{0/0/1/3}{white}{.4}\hspace*{\d}\dyckFour{0/0/0/3}{white}{.4}};
			\draw(5*\x,6*\y) node(n13){\dyckFour{0/1/2/3}{white}{.4}\hspace*{\d}\dyckFour{0/0/0/2}{white}{.4}};
			\draw(7*\x,6*\y) node(n14){\dyckFour{0/0/1/3}{white}{.4}\hspace*{\d}\dyckFour{0/0/1/2}{white}{.4}};
			\draw(8.67*\x,6.33*\y) node(n15){\dyckFour{0/1/2/3}{white}{.4}\hspace*{\d}\dyckFour{0/1/1/2}{white}{.4}};
			\draw(16.67*\x,6*\y) node(n16){\dyckFour{0/1/2/3}{white}{.4}\hspace*{\d}\dyckFour{0/1/2/2}{white}{.4}};
			\draw(1*\x,7*\y) node(n17){\dyckFour{0/0/0/3}{white}{.4}\hspace*{\d}\dyckFour{0/0/0/3}{white}{.4}};
			\draw(3*\x,7*\y) node(n18){\dyckFour{0/1/1/3}{orange!80!gray}{.4}\hspace*{\d}\dyckFour{0/0/0/2}{orange!80!gray}{.4}};
			\draw(5*\x,7*\y) node(n19){\dyckFour{0/0/1/3}{orange!80!gray}{.4}\hspace*{\d}\dyckFour{0/0/0/2}{orange!80!gray}{.4}};
			\draw(7*\x,7*\y) node(n20){\dyckFour{0/1/1/3}{white}{.4}\hspace*{-.2cm}\dyckFour{0/1/1/2}{white}{.4}};
			\draw(8*\x,7*\y) node(n21){\dyckFour{0/0/1/2}{white}{.4}\hspace*{-.2cm}\dyckFour{0/0/1/2}{white}{.4}};
			\draw(16.67*\x,7*\y) node(n22){\dyckFour{0/1/2/3}{white}{.4}\hspace*{\d}\dyckFour{0/0/2/2}{white}{.4}};
			\draw(18.67*\x,7*\y) node(n23){\dyckFour{0/1/2/2}{white}{.4}\hspace*{\d}\dyckFour{0/1/2/2}{white}{.4}};
			\draw(3*\x,8*\y) node(n24){\dyckFour{0/0/0/3}{white}{.4}\hspace*{\d}\dyckFour{0/0/0/2}{white}{.4}};
			\draw(6*\x,8*\y) node(n25){\dyckFour{0/0/1/2}{white}{.4}\hspace*{\d}\dyckFour{0/0/0/2}{white}{.4}};
			\draw(8*\x,8*\y) node(n26){\dyckFour{0/1/2/3}{white}{.4}\hspace*{\d}\dyckFour{0/0/0/1}{white}{.4}};
			\draw(9*\x,8*\y) node(n27){\dyckFour{0/1/1/2}{white}{.4}\hspace*{\d}\dyckFour{0/1/1/2}{white}{.4}};
			\draw(11.67*\x,8*\y) node(n28){\dyckFour{0/1/2/3}{white}{.4}\hspace*{\d}\dyckFour{0/1/1/1}{white}{.4}};
			\draw(14.67*\x,8*\y) node(n29){\dyckFour{0/1/2/3}{white}{.4}\hspace*{\d}\dyckFour{0/0/1/1}{white}{.4}};
			\draw(16.67*\x,8*\y) node(n30){\dyckFour{0/0/2/3}{white}{.4}\hspace*{\d}\dyckFour{0/0/2/2}{white}{.4}};
			\draw(18.67*\x,8*\y) node(n31){\dyckFour{0/1/2/2}{white}{.4}\hspace*{\d}\dyckFour{0/0/2/2}{white}{.4}};
			\draw(4*\x,9*\y) node(n32){\dyckFour{0/0/0/2}{white}{.4}\hspace*{\d}\dyckFour{0/0/0/2}{white}{.4}};
			\draw(6*\x,9*\y) node(n33){\dyckFour{0/1/1/3}{orange!80!gray}{.4}\hspace*{\d}\dyckFour{0/0/0/1}{orange!80!gray}{.4}};
			\draw(8*\x,9*\y) node(n34){\dyckFour{0/0/1/2}{orange!80!gray}{.4}\hspace*{\d}\dyckFour{0/0/0/1}{orange!80!gray}{.4}};
			\draw(10.33*\x,8.67*\y) node(n35){\dyckFour{0/1/1/2}{white}{.4}\hspace*{\d}\dyckFour{0/1/1/1}{white}{.4}};
			\draw(13.67*\x,9*\y) node(n36){\dyckFour{0/1/2/2}{white}{.4}\hspace*{\d}\dyckFour{0/1/1/1}{white}{.4}};
			\draw(14.67*\x,9*\y) node(n37){\dyckFour{0/0/2/3}{orange!80!gray}{.4}\hspace*{\d}\dyckFour{0/0/1/1}{orange!80!gray}{.4}};
			\draw(16.67*\x,9*\y) node(n38){\dyckFour{0/1/2/2}{orange!80!gray}{.4}\hspace*{\d}\dyckFour{0/0/1/1}{orange!80!gray}{.4}};
			\draw(18.67*\x,9*\y) node(n39){\dyckFour{0/0/2/2}{white}{.4}\hspace*{\d}\dyckFour{0/0/2/2}{white}{.4}};
			\draw(6*\x,10*\y) node(n40){\dyckFour{0/0/0/2}{white}{.4}\hspace*{\d}\dyckFour{0/0/0/1}{white}{.4}};
			\draw(8*\x,10*\y) node(n41){\dyckFour{0/1/1/2}{white}{.4}\hspace*{\d}\dyckFour{0/0/0/1}{white}{.4}};
			\draw(11.33*\x,9.67*\y) node(n42){\dyckFour{0/1/2/3}{white}{.4}\hspace*{-.05cm}\dyckFour{0/0/0/0}{white}{.4}};
			\draw(12.33*\x,9.67*\y) node(n43){\dyckFour{0/1/1/1}{white}{.4}\hspace*{\d}\dyckFour{0/1/1/1}{white}{.4}};
			\draw(13.33*\x,9.67*\y) node(n44){\dyckFour{0/0/1/2}{white}{.4}\hspace*{\d}\dyckFour{0/0/1/1}{white}{.4}};
			\draw(16.67*\x,10*\y) node(n45){\dyckFour{0/0/2/2}{white}{.4}\hspace*{\d}\dyckFour{0/0/1/1}{white}{.4}};
			\draw(8*\x,11*\y) node(n46){\dyckFour{0/0/0/1}{white}{.4}\hspace*{\d}\dyckFour{0/0/0/1}{white}{.4}};
			\draw(9.33*\x,10.67*\y) node(n47){\dyckFour{0/1/1/2}{orange!80!gray}{.4}\hspace*{-.05cm}\dyckFour{0/0/0/0}{orange!80!gray}{.4}};
			\draw(11.33*\x,10.67*\y) node(n48){\dyckFour{0/0/1/2}{orange!80!gray}{.4}\hspace*{-.05cm}\dyckFour{0/0/0/0}{orange!80!gray}{.4}};
			\draw(13.33*\x,10.67*\y) node(n49){\dyckFour{0/1/2/2}{orange!80!gray}{.4}\hspace*{-.05cm}\dyckFour{0/0/0/0}{orange!80!gray}{.4}};
			\draw(15.33*\x,10.67*\y) node(n50){\dyckFour{0/0/1/1}{white}{.4}\hspace*{\d}\dyckFour{0/0/1/1}{white}{.4}};
			\draw(9.33*\x,11.67*\y) node(n51){\dyckFour{0/0/0/1}{white}{.4}\hspace*{-.05cm}\dyckFour{0/0/0/0}{white}{.4}};
			\draw(11.33*\x,11.67*\y) node(n52){\dyckFour{0/1/1/1}{white}{.4}\hspace*{-.05cm}\dyckFour{0/0/0/0}{white}{.4}};
			\draw(13.33*\x,11.67*\y) node(n53){\dyckFour{0/0/1/1}{white}{.4}\hspace*{-.05cm}\dyckFour{0/0/0/0}{white}{.4}};
			\draw(11.33*\x,12.67*\y) node(n54){\dyckFour{0/0/0/0}{white}{.4}\hspace*{-.05cm}\dyckFour{0/0/0/0}{white}{.4}};
			\draw(n0) -- (n1);
			\draw(n0) -- (n2);
			\draw(n0) -- (n16);
			\draw(n1) -- (n3);
			\draw(n1) -- (n4);
			\draw(n1) -- (n22);
			\draw(n2) -- (n5);
			\draw(n2) -- (n8);
			\draw(n2) -- (n15);
			\draw(n3) -- (n6);
			\draw(n3) -- (n7);
			\draw(n3) -- (n8);
			\draw(n4) -- (n6);
			\draw(n4) -- (n30);
			\draw(n5) -- (n11);
			\draw(n5) -- (n20);
			\draw(n6) -- (n9);
			\draw(n6) -- (n10);
			\draw(n7) -- (n10);
			\draw(n7) -- (n13);
			\draw(n7) -- (n29);
			\draw(n8) -- (n11);
			\draw(n8) -- (n12);
			\draw(n8) -- (n13);
			\draw(n9) -- (n12);
			\draw(n9) -- (n14);
			\draw(n10) -- (n14);
			\draw(n10) -- (n37);
			\draw(n11) -- (n17);
			\draw(n11) -- (n18);
			\draw(n12) -- (n17);
			\draw(n12) -- (n19);
			\draw(n13) -- (n18);
			\draw(n13) -- (n19);
			\draw(n13) -- (n26);
			\draw(n14) -- (n19);
			\draw(n14) -- (n21);
			\draw(n15) -- (n20);
			\draw(n15) -- (n26);
			\draw(n15) -- (n28);
			\draw(n16) -- (n22);
			\draw(n16) -- (n23);
			\draw(n16) -- (n28);
			\draw(n17) -- (n24);
			\draw(n18) -- (n24);
			\draw(n18) -- (n33);
			\draw(n19) -- (n24);
			\draw(n19) -- (n25);
			\draw(n20) -- (n27);
			\draw(n20) -- (n33);
			\draw(n21) -- (n25);
			\draw(n21) -- (n44);
			\draw(n22) -- (n29);
			\draw(n22) -- (n30);
			\draw(n22) -- (n31);
			\draw(n23) -- (n31);
			\draw(n23) -- (n36);
			\draw(n24) -- (n32);
			\draw(n25) -- (n32);
			\draw(n25) -- (n34);
			\draw(n26) -- (n33);
			\draw(n26) -- (n34);
			\draw(n26) -- (n42);
			\draw(n27) -- (n35);
			\draw(n27) -- (n41);
			\draw(n28) -- (n35);
			\draw(n28) -- (n36);
			\draw(n28) -- (n42);
			\draw(n29) -- (n37);
			\draw(n29) -- (n38);
			\draw(n29) -- (n42);
			\draw(n30) -- (n37);
			\draw(n30) -- (n39);
			\draw(n31) -- (n38);
			\draw(n31) -- (n39);
			\draw(n32) -- (n40);
			\draw(n33) -- (n40);
			\draw(n33) -- (n41);
			\draw(n34) -- (n40);
			\draw(n34) -- (n48);
			\draw(n35) -- (n43);
			\draw(n35) -- (n47);
			\draw(n36) -- (n43);
			\draw(n36) -- (n49);
			\draw(n37) -- (n44);
			\draw(n37) -- (n45);
			\draw(n38) -- (n45);
			\draw(n38) -- (n49);
			\draw(n39) -- (n45);
			\draw(n40) -- (n46);
			\draw(n41) -- (n46);
			\draw(n41) -- (n47);
			\draw(n42) -- (n48);
			\draw(n42) -- (n49);
			\draw(n43) -- (n52);
			\draw(n44) -- (n48);
			\draw(n44) -- (n50);
			\draw(n45) -- (n50);
			\draw(n46) -- (n51);
			\draw(n47) -- (n51);
			\draw(n47) -- (n52);
			\draw(n48) -- (n51);
			\draw(n48) -- (n53);
			\draw(n49) -- (n52);
			\draw(n49) -- (n53);
			\draw(n50) -- (n53);
			\draw(n51) -- (n54);
			\draw(n52) -- (n54);
			\draw(n53) -- (n54);
		\end{tikzpicture}
		\caption{The lattice $\DM\bigl(\tmn{4}{2}\bigr)$.  The highlighted elements are added during the lattice completion.}
		\label{fig:dm_2tam4}
	\end{figure}
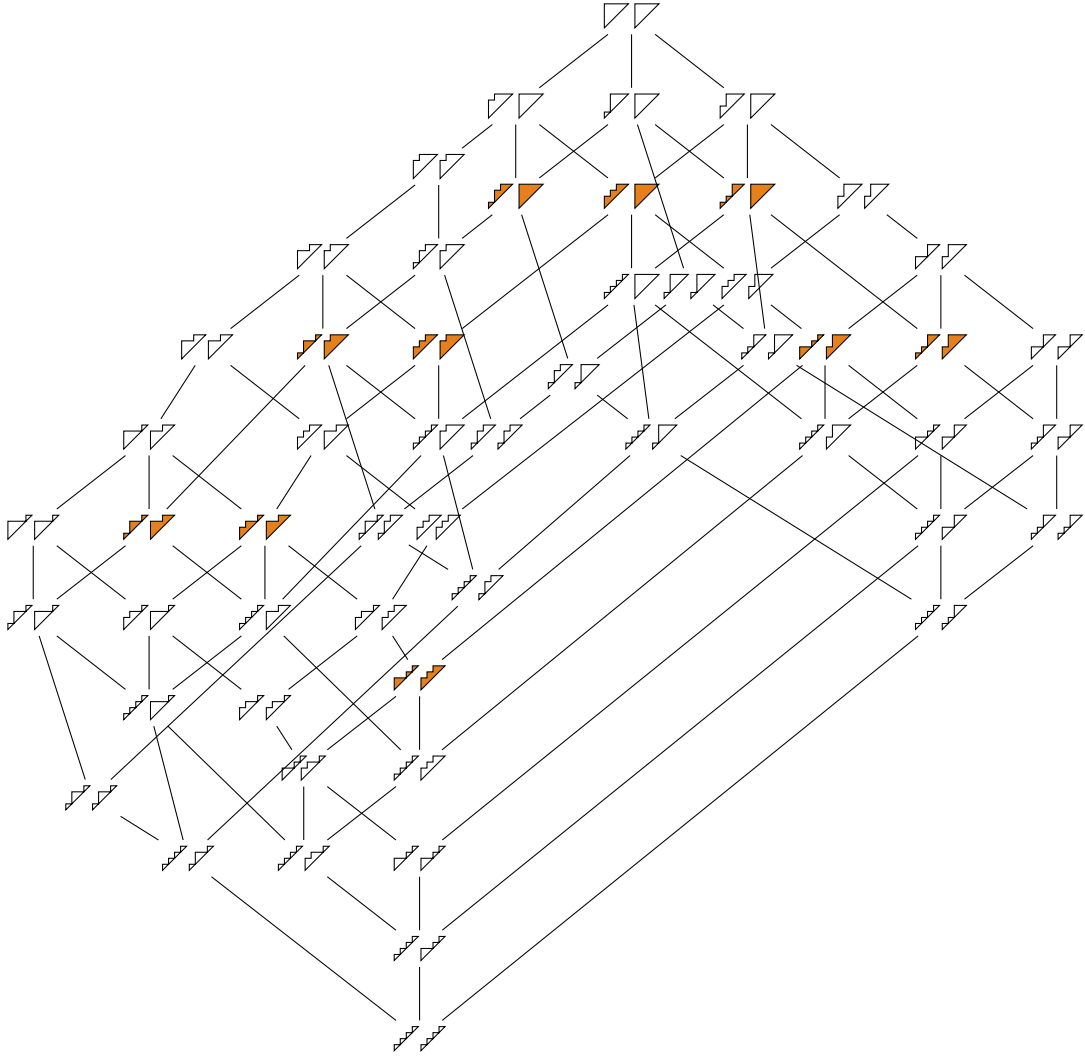
\end{example}

\subsection{Further Properties of the Strip Decomposition of $m$-Dyck Paths}
  \label{sec:strip_decomposition}
Now let $\qf,\qf'\in D_n$ have height sequences $\hh_{\qf}=(h_{1},h_{2},\ldots,h_{n})$, and $\hh_{\qf'}=(h'_{1},h'_{2},\ldots,h'_{n})$.  We say that \alert{$\qf'$ dominates $\qf$} if and only if $h_{i}\leq h'_{i}$ for all $i\in\{1,2,\ldots,n\}$, and we denote it by $\qf\leq_{\text{dom}}\qf'$.  (In other words, the two Dyck paths $\qf$ and $\qf'$ are noncrossing, but might share common edges.)  The partial order $\leq_{\text{dom}}$ on is called the \alert{dominance order} on $D_n$.  An \alert{$m$-fan of Dyck paths} is a tuple $(\qf_{1},\qf_{2},\ldots,\qf_{m})$ with $\qf_{i}\in D_n$ for $i\in\{1,2,\ldots,m\}$ such that the paths $\qf_{i}$ are pairwise noncrossing.  We say that an $m$-fan of Dyck paths $(\qf_{1},\qf_{2},\ldots,\qf_{m})$ is \alert{increasing} if $\qf_{1}\leq_{\text{dom}}\qf_{2}\leq_{\text{dom}}\cdots\leq_{\text{dom}}\qf_{m}$.  The following result gives an enumeration formula for the set of increasing $m$-fans of Dyck paths of length $2n$. 

\begin{theorem}[\cite{jonsson05generalized}*{Corollary~17},\cite{krattenthaler06growth}]\label{thm:increasing_fans}
	The number of increasing $m$-fans of Dyck paths of length $2n$ is given by 
	\begin{displaymath}
		\prod_{1\leq i\leq j<n}\frac{i+j+2m}{i+j}.
	\end{displaymath}
\end{theorem}

\begin{lemma}\label{lem:decomposition_increasing}
	If $\pf\in D_{n}^{(m)}$, then $\delta(\pf)$ forms an increasing $m$-fan of Dyck paths.
\end{lemma}
\begin{proof}
	Let $\pf\in D_{n}^{(m)}$ with height sequence $\hh_{\pf}=(h_{1},h_{2},\ldots,h_{mn})$, and let $\delta(\pf)=(\qf_{1},\qf_{2},\ldots,\qf_{m})$.  For $i\in\{1,2,\ldots,m\}$, let $\hh_{\qf_{i}}=(h_{1}^{(i)},h_{2}^{(i)},\ldots,h_{n}^{(i)})$ denote the height sequence associated with $\qf_{i}$.  By construction, we have $h_{k}^{(i)}=h_{(i-1)m+k}$ for $k\in\{1,2,\ldots,n\}$ and $i\in\{1,2,\ldots,m\}$.  Equation \eqref{eq:height_1} implies that $h_{1}\leq h_{2}\leq\cdots\leq h_{mn}$.  Thus for $i,j\in\{1,2,\ldots,m\}$ with $i<j$ and $k\in\{1,2,\ldots,n\}$ we have $h_{k}^{(i)}=h_{(i-1)m+k}\leq h_{(j-1)m+k}=h_{k}^{(j)}$.  Thus $\qf_{j}$ dominates $\qf_{i}$, which implies that $\delta(\pf)=(\qf_{1},\qf_{2},\ldots,\qf_{m})$ is an increasing $m$-fan of Dyck paths.
\end{proof}

\begin{example}\label{ex:dyck23_2}
	\begin{figure}
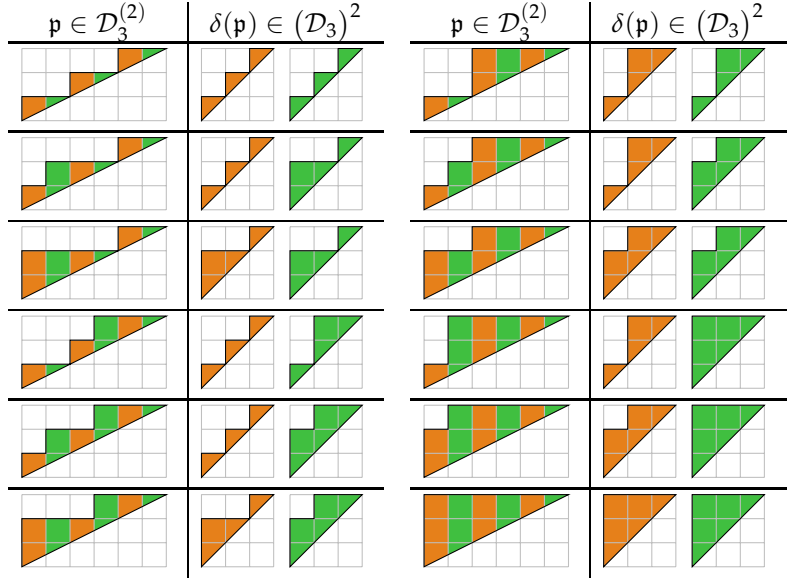

		\centering
		\begin{tabular}{c|ccc|c}
			$\pf\in\mathcal{D}_{3}^{(2)}$ & $\delta(\pf)\in\bigl(\mathcal{D}_{3}\bigr)^{2}$ & & $\pf\in\mathcal{D}_{3}^{(2)}$ & $\delta(\pf)\in\bigl(\mathcal{D}_{3}\bigr)^{2}$\\
			\cline{1-2}\cline{4-5}
			\dyckTwoThree{0/2/4}{\cOne}{\cTwo}{.8}{1} & \dyckThree{0/1/2}{\cOne}{.8}{1}\dyckThree{0/1/2}{\cTwo}{.8}{1} 
			& & \dyckTwoThree{0/2/2}{\cOne}{\cTwo}{.8}{1} & \dyckThree{0/1/1}{\cOne}{.8}{1}\dyckThree{0/1/1}{\cTwo}{.8}{1} \\
			\cline{1-2}\cline{4-5}
			\dyckTwoThree{0/1/4}{\cOne}{\cTwo}{.8}{1} & \dyckThree{0/1/2}{\cOne}{.8}{1}\dyckThree{0/0/2}{\cTwo}{.8}{1} 
			& & \dyckTwoThree{0/1/2}{\cOne}{\cTwo}{.8}{1} & \dyckThree{0/1/1}{\cOne}{.8}{1}\dyckThree{0/0/1}{\cTwo}{.8}{1} \\
			\cline{1-2}\cline{4-5}
			\dyckTwoThree{0/0/4}{\cOne}{\cTwo}{.8}{1} & \dyckThree{0/0/2}{\cOne}{.8}{1}\dyckThree{0/0/2}{\cTwo}{.8}{1} 
			& & \dyckTwoThree{0/0/2}{\cOne}{\cTwo}{.8}{1} & \dyckThree{0/0/1}{\cOne}{.8}{1}\dyckThree{0/0/1}{\cTwo}{.8}{1} \\
			\cline{1-2}\cline{4-5}
			\dyckTwoThree{0/2/3}{\cOne}{\cTwo}{.8}{1} & \dyckThree{0/1/2}{\cOne}{.8}{1}\dyckThree{0/1/1}{\cTwo}{.8}{1} 
			& & \dyckTwoThree{0/1/1}{\cOne}{\cTwo}{.8}{1} & \dyckThree{0/1/1}{\cOne}{.8}{1}\dyckThree{0/0/0}{\cTwo}{.8}{1} \\
			\cline{1-2}\cline{4-5}
			\dyckTwoThree{0/1/3}{\cOne}{\cTwo}{.8}{1} & \dyckThree{0/1/2}{\cOne}{.8}{1}\dyckThree{0/0/1}{\cTwo}{.8}{1} 
			& & \dyckTwoThree{0/0/1}{\cOne}{\cTwo}{.8}{1} & \dyckThree{0/0/1}{\cOne}{.8}{1}\dyckThree{0/0/0}{\cTwo}{.8}{1} \\
			\cline{1-2}\cline{4-5}
			\dyckTwoThree{0/0/3}{\cOne}{\cTwo}{.8}{1} & \dyckThree{0/0/2}{\cOne}{.8}{1}\dyckThree{0/0/1}{\cTwo}{.8}{1} 
			& & \dyckTwoThree{0/0/0}{\cOne}{\cTwo}{.8}{1} & \dyckThree{0/0/0}{\cOne}{.8}{1}\dyckThree{0/0/0}{\cTwo}{.8}{1} \\
		\end{tabular}
		\caption{The twelve $2$-Dyck paths of length $9$ and the corresponding (increasing) $2$-fans of Dyck paths of length $6$.}
		\label{fig:decomposition_dyck23}
	\end{figure}
 
	Figure~\ref{fig:decomposition_dyck23} shows the twelve $2$-Dyck paths of length $9$ and the corresponding increasing $2$-fans of Dyck paths of length $6$.  However, Theorem~\ref{thm:increasing_fans} implies that there are fourteen increasing $2$-fans of Dyck paths of length $6$, and the two increasing $2$-fans of Dyck paths of length $6$ that are missing in Figure~\ref{fig:decomposition_dyck23} are shown in Figure~\ref{fig:invalid_decomposition_dyck23}.
	
	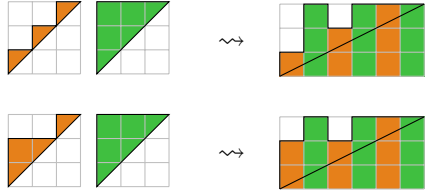
\begin{figure}
		\centering
		\begin{tikzpicture}
			\draw(0,0) node{\dyckThree{0/1/2}{\cOne}{.8}{1}\dyckThree{0/0/0}{\cTwo}{.8}{1}};
			\draw(1.9,-.1) node{$\leadsto$};
			\draw(3.5,0) node{
				\begin{tikzpicture}[scale=.8]\tiny
					\def\r{.4};
					\fill[\cOne](0,0) -- (1*\r,0) -- (1*\r,1*\r) -- (0,1*\r) -- cycle;
					\fill[\cTwo](1*\r,0) -- (2*\r,0) -- (2*\r,3*\r) -- (1*\r,3*\r) -- cycle;
					\fill[\cOne](2*\r,0) -- (3*\r,0) -- (3*\r,2*\r) -- (2*\r,2*\r) -- cycle;
					\fill[\cTwo](3*\r,0) -- (4*\r,0) -- (4*\r,3*\r) -- (3*\r,3*\r) -- cycle;
					\fill[\cOne](4*\r,0) -- (5*\r,0) -- (5*\r,3*\r) -- (4*\r,3*\r) -- cycle;
					\fill[\cTwo](5*\r,0) -- (6*\r,0) -- (6*\r,3*\r) -- (5*\r,3*\r) -- cycle;
					\draw[white!50!gray](0,0) grid[step=.4] (2.4,1.2);
					\draw(0,0) -- (2.4,1.2);
					\draw(1.2,1.3) node{};
					\draw(0,0) -- (0,.4) -- (.4,.4) -- (.4,1.2) -- (.8,1.2) -- (.8,.8) 
					  -- (1.2,.8) -- (1.2,1.2) -- (2.4,1.2);
				\end{tikzpicture}};
			\draw(0,-1.5) node{\dyckThree{0/0/2}{\cOne}{.8}{1}\dyckThree{0/0/0}{\cTwo}{.8}{1}};
			\draw(1.9,-1.6) node{$\leadsto$};
			\draw(3.5,-1.5) node{
				\begin{tikzpicture}[scale=.8]\tiny
					\def\r{.4}
					\fill[\cOne](0,0) -- (1*\r,0) -- (1*\r,2*\r) -- (0,2*\r) -- cycle;
					\fill[\cTwo](1*\r,0) -- (2*\r,0) -- (2*\r,3*\r) -- (1*\r,3*\r) -- cycle;
					\fill[\cOne](2*\r,0) -- (3*\r,0) -- (3*\r,2*\r) -- (2*\r,2*\r) -- cycle;
					\fill[\cTwo](3*\r,0) -- (4*\r,0) -- (4*\r,3*\r) -- (3*\r,3*\r) -- cycle;
					\fill[\cOne](4*\r,0) -- (5*\r,0) -- (5*\r,3*\r) -- (4*\r,3*\r) -- cycle;
					\fill[\cTwo](5*\r,0) -- (6*\r,0) -- (6*\r,3*\r) -- (5*\r,3*\r) -- cycle;
					\draw[white!50!gray](0,0) grid[step=.4] (2.4,1.2);
					\draw(0,0) -- (2.4,1.2);
					\draw(1.2,1.3) node{};
					\draw(0,0) -- (0,.8) -- (.4,.8) -- (.4,1.2) -- (.8,1.2) -- (.8,.8) 
					  -- (1.2,.8) -- (1.2,1.2) -- (2.4,1.2);
				\end{tikzpicture}};
		\end{tikzpicture}
		\caption{The two increasing $2$-fans of Dyck paths of length $6$ that do not produce a valid $2$-Dyck path of length $9$.}
		\label{fig:invalid_decomposition_dyck23}
	\end{figure}
\end{example}

Recall that given two posets $\PP=(P,\leq_{P})$ and $\QQ=(Q,\leq_{Q})$, a map $\varphi:P\to Q$ is called \alert{order-preserving} if $p\leq_{P}p'$ implies $\varphi(p)\leq_{Q}\varphi(p')$ for all $p,p'\in P$.  By abuse of notation, we denote the partial order of a poset and of its $m$-fold direct product by the same symbol.

\begin{lemma}\label{lem:delta_order_preserving_dom}
	The map $\delta$ is order-preserving from $\Bigl(D_{n}^{(m)},\leq_{\text{rot}}\Bigr)$ to $\Bigl(\delta\bigl(D_{n}^{(m)}\bigr),\leq_{\text{dom}}\Bigr)$.  More precisely, let $\pf,\pf'\in D_{n}^{(m)}$ with $\pf\leq_{rot}\pf'$.  If $\delta(\pf)=(\qf_{1},\qf_{2},\ldots,\qf_{m})$ and $\delta(\pf')=(\qf'_{1},\qf'_{2},\ldots,\qf'_{m})$, then we have $\qf_{k}\leq_{\text{dom}}\qf'_{k}$ for all $k\in\{1,2,\ldots,m\}$.
\end{lemma}
\begin{proof}
	This follows immediately from Lemma~\ref{lem:sequence_order}.
\end{proof}

\begin{remark}
  \label{rem:delta_not_order_reflecting_dom}
	The converse of Lemma~\ref{lem:delta_order_preserving_dom} is not true, \ie $\delta^{-1}$ is not an order-preserving map from $\Bigl(\delta\bigl(D_{n}^{(m)}\bigr),\leq_{\text{dom}}\Bigr)$ to $\Bigl(D_{n}^{(m)},\leq_{\text{rot}}\Bigr)$.  Consider for instance the Dyck paths $\qf$ and $\qf'$ with step sequences $\uu_{\qf}=(0,1,1)$ and $\uu_{\qf'}=(0,0,1)$.  In view of Figure~\ref{fig:decomposition_dyck23} we find that $\pf=\delta^{-1}(\qf,\qf)$ has step sequence $\uu_{\pf}=(0,2,2)$ and $\pf'=\delta^{-1}(\qf,\qf')$ has step sequence $\uu_{\pf'}=(0,1,2)$.  We have $\qf\leq_{\text{dom}}\qf$ and $\qf\leq_{\text{dom}}\qf'$, but $\pf\not\leq_{\text{rot}}\pf'$. See Figure~\ref{fig:tamari23} and Figure~\ref{fig:bruhat23} for an illustration of the posets $\Bigl(\mdyck{3}{2},\leq_{\text{rot}}\Bigr)$ and $\Bigl(\delta\bigl(\mdyck{3}{2}\bigr),\leq_{\text{dom}}\Bigr)$.
	
	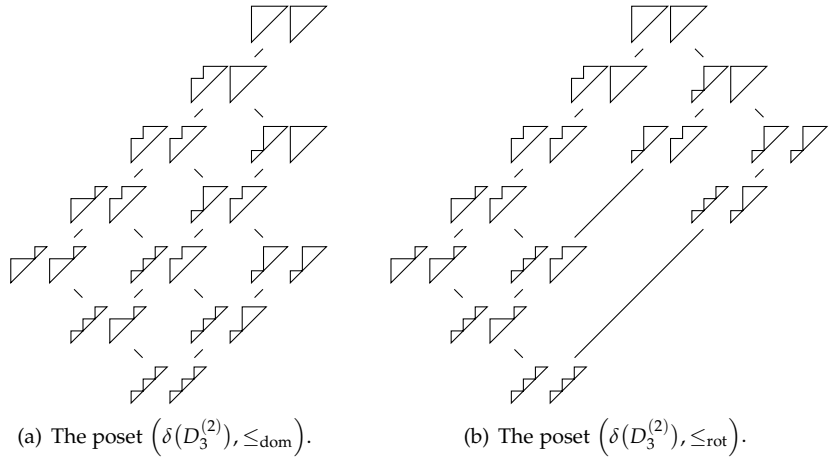
\begin{figure}
		\centering
		\subfigure[The poset $\Bigl(\delta\bigl(D_{3}^{(2)}\bigr),\leq_{\text{dom}}\Bigr)$.]{\label{fig:bruhat23}
		  \begin{tikzpicture}\tiny
			\def\x{.8};
			\def\y{.8};
			\draw(3*\x,1*\y) node(v1){\dyckThree{0/1/2}{white}{.4}{-1}\hspace*{-.1cm}\dyckThree{0/1/2}{white}{.4}{-1}};
			\draw(2*\x,2*\y) node(v2){\dyckThree{0/1/2}{white}{.4}{-1}\hspace*{-.1cm}\dyckThree{0/0/2}{white}{.4}{-1}};
			\draw(4*\x,2*\y) node(v9){\dyckThree{0/1/2}{white}{.4}{-1}\hspace*{-.1cm}\dyckThree{0/1/1}{white}{.4}{-1}};
			\draw(1*\x,3*\y) node(v3){\dyckThree{0/0/2}{white}{.4}{-1}\hspace*{-.1cm}\dyckThree{0/0/2}{white}{.4}{-1}};
			\draw(3*\x,3*\y) node(v4){\dyckThree{0/1/2}{white}{.4}{-1}\hspace*{-.1cm}\dyckThree{0/0/1}{white}{.4}{-1}};
			\draw(5*\x,3*\y) node(v10){\dyckThree{0/1/1}{white}{.4}{-1}\hspace*{-.1cm}\dyckThree{0/1/1}{white}{.4}{-1}};
			\draw(2*\x,4*\y) node(v5){\dyckThree{0/0/2}{white}{.4}{-1}\hspace*{-.1cm}\dyckThree{0/0/1}{white}{.4}{-1}};
			\draw(4*\x,4*\y) node(v11){\dyckThree{0/1/1}{white}{.4}{-1}\hspace*{-.1cm}\dyckThree{0/0/1}{white}{.4}{-1}};
			\draw(3*\x,5*\y) node(v6){\dyckThree{0/0/1}{white}{.4}{-1}\hspace*{-.1cm}\dyckThree{0/0/1}{white}{.4}{-1}};
			\draw(5*\x,5*\y) node(v12){\dyckThree{0/1/1}{white}{.4}{-1}\hspace*{-.1cm}\dyckThree{0/0/0}{white}{.4}{-1}};
			\draw(4*\x,6*\y) node(v7){\dyckThree{0/0/1}{white}{.4}{-1}\hspace*{-.1cm}\dyckThree{0/0/0}{white}{.4}{-1}};
			\draw(5*\x,7*\y) node(v8){\dyckThree{0/0/0}{white}{.4}{-1}\hspace*{-.1cm}\dyckThree{0/0/0}{white}{.4}{-1}};
			\draw(v1) -- (v2);
			\draw(v1) -- (v9);
			\draw(v2) -- (v3);
			\draw(v2) -- (v4);
			\draw(v3) -- (v5);
			\draw(v4) -- (v5);
			\draw(v4) -- (v11);
			\draw(v5) -- (v6);
			\draw(v6) -- (v7);
			\draw(v7) -- (v8);
			\draw(v9) -- (v4);
			\draw(v9) -- (v10);
			\draw(v10) -- (v11);
			\draw(v11) -- (v6);
			\draw(v11) -- (v12);
			\draw(v12) -- (v7);
		  \end{tikzpicture}}\hspace*{.5cm}
		\subfigure[The poset $\Bigl(\delta\bigl(D_{3}^{(2)}\bigr),\leq_{\text{rot}}\Bigr)$.]{\label{fig:weak_23}
		  \begin{tikzpicture}\tiny
			\def\x{.8};
			\def\y{.8};
			\draw(3*\x,1*\y) node(v1){\dyckThree{0/1/2}{white}{.4}{-1}\hspace*{-.1cm}\dyckThree{0/1/2}{white}{.4}{-1}};
			\draw(2*\x,2*\y) node(v2){\dyckThree{0/1/2}{white}{.4}{-1}\hspace*{-.1cm}\dyckThree{0/0/2}{white}{.4}{-1}};
			\draw(1*\x,3*\y) node(v3){\dyckThree{0/0/2}{white}{.4}{-1}\hspace*{-.1cm}\dyckThree{0/0/2}{white}{.4}{-1}};
			\draw(3*\x,3*\y) node(v4){\dyckThree{0/1/2}{white}{.4}{-1}\hspace*{-.1cm}\dyckThree{0/0/1}{white}{.4}{-1}};
			\draw(2*\x,4*\y) node(v5){\dyckThree{0/0/2}{white}{.4}{-1}\hspace*{-.1cm}\dyckThree{0/0/1}{white}{.4}{-1}};
			\draw(6*\x,4*\y) node(v6){\dyckThree{0/1/2}{white}{.4}{-1}\hspace*{-.1cm}\dyckThree{0/1/1}{white}{.4}{-1}};
			\draw(3*\x,5*\y) node(v7){\dyckThree{0/0/1}{white}{.4}{-1}\hspace*{-.1cm}\dyckThree{0/0/1}{white}{.4}{-1}};
			\draw(5*\x,5*\y) node(v8){\dyckThree{0/1/1}{white}{.4}{-1}\hspace*{-.1cm}\dyckThree{0/0/1}{white}{.4}{-1}};
			\draw(7*\x,5*\y) node(v9){\dyckThree{0/1/1}{white}{.4}{-1}\hspace*{-.1cm}\dyckThree{0/1/1}{white}{.4}{-1}};
			\draw(4*\x,6*\y) node(v10){\dyckThree{0/0/1}{white}{.4}{-1}\hspace*{-.1cm}\dyckThree{0/0/0}{white}{.4}{-1}};
			\draw(6*\x,6*\y) node(v11){\dyckThree{0/1/1}{white}{.4}{-1}\hspace*{-.1cm}\dyckThree{0/0/0}{white}{.4}{-1}};
			\draw(5*\x,7*\y) node(v12){\dyckThree{0/0/0}{white}{.4}{-1}\hspace*{-.1cm}\dyckThree{0/0/0}{white}{.4}{-1}};
			\draw(v1) -- (v2);
			\draw(v1) -- (v6);
			\draw(v2) -- (v3);
			\draw(v2) -- (v4);
			\draw(v3) -- (v5);
			\draw(v4) -- (v5);
			\draw(v4) -- (v8);
			\draw(v5) -- (v7);
			\draw(v6) -- (v9);
			\draw(v7) -- (v10);
			\draw(v8) -- (v11);
			\draw(v9) -- (v11);
			\draw(v10) -- (v12);
			\draw(v11) -- (v12);
		  \end{tikzpicture}}
		\caption{Two partial orders on $\delta\bigl(D_{3}^{(2)}\bigr)$.}
		\label{fig:poset_23}
	\end{figure}
\end{remark}

\begin{remark}\label{rem:delta_not_order_preserving_rot}
	We can also consider $\delta$ as a map from $\Bigl(D_{n}^{(m)},\leq_{\text{rot}}\Bigr)$ to $\Bigl(\delta\bigl(D_{n}^{(m)}),\leq_{\text{rot}}\Bigr)$.  However, in this case $\delta$ is not order-preserving.  Consider for instance $\pf,\pf'\in D_{3}^{(2)}$ with $\uu_{\pf}=(0,1,2)$ and $\uu_{\pf'}=(0,0,1)$.  Then, we have $\pf\leq_{\text{rot}}\pf'$, and $\delta(\pf)=(\qf_{1},\qf_{2}),\ \delta(\pf')=(\qf'_{1},\qf'_{2})$ with $\uu_{\qf_{1}}=(0,1,1)$ and $\uu_{\qf_{2}}=(0,0,1)$, as well as $\uu_{\qf'_{1}}=(0,0,1)$ and $\uu_{\qf'_{2}}=(0,0,0)$.  We see immediately that $\qf_{1}\not\leq_{\text{rot}}\qf'_{1}$. See Figure~\ref{fig:tamari23} and Figure~\ref{fig:weak_23} for an illustration of the posets $\Bigl(\mdyck{3}{2},\leq_{\text{rot}}\Bigr)$ and $\Bigl(\delta\bigl(\mdyck{3}{2}\bigr),\leq_{\text{rot}}\Bigr)$.
\end{remark}

The following proposition characterizes $\delta\bigl(D_{n}^{(m)}\bigr)$. 

\begin{proposition}\label{prop:m_classification}
	Let $(\qf_{1},\qf_{2},\ldots,\qf_{m})$ be an increasing $m$-fan of Dyck paths of length $2n$, with associated height sequences $\hh_{\qf_{j}}=(h_{1}^{(j)},h_{2}^{(j)},\ldots,h_{n}^{(j)})$ for all $j\in\{1,2,\ldots,m\}$.  Then, $(\qf_{1},\qf_{2},\ldots,\qf_{m})$ induces an $m$-Dyck path $\pf\in D_{n}^{(m)}$ via $\delta^{-1}$ if and only if $h_{i}^{(k)}\leq h_{i+1}^{(j)}$ for all $i\in\{1,2,\ldots,n-2\}$ and for all $k>j$.
\end{proposition}
\begin{proof}
	First suppose that $m=2$.  In this case, the map $\delta^{-1}$ constructs a sequence 
	\begin{displaymath}
		\hh=\bigl(h_{1}^{(1)},h_{1}^{(2)},h_{2}^{(1)},h_{2}^{(2)},h_{3}^{(1)},\ldots,h_{n}^{(1)},h_{n}^{(2)}\bigr).
	\end{displaymath}
	Since $(\qf_{1},\qf_{2})$ is increasing we obtain $h_{i}^{(1)}\leq h_{i}^{(2)}$ for all $i\in\{1,2,\ldots,n\}$.  By assumption we have $h_{i}^{(2)}\leq h_{i+1}^{(1)}$ for all $i\in\{1,2,\ldots,n-2\}$.  Moreover, we have $h_{n}^{(1)}=h_{n}^{(2)}$, which implies $h_{1}^{(1)}\leq h_{1}^{(2)}\leq h_{2}^{(1)}\leq\cdots\leq h_{n}^{(2)}$.  Thus $\hh$ satisfies \eqref{eq:height_1}.  By construction we have $h_{2i-1}=h_{i}^{(1)}$ and $h_{2i}=h_{i}^{(2)}$ for $i\in\{1,2,\ldots,n\}$.  Then, \eqref{eq:height_2} applied to $\hh_{\qf_{1}}$ and $\hh_{\qf_{2}}$ yields $h_{2i-1}=h_{i}^{(1)}\geq i=\lceil\tfrac{2i-1}{2}\rceil$ and $h_{2}=h_{i}^{(2)}\geq i=\lceil\tfrac{2i}{2}\rceil$, and thus $\hh$ satisfies \eqref{eq:height_2}.  Hence it is the height sequence of some $\pf\in D_{n}^{(2)}$. 
	
 	Conversely let $\pf\in D_{n}^{(2)}$ have height sequence $\hh_{\pf}=(h_{1},h_{2},\ldots,h_{2n}\}$ and let $\delta(\pf)=(\qf_{1},\qf_{2})$.  First of all, Lemma~\ref{lem:decomposition_increasing} implies that $(\qf_{1},\qf_{2})$ is an increasing fan of Dyck paths.  By construction the height sequences of $\qf_{1}$ and $\qf_{2}$ are $\hh_{\qf_{1}}=(h_{1},h_{3},\ldots,h_{2n-1})$ and $\hh_{\qf_{2}}=(h_{2},h_{4},\ldots,h_{2n})$, respectively.  In view of \eqref{eq:height_1} we obtain $h_{2i}\leq h_{2i+1}$ for all $i\in\{1,2,\ldots,n-2\}$, and we are done.

	\smallskip

	The reasoning for $m>2$ is exactly analogous.
\end{proof}

We observe that the connection between $\tmn{n}{m}$ and $\mtam{n}{m}$ described in Theorem~\ref{thm:mtamari} is rather implicit, since for large $m$ and $n$, the elements of $\tmn{n}{m}$ only make a small fraction of the elements in $\DM\Bigl(\tmn{n}{m}\Bigr)$.  We conclude this article with an explicit, but conjectural description of these elements.  Again the strip decomposition of $m$-Dyck paths plays an important role.

Let $\pf\in D_{n}^{(m)}$ and let $\delta(\pf)=(\qf_{1},\qf_{2},\ldots,\qf_{m})$.  For every $i,j\in\{1,2,\ldots,m\}$ with $i\neq j$ we define a map 
\begin{multline*}
	\beta_{i,j}:\bigl(D_{n}\bigr)^{m}\to\bigl(D_{n}\bigr)^{m},\quad (\qf_{1},\qf_{2},\ldots,\qf_{m})\mapsto\\(\qf_{1},\ldots,\qf_{i-1},\qf_{i}\wedge\qf_{j},\qf_{i+1},\ldots,\qf_{j-1},\qf_{i}\vee\qf_{j},\qf_{j+1},\ldots,\qf_{m}).
\end{multline*}
Now consider the composition 
\begin{equation}\label{eq:m_bouncing}
	\beta=\beta_{m-1,m}\circ\cdots\circ\beta_{2,3}\circ\beta_{1,m}\circ\cdots\circ\beta_{1,3}\circ\beta_{1,2},
\end{equation}
acting from the left, which we will refer to as the \alert{bouncing map}.  In particular, $\beta(\qf_{1},\qf_{2},\ldots,\qf_{m})$ is a multichain in $\mtam{n}{1}$.  If we abbreviate $\zeta=\beta\circ\delta$, then we obtain a map
\begin{equation}
	\zeta:D_{n}^{(m)}\to\bigl(D_{n}\bigr)^m, \quad p\mapsto\zeta(p).
\end{equation}

\begin{example}\label{ex:dyck23_4}
	\begin{figure}
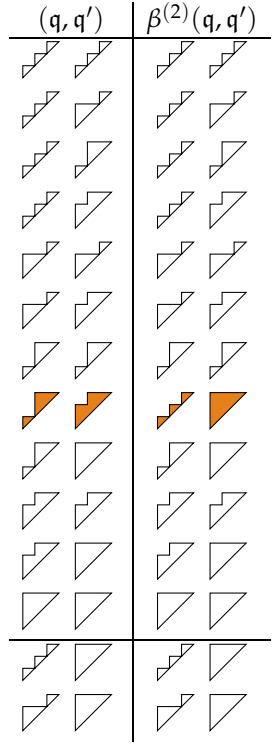

		\centering
		\begin{tabular}{c|c}
			$(\qf,\qf')$ & $\beta^{(2)}(\qf,\qf')$\\
			\hline
			\dyckThree{0/1/2}{white}{.4}{-1}\dyckThree{0/1/2}{white}{.4}{-1} & \dyckThree{0/1/2}{white}{.4}{-1}\dyckThree{0/1/2}{white}{.4}{-1}\\
			\dyckThree{0/1/2}{white}{.4}{-1}\dyckThree{0/0/2}{white}{.4}{-1} & \dyckThree{0/1/2}{white}{.4}{-1}\dyckThree{0/0/2}{white}{.4}{-1}\\
			\dyckThree{0/1/2}{white}{.4}{-1}\dyckThree{0/1/1}{white}{.4}{-1} & \dyckThree{0/1/2}{white}{.4}{-1}\dyckThree{0/1/1}{white}{.4}{-1}\\
			\dyckThree{0/1/2}{white}{.4}{-1}\dyckThree{0/0/1}{white}{.4}{-1} & \dyckThree{0/1/2}{white}{.4}{-1}\dyckThree{0/0/1}{white}{.4}{-1}\\
			\dyckThree{0/0/2}{white}{.4}{-1}\dyckThree{0/0/2}{white}{.4}{-1} & \dyckThree{0/0/2}{white}{.4}{-1}\dyckThree{0/0/2}{white}{.4}{-1}\\
			\dyckThree{0/0/2}{white}{.4}{-1}\dyckThree{0/0/1}{white}{.4}{-1} & \dyckThree{0/0/2}{white}{.4}{-1}\dyckThree{0/0/1}{white}{.4}{-1}\\
			\dyckThree{0/1/1}{white}{.4}{-1}\dyckThree{0/1/1}{white}{.4}{-1} & \dyckThree{0/1/1}{white}{.4}{-1}\dyckThree{0/1/1}{white}{.4}{-1}\\
			\dyckThree{0/1/1}{\cOne}{.4}{-1}\dyckThree{0/0/1}{\cOne}{.4}{-1} & \dyckThree{0/1/2}{\cOne}{.4}{-1}\dyckThree{0/0/0}{\cOne}{.4}{-1}\\
			\dyckThree{0/1/1}{white}{.4}{-1}\dyckThree{0/0/0}{white}{.4}{-1} & \dyckThree{0/1/1}{white}{.4}{-1}\dyckThree{0/0/0}{white}{.4}{-1}\\
			\dyckThree{0/0/1}{white}{.4}{-1}\dyckThree{0/0/1}{white}{.4}{-1} & \dyckThree{0/0/1}{white}{.4}{-1}\dyckThree{0/0/1}{white}{.4}{-1}\\
			\dyckThree{0/0/1}{white}{.4}{-1}\dyckThree{0/0/0}{white}{.4}{-1} & \dyckThree{0/0/1}{white}{.4}{-1}\dyckThree{0/0/0}{white}{.4}{-1}\\
			\dyckThree{0/0/0}{white}{.4}{-1}\dyckThree{0/0/0}{white}{.4}{-1} & \dyckThree{0/0/0}{white}{.4}{-1}\dyckThree{0/0/0}{white}{.4}{-1}\\
			\hline
			\dyckThree{0/1/2}{white}{.4}{-1}\dyckThree{0/0/0}{white}{.4}{-1} & \dyckThree{0/1/2}{white}{.4}{-1}\dyckThree{0/0/0}{white}{.4}{-1}\\
			\dyckThree{0/0/2}{white}{.4}{-1}\dyckThree{0/0/0}{white}{.4}{-1} & \dyckThree{0/0/2}{white}{.4}{-1}\dyckThree{0/0/0}{white}{.4}{-1}\\
		\end{tabular}
		\caption{The bouncing map in action.  The highlighted pair is the only pair on which $\beta$ acts non-trivially.}
		\label{fig:bouncing_dyck23}
	\end{figure}

	The first column in Figure~\ref{fig:bouncing_dyck23} shows the fourteen pairs of Dyck paths in $D_{3}$ that we found in Example~\ref{ex:dyck23_2}.  The two pairs which do not satisfy the conditions of Proposition~\ref{prop:m_classification}, are placed at the end, separated by a horizontal line.  The second column shows the corresponding pairs of Dyck paths after the application of the bouncing map $\beta$.  If we order these pairs by componentwise rotation order, then we obtain the lattice shown in Figure~\ref{fig:dyck23_decomp}, and we notice that this lattice is isomorphic to $\mtam{3}{2}$ shown in Figure~\ref{fig:tamari23}. 

	\begin{figure}
		\centering	
		\begin{tikzpicture}\tiny
			\def\x{.8};
			\def\y{.8};
			\draw(3*\x,1*\y) node(v1){\dyckThree{0/1/2}{white}{.4}{-1}\hspace*{-.1cm}\dyckThree{0/1/2}{white}{.4}{-1}};
			\draw(2*\x,2*\y) node(v2){\dyckThree{0/1/2}{white}{.4}{-1}\hspace*{-.1cm}\dyckThree{0/0/2}{white}{.4}{-1}};
			\draw(1*\x,3*\y) node(v3){\dyckThree{0/0/2}{white}{.4}{-1}\hspace*{-.1cm}\dyckThree{0/0/2}{white}{.4}{-1}};
			\draw(3*\x,3*\y) node(v4){\dyckThree{0/1/2}{white}{.4}{-1}\hspace*{-.1cm}\dyckThree{0/0/1}{white}{.4}{-1}};
			\draw(2*\x,4*\y) node(v5){\dyckThree{0/0/2}{white}{.4}{-1}\hspace*{-.1cm}\dyckThree{0/0/1}{white}{.4}{-1}};
			\draw(6*\x,4*\y) node(v6){\dyckThree{0/1/2}{white}{.4}{-1}\hspace*{-.1cm}\dyckThree{0/1/1}{white}{.4}{-1}};
			\draw(3*\x,5*\y) node(v7){\dyckThree{0/0/1}{white}{.4}{-1}\hspace*{-.1cm}\dyckThree{0/0/1}{white}{.4}{-1}};
			\draw(5*\x,5*\y) node(v8){\dyckThree{0/1/2}{white}{.4}{-1}\hspace*{-.1cm}\dyckThree{0/0/0}{white}{.4}{-1}};
			\draw(7*\x,5*\y) node(v9){\dyckThree{0/1/1}{white}{.4}{-1}\hspace*{-.1cm}\dyckThree{0/1/1}{white}{.4}{-1}};
			\draw(4*\x,6*\y) node(v10){\dyckThree{0/0/1}{white}{.4}{-1}\hspace*{-.1cm}\dyckThree{0/0/0}{white}{.4}{-1}};
			\draw(6*\x,6*\y) node(v11){\dyckThree{0/1/1}{white}{.4}{-1}\hspace*{-.1cm}\dyckThree{0/0/0}{white}{.4}{-1}};
			\draw(5*\x,7*\y) node(v12){\dyckThree{0/0/0}{white}{.4}{-1}\hspace*{-.1cm}\dyckThree{0/0/0}{white}{.4}{-1}};
			\draw(v1) -- (v2);
			\draw(v1) -- (v6);
			\draw(v2) -- (v3);
			\draw(v2) -- (v4);
			\draw(v3) -- (v5);
			\draw(v4) -- (v5);
			\draw(v4) -- (v8);
			\draw(v5) -- (v7);
			\draw(v6) -- (v8);
			\draw(v6) -- (v9);
			\draw(v7) -- (v10);
			\draw(v8) -- (v10);
			\draw(v8) -- (v11);
			\draw(v9) -- (v11);
			\draw(v10) -- (v12);
			\draw(v11) -- (v12);
		\end{tikzpicture}
		\caption{Componentwise rotation order on the bounced strip decomposition of the set of $2$-Dyck paths of length $9$ yields the $2$-Tamari lattice of parameter $3$.}
		\label{fig:dyck23_decomp}
	\end{figure}
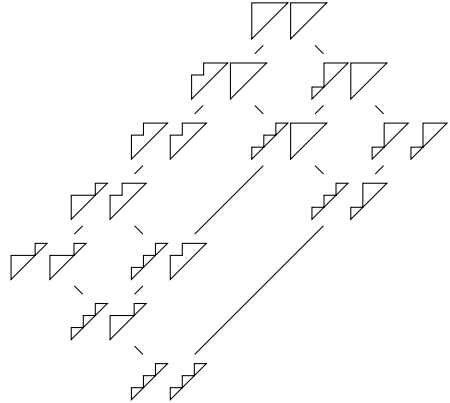
\end{example}

Computer experiments suggest the following property of the bouncing map.

\begin{conjecture}\label{conj:zeta_order_isomorphism}
	The posets $\Bigl(D_{n}^{(m)},\leq_{\text{rot}}\Bigr)$ and $\Bigl(\zeta\bigl(D_{n}^{(m)}\bigr),\leq_{\text{rot}}\Bigr)$ are isomorphic.
\end{conjecture}

This conjecture was verified for $n\leq 5$ and $m\leq 7$, as well as for $n=6$ and $m\leq 4$ with \text{Sage-Combinat}~\cites{sage,sagecombinat}.  The corresponding script can be obtained from \url{http://homepage.univie.ac.at/henri.muehle/files/m_tamari_decomposition.sage}.

\begin{remark}
  \label{rem:beta_not_order_preserving_dom}
	In general, $\beta$ is not an order-preserving map from $\bigl(\delta(D_{n}^{(m)}),\leq_{\text{dom}}\bigr)$ to $\bigl(\zeta(D_{n}^{(m)}),\leq_{\text{rot}}\bigr)$.
	
	Consider for instance the Dyck paths $\qf_{1},\qf_{2},\qf'_{1},\qf'_{2}\in D_{5}$ given by the step sequences $\hh_{\qf_{1}}=(1,3,3,4,5),\hh_{\qf_{2}}=(2,3,4,4,5)$ as well as $\hh_{\qf'_{1}}=(2,3,3,5,5),\hh_{\qf'_{2}}=(2,3,4,5,5)$.  Then we have $\qf_{1}\leq_{\text{dom}}\qf'_{1}$ and $\qf_{2}\leq_{\text{dom}}\qf'_{2}$.  Moreover, the pairs $(\qf_{1},\qf_{2})$ and $(\qf'_{1},\qf'_{2})$ satisfy the conditions from Proposition~\ref{prop:m_classification}, so they are indeed contained in $\delta\bigl(D_{5}^{(2)}\bigr)$.  We obtain $\hh_{\qf_{1}\wedge\qf_{2}}=(1,2,3,4,5)$ and $\hh_{\qf_{1}\vee\qf_{2}}=(3,3,4,4,5)$ as well as $\hh_{\qf'_{1}\wedge\qf'_{2}}=(2,3,3,4,5)$ and $\hh_{\qf'_{1}\vee\qf'_{2}}=(2,3,5,5,5)$.  The corresponding step sequences are $\uu_{\qf_{1}\wedge\qf_{2}}=(0,1,2,3,4)$ and $\uu_{\qf_{1}\vee\qf_{2}}=(0,0,0,2,4)$, as well as $\uu_{\qf'_{1}\wedge\qf'_{2}}=(0,0,1,3,4)$ and $\uu_{\qf'_{1}\vee\qf'_{2}}=(0,0,1,2,2)$.  This implies that $\beta(\qf_{1},\qf_{2})\not\leq\beta(\qf'_{1},\qf'_{2})$.
		
	However, if $\pf,\pf'\in D_{5}^{(2)}$ are the $2$-Dyck paths satisfying $\delta(\pf)=(\qf_{1},\qf_{2})$ and $\delta(\pf')=(\qf'_{1},\qf'_{2})$, then we can quickly check that $\pf$ and $\pf'$ are determined by the step sequences $\uu_{\pf}=(0,1,2,5,8)$ and $\uu_{\pf'}=(0,0,2,5,6)$, which implies $\pf\not\leq_{\text{rot}}\pf'$. So this is \emph{not} a counterexample to Conjecture~\ref{conj:zeta_order_isomorphism}.
\end{remark}

\section*{Acknowledgments}
	We are very grateful to the anonymous referee for the many valuable comments and suggestions which helped us improve the presentation and content of this article. 
	
\bibliography{literature}

\begin{bibdiv}
\begin{biblist}

\bib{banaschewski56huellensysteme}{article}{
      author={Banaschewski, Bernhard},
       title={{H{\"u}llensysteme und Erweiterungen von Quasi-Ordnungen}},
        date={1956},
     journal={Zeitschrift f{\"u}r mathematische Logik und Grundlagen der
  Mathematik},
      volume={2},
       pages={117\ndash 130},
}

\bib{bergeron13multivariate}{article}{
      author={Bergeron, Fran{\c c}ois},
       title={{Multivariate Diagonal Coinvariant Spaces for Complex Reflection
  Groups}},
        date={2013},
     journal={Advances in Mathematics},
      volume={239},
       pages={97\ndash 108},
}

\bib{bergeron12higher}{article}{
      author={Bergeron, Fran{\c c}ois},
      author={Pr{\'e}ville-Ratelle, Louis-Fran{\c c}ois},
       title={{Higher Trivariate Diagonal Harmonics via Generalized Tamari
  Posets}},
        date={2012},
     journal={Journal of Combinatorics},
      volume={3},
       pages={317\ndash 341},
}

\bib{bjorner97shellable}{article}{
      author={Bj{\"o}rner, Anders},
      author={Wachs, Michelle~L.},
       title={{Shellable and Nonpure Complexes and Posets II}},
        date={1997},
     journal={Transactions of the American Mathematical Society},
      volume={349},
       pages={3945\ndash 3975},
}

\bib{bousquet11number}{article}{
      author={Bousquet-M{\'e}lou, Mireille},
      author={Fusy, {\'E}ric},
      author={Pr{\'e}ville-Ratelle, Louis-Fran{\c c}ois},
       title={{The Number of Intervals in the $m$-Tamari Lattices}},
        date={2011},
     journal={The Electronic Journal of Combinatorics},
      volume={18},
}

\bib{davey02introduction}{book}{
      author={Davey, Brian~A.},
      author={Priestley, Hilary~A.},
       title={{Introduction to Lattices and Order}},
   publisher={Cambridge University Press},
     address={Cambridge},
        date={2002},
}

\bib{duchon00enumeration}{article}{
      author={Duchon, Philippe},
       title={{On the Enumeration and Generation of Generalized Dyck Words}},
        date={2000},
     journal={Discrete Mathematics},
      volume={225},
       pages={121\ndash 135},
}

\bib{fishburn70intransitive}{article}{
      author={Fishburn, Peter~C.},
       title={{Intransitive Indifference with Unequal Indifference Intervals}},
        date={1970},
     journal={Journal of Mathematical Psychology},
      volume={7},
       pages={144\ndash 149},
}

\bib{fomin05generalized}{article}{
      author={Fomin, Sergey},
      author={Reading, Nathan},
       title={{Generalized Cluster Complexes and Coxeter Combinatorics}},
        date={2005},
     journal={International Mathematics Research Notices},
      volume={44},
       pages={2709\ndash 2757},
}

\bib{freese95free}{book}{
      author={Freese, Ralph},
      author={Je{\v{z}}ek, Jaroslav},
      author={Nation, James~B.},
       title={{Free Lattices}},
   publisher={American Mathematical Society},
     address={Providence},
        date={1995},
}

\bib{geyer94on}{article}{
      author={Geyer, Winfried},
       title={{On Tamari Lattices}},
        date={1994},
     journal={Discrete Mathematics},
      volume={133},
       pages={99\ndash 122},
}

\bib{hohlweg11permutahedra}{article}{
      author={Hohlweg, Christophe},
      author={Lange, Carsten},
      author={Thomas, Hugh},
       title={{Permutahedra and Generalized Associahedra}},
        date={2011},
     journal={Advances in Mathematics},
      volume={226},
       pages={608\ndash 640},
}

\bib{jonsson05generalized}{article}{
      author={Jonsson, Jakob},
       title={{Generalized Triangulations and Diagonal-Free Subsets of Stack
  Polyominoes}},
        date={2005},
     journal={Journal of Combinatorial Theory (Series A)},
      volume={112},
       pages={117\ndash 142},
}

\bib{kim14dyck}{article}{
      author={Kim, Jang~Soo},
      author={M{\'e}sz{\'a}ros, Karola},
      author={Panova, Greta},
      author={Wilson, David~B.},
       title={{Dyck Tilings, Increasing Trees, Descents and Inversions}},
        date={2014},
     journal={Journal of Combinatorial Theory (Series A)},
      volume={122},
       pages={9\ndash 27},
}

\bib{krattenthaler06growth}{article}{
      author={Krattenthaler, Christian},
       title={{Growth Diagrams, and Increasing and Decreasing Chains in
  Fillings of Ferrers Shapes}},
        date={2006},
     journal={Advances in Applied Mathematics},
      volume={37},
       pages={404\ndash 431},
}

\bib{liu00left}{article}{
      author={Liu, Shu-Chung},
      author={Sagan, Bruce~E.},
       title={{Left-Modular Elements of Lattices}},
        date={2000},
     journal={Journal of Combinatorial Theory (Series A)},
      volume={91},
       pages={369\ndash 385},
}

\bib{markowsky92primes}{article}{
      author={Markowsky, George},
       title={{Primes, Irreducibles and Extremal Lattices}},
        date={1992},
     journal={Order},
      volume={9},
       pages={265\ndash 290},
}

\bib{mcnamara06poset}{article}{
      author={McNamara, Peter},
      author={Thomas, Hugh},
       title={{Poset Edge-Labellings and Left Modularity}},
        date={2006},
     journal={European Journal of Combinatorics},
      volume={27},
       pages={101\ndash 113},
}

\bib{hoissen12associahedra}{book}{
      author={M{\"u}ller-Hoissen, Folkert},
      author={Pallo, Jean~Marcel},
      author={(Eds.), Jim~Stasheff},
       title={{Associahedra, Tamari Lattices and Related Structures}},
   publisher={Birkh{\"a}user},
     address={Basel},
        date={2012},
}

\bib{pons13combinatoire}{thesis}{
      author={Pons, Viviane},
       title={{Combinatoire Alg{\'e}brique Li{\'e}e aux Ordres sur les
  Permutations}},
        type={Th{\`e}se de Doctorat},
 institution={Universit{\'e} Paris-Est},
        date={2013},
}

\bib{reading06cambrian}{article}{
      author={Reading, Nathan},
       title={{Cambrian Lattices}},
        date={2006},
     journal={Advances in Mathematics},
      volume={205},
       pages={313\ndash 353},
}

\bib{reading07clusters}{article}{
      author={Reading, Nathan},
       title={{Clusters, Coxeter-Sortable Elements and Noncrossing
  Partitions}},
        date={2007},
     journal={Transactions of the American Mathematical Society},
      volume={359},
       pages={5931\ndash 5958},
}

\bib{stanley97enumerative}{book}{
      author={Stanley, Richard~P.},
       title={{Enumerative Combinatorics, Vol. 1}},
   publisher={Cambridge University Press},
     address={Cambridge},
        date={1997},
}

\bib{stasheff63homotopy1}{article}{
      author={Stasheff, James~D.},
       title={{Homotopy Associativity of $H$-Spaces. I}},
        date={1963},
     journal={Transactions of the American Mathematical Society},
      volume={108},
       pages={275\ndash 292},
}

\bib{sage}{misc}{
      author={Stein, William~A.},
      author={others},
       title={{Sage Mathematics Software (Version 6.5)}},
organization={{The Sage Development Team}},
        date={2015},
        note={\url{http://www.sagemath.org}},
}

\bib{stump15cataland}{article}{
      author={Stump, Christian},
      author={Thomas, Hugh},
      author={Williams, Nathan},
       title={{Cataland: Why the Fuss?}},
        date={2015},
      eprint={arXiv:1503.00710},
}

\bib{tamari62algebra}{article}{
      author={Tamari, Dov},
       title={{The Algebra of Bracketings and their Enumeration}},
        date={1962},
     journal={Nieuw Archief voor Wiskunde},
      volume={10},
       pages={131\ndash 146},
}

\bib{sagecombinat}{misc}{
      author={{The Sage-Combinat Community}},
       title={{\texttt{sage-combinat}: Enhancing Sage as a Toolbox for Computer
  Exploration in Algebraic Combinatorics}},
        date={2015},
        note={\url{http://combinat.sagemath.org}},
}

\bib{thomas06analogue}{article}{
      author={Thomas, Hugh},
       title={{An Analogue of Distributivity for Ungraded Lattices}},
        date={2006},
     journal={Order},
      volume={23},
       pages={249\ndash 269},
}

\bib{urquhart78topological}{article}{
      author={Urquhart, Alasdair},
       title={{A Topological Representation Theory for Lattices}},
        date={1978},
     journal={Algebra Universalis},
      volume={8},
       pages={45\ndash 58},
}

\end{biblist}
\end{bibdiv}
  \label{sec:bibliography}

\end{document}